\documentclass{siamart190516}

\usepackage{amstext,amssymb,amsmath,bm}
\usepackage{eepic,epic}
\usepackage{epsfig,subfigure,epstopdf}
\usepackage{pdfsync}
\usepackage{color,graphicx}
\usepackage{hyperref}
\usepackage{tcolorbox}
\usepackage{subfigure}
\usepackage{wrapfig}
\usepackage{fullpage}
\usepackage{multirow}
\usepackage{tabulary}
\usepackage{booktabs}
\usepackage{enumerate}
\usepackage{ulem}
\usepackage{algorithm, algorithmic}
\renewcommand{\algorithmiccomment}[1]{\bgroup\hfill//~#1\egroup}
\newcommand{\nn}{\nonumber}

\newsiamthm{remark}{Remark}
\newsiamthm{example}{Example}
\newsiamthm{assumption}{Assumption}
\newsiamthm{property}{Property}

\newcommand{\DIV}{\mbox{div}}

\newcommand{\dx}{\mathrm{dx}}

\newcommand{\ds}{\mathrm{ds}}
\newcommand{\I}{\mathcal{I}}
\DeclareMathOperator*{\argmin}{arg\,min}

\newcommand{\ve}{\varepsilon}
\newcommand{\vphi}{\varphi}
\newcommand{\Lphi}{L^\varphi}
\newcommand{\Wphi}{W^{1, \varphi}}

\newcommand{\T}{\mathcal{T}}
\newcommand{\Th}{\T_h}
\newcommand{\Tc}{\T_H}
\newcommand{\Vc}{\mathcal{V}_H}
\newcommand{\Vcn}{\mathcal{V}_H^{(n)}}
\newcommand{\Vh}{\mathcal{V}_h}
\newcommand{\Vf}{\mathcal{V}_f}
\newcommand{\Vi}{\mathcal{V}_i}
\newcommand{\V}[2]{\mathcal{V}^{(#1)}_{#2}}

\def\Nc{N_c}

\newcommand{\R}{\mathbb{R}}
\newcommand{\bA}{A}
\newcommand{\bB}{\mathbf{B}}
\newcommand{\bV}{\mathbf{V}}
\newcommand{\bxi}{\bm{\xi}}
\newcommand{\bzeta}{\bm{\zeta}}
\newcommand{\bPsi}{\bm{\Psi}}
\newcommand{\J}{\mathcal{J}}

\newcommand{\sd}{\mathrm{G}}
\newcommand{\nt}{\mathrm{N}}

\newcommand{\un}{u^{(n)}}
\newcommand{\wn}{w^{(n)}}
\newcommand{\wnH}{w^{(n)}_H}

\newcommand{\phin}[2]{\phi^{(#1)}_{#2}}

\newcommand{\reg}{\mathrm{reg}}
\newcommand{\err}{\mathrm{err}}

\newcommand{\Ar}{A}

\def\<{\langle}
\def\>{\rangle}

\newcommand{\leqs}{\leqslant}
\newcommand{\geqs}{\geqslant}

\newcommand{\bdist}[2]{\mathcal{D}_\J(#1, #2)}
\newcommand{\cj}{c_{\J}}
\newcommand{\Cj}{C_{\J}}
\newcommand{\Dphi}{\Delta_2({\vphi, \vphi^*})}

\hfuzz=\maxdimen
\newlength{\fixboxwidth}
\title{Iterated numerical homogenization for multi-scale elliptic equations with monotone nonlinearity}

\author{
	Xinliang Liu
	\footnote{{\tt liuxinliang@sjtu.edu.cn}
		Institute of Natural Sciences, School of Mathematical Sciences,
		and MOE-LSC, Shanghai Jiao Tong University.
	},
	\\
	Eric Chung
	\footnote{{\tt tschung@math.cuhk.edu.hk}
		Department of Mathematics, The Chinese University of Hong Kong, Shatin, Hong Kong.
	},
	and Lei Zhang
	\footnote{{\tt lzhang2012@sjtu.edu.cn}
		Institute of Natural Sciences, School of Mathematical Sciences,
		and MOE-LSC, Shanghai Jiao Tong University.
	}
}

\begin{document}

 \maketitle
%%%%%%%%%%%%%%%%%%%%%%%%%%%%%%%%%%%%%%%%%%%%%%%%%%%%%
\begin{abstract}
Nonlinear multi-scale problems are ubiquitous in materials science and biology. Complicated interactions between nonlinearities and (nonseparable) multiple scales pose a major challenge for analysis and simulation. In this paper, we study the numerical homogenization for multi-scale elliptic PDEs with monotone nonlinearity, in particular the Leray-Lions problem (a prototypical example is the p-Laplacian equation), where the nonlinearity cannot be parameterized with low dimensional parameters, and the linearization error is non-negligible. We develop the iterated numerical homogenization scheme by combining numerical homogenization methods for linear equations, and the so-called "quasi-norm" based iterative approach for monotone nonlinear equation. We propose a residual regularized nonlinear iterative method, and in addition, develop the sparse updating method for the efficient update of coarse spaces. A number of numerical results are presented to complement the analysis and valid the numerical method.

\textbf{Keywords:} numerical homogenization,  multi-scale elliptic problem, monotone nonlinearity, $p$-Laplacian, regularization, sparse updating.
\end{abstract}

%%%%%%%%%%%%%%%%%%%%%%%%%%%%%%%%%%%%%%%%%%%%%%%%%%%%

%=============================================
\section{Introduction} \label{intro}
%=============================================

%General statements about multi-scale and nonlinearities and the need for multi-scale methods for nonlinear problems.

Problems with a wide range of coupled temporal and spatial scales are ubiquitous in many phenomena and processes of materials science and biology. Multi-scale modeling and simulation is essential in underpinning the discovery and synthesis of new materials and chemicals with novel functionalities in key areas such as energy, information technology and bio-medicine. In particular, for nonlinear multi-scale problems such as liver surgery, tissue growth, crack propagation, phase transformations, etc. \cite{Christ:2017,Rahman:2017,castaneda2006nonlinear,geers2017homogenization}, complicated interactions between nonlinearity and (nonseparable) multiple scales pose a major challenge for the design and analysis of efficient and robust multi-scale numerical methods. 

In this paper, we discuss the design, analysis and implementation of numerical homogenization type multi-scale methods for multi-scale elliptic PDEs with monotone nonlinearity. Let $\Omega\subset \mathbb{R}^d$, $d\geqs 2$, be a polygonal (polyhedral) domain (open, bounded and connected), we consider the following nonlinear elliptic problem: find $u: \Omega \to \mathbb{R}$, such that 
\begin{equation}
\left\{
	\begin{aligned}
		-\nabla \cdot a(x, u(x), \nabla u(x)) = f\quad \text{in }\Omega\\
		u = 0 \quad \text{on }\partial\Omega
	\end{aligned}
\right.	
\label{eqn:nonlinear}
\end{equation}  
where $a(x, u(x), \nabla u(x)): \Omega\times\mathbb{R}\times\mathbb{R}^d\to \mathbb{R}^d$ is the nonlinear flux function, and $f:\Omega\to \mathbb{R}$ is the source term. The nonlinear flux function $a$ takes the general form $a(x, v, \bxi) = \bA(x, v, \bxi)\bxi$ for $(x,v,\bxi)\in \Omega\times \R \times \R^d$, where 
$\bA:  \Omega\times \R \times \R^d \to \R^{d\times d}$ is a matrix-valued Carath\'{e}odory function (measurable in $x$, continuous in $v$ and $\bxi$).  

We will focus on the Leray-Lions type problem \cite{LerayLions:1965} such that $\bA(x,v,\bxi) = \bA(x, \bxi) $ for $(x, \bxi)\in \Omega\times \R^d$. This situation is very different from the quasilinear case where $\bA(x,v,\bxi) = \bA(x, v)$. In the quasilinear case, the nonlinearities within coarse regions, that induce the change in the heterogeneities, can be parametrized with a low dimensional parameter and linear multiscale theories can apply. For the Leray-Lions case, $\nabla u$ can be highly heterogeneous and one can not use any low dimensional approximation and linear theories. This is true even for a separable case
$\bA(x,u,\nabla u)=\kappa(x) b(\nabla u) $, which require nonlinear cell problems. 

We assume that the Leray-Lions flux takes the following heterogeneous $\vphi$-Laplacian form: $a(x, \bxi) = \kappa(x) \vphi'(|\bxi|)\bxi/|\bxi|$, for $(x, \bxi)\in \Omega\times \R^d$ \cite{Diening:2008b}. $\kappa(\bm{x})\in L^{\infty}(\Omega)$ is symmetric, uniformly elliptic on $\Omega$, and may contain (non-separable) multiple scales. If $\vphi\in C^2:\R^+\to \R^+$ is the so-called $N-$function \cite{Diening:2008b} (see also Section \ref{sec:prelim:form}), where the well-know p-Laplacian is a special case such that $\vphi(t) = t^p/p$ for $p> 1$, the nonlinear multi-scale problem \eqref{eqn:nonlinear} admits a natural variational form \eqref{eqn:energy}, and also admits a unique solution in the Orlicz-Sobolev space $V:=W_0^{1,\vphi}(\Omega)$ \cite{Diening:2008b} ($W_0^{1,p}$ for the p-Laplacian case). More general nonlinear flux will be treated in our future work. The heterogeneous $\vphi$-Laplacian model have applications in many areas such as nonlinear materials \cite{geers2017homogenization}, non-Newtonian fluids \cite{Diening:2013}, image processing \cite{Elmoataz:2015}, machine learning \cite{Slepcev:2017}, etc. 

There are various numerical homogenization approaches for linear multi-scale problems, such as homogenization \cite{papanicolau1978asymptotic,jikov2012homogenization}, numerical homogenization \cite{dur91,ab05,weh02}, heterogeneous multi-scale methods \cite{ee03,abdulle2014analysis, ming2005analysis}, multi-scale network approximations \cite{berlyand2013introduction}, multi-scale finite element methods \cite{Arbogast_two_scale_04, egw10, eh09, Review},
variational multi-scale methods \cite{hughes98, bazilevs2007variational}, flux norm homogenization \cite{berlyand2010flux,Owhadi:2011}, rough polyharmonic splines (RPS) \cite{owhadi2014polyharmonic}, generalized multi-scale finite element methods \cite{egh12, chung2014adaptiveDG,  chung2015residual}. In this paper, we will use the so-called RPS \cite{owhadi2014polyharmonic} or GRPS (generalized rough polyharmonic splines) method as a representative method.

Some of those ideas can be extended to nonlinear problems \cite{feyel1999,geers2017homogenization,ehg04,Efendiev_GKiL_12,chung2014adaptive}. These approaches approximate the solution of nonlinear PDEs on a coarse grid (see Figure \ref{fig:mesh} for an illustration of coarse and fine grids) by using subgrid models. Some common ingredients in these methods are that local solutions are pre-computed over coarse patches, and coarse solves are performed over the coarse space spanned by those bases. The extensions of these methods to nonlinear problems use nonlinear local problems. For example, for the Leray-Lions flux $a(x,\nabla u)$, one can use a local problem to compute a basis $\phi$ in each coarse cell,  such that $-\DIV(a(x,\nabla \phi_\xi))=0$, with boundary conditions $\phi=\xi\cdot x$. The homogenized fluxes are computed by averaging the flux $a^*(\xi)=\langle a(x,\nabla \phi_\xi)\rangle$. These approaches follow homogenization theory
(\cite{ep04, pankov2013g, henning2015error,henning2011heterogeneous, Chiadopiat:1990, allaire1992homogenization, lions2001reiterated, nguetseng2003homogenization}, see also \cite{abdulle2014analysis,henning2015error, ming2005analysis} and, the references therein, for numerical homogenization. Also,
Desbrun et. al. have applied ideas from linear numerical homogenization to coarse graining and model reduction of heterogeneous and nonlinear elasticity problems, as well as Navier-Stokes equations \cite{Kharevych:2009,Chen:2018,Budninskiy:2019,Chen:2019}.

For linear problems, one can construct one linear basis function per coarse node that contains the effects of small scales, in contrast, for nonlinear problems, one may need more multi-scale basis functions for each coarse cell. In \cite{egpz14,chung2017multiscale}, the authors developed a systematic enrichment method, which calculates multi-scale basis functions via local spectral decomposition in each coarse cell, combined with nonlinear harmonic extensions, in order to capture small scales. The approach guarantees the recovery of homogenization results when there is a scale separation, and it is in a spirit of hybridization techniques \cite{cockburn2009unified, efendiev2015spectral, chung2014staggered}. 

In this paper, we take an alternative view for the numerical homogenization of nonlinear elliptic problems. Our view is based on the well-established fact that linear multi-scale problem can be well approximated by the coarse space up to the coarse resolution, even with nonseparable scales. Therefore, our goal is to find a sequence of approximate linear problems and the corresponding coarse spaces, which we call \emph{iterated numerical homogenization}, for the original nonlinear problem. We first present an idealized construction using the so-called "quasi-norm" \cite{Diening:2008a,Diening:2013} and prove the exponential energy convergence of this approach, which is independent of the heterogeneity. Although the resulting numerical scheme is implicit, we can interpret the preconditioned gradient descent method and Newton's method as explicit relaxations of the quasi-norm approach, and show similar exponential energy convergence up to some reasonable assumptions. Although we do not need multiple basis functions per coarse node, we update the coarse space at each iteration. Therefore, we propose a sparse updating strategy to compute only a fraction of basis functions per iteration.

The paper is organized as follows. In Section \ref{sec:prelim}, we present the abstract setting of the nonlinear problem and its regularization, also the finite element formulation and the numerical homogenization for linear problem. The iterative method, and the iterated numerical homogenization are formulated and analyzed in Section \ref{sec:methods}. We present numerical experiments to validate the methods in Section \ref{sec:numerics}, and conclude the paper in Section \ref{sec:conclusion}.

\textbf{Notations:}
The symbol $C$ (or $c$) denotes generic positive constant that may change from one line of an estimate to the next. The dependence of $C$ will be clear from the context or stated explicitly. To further simplify notation we will often write $\lesssim$ to mean $\leqs C$ as well as $\eqsim$ to mean both $\lesssim$ and $\gtrsim$. We use the standard definitions and notations $L^p$, $W^{k,p}$, $H^k$ for Lebesgue and Sobolev spaces. 
We denote the open ball with radius $r$ about $x$ and $0$ by $B_r(x)$ and $B_r:=B_r(0)$. 

% \input{./Sep/Preliminaries}

%=============================================
\section{Preliminaries} \label{sec:prelim}
%=============================================
We briefly describe some preliminary knowledge needed in the paper. We first introduce the weak formulation and N-functions for the monotone nonlinear elliptic problem. We then discuss the regularization of the possible degeneracy/singularity, which can be used to guarantee the well-posedness of linearized problems. In the end, we formulate the finite element method for the nonlinear problem, and also the numerical homogenization for linear problems. 

%=============================================
\subsection{Weak Formulation and N-function}
\label{sec:prelim:form}
%=============================================

We define the following energy functional
\begin{equation}
	\mathcal{J}(u) : = \int_\Omega \kappa(x)\vphi(|\nabla u|)  - \int_\Omega f u,  \quad u \in V,
	\label{eqn:energy}
\end{equation}
where $\Omega$ is a bounded open set in $\mathbb{R}^d$ with piecewise Lipschitz boundary $\partial \Omega$. $\kappa_{\max} \geqs \kappa(x)\geqs \kappa_{\min} >0$, $\kappa(x)$ is a heterogeneous coefficient with oscillations and possible high contrast, i.e., $\kappa_{\max} / \kappa_{\min}$ is large. In this paper, we assume that the forcing term $f\in L^2(\Omega)$, unless otherwise specified. It is well known that the variational problem: Find $u\in V$, such that
\begin{equation}
u = \argmin _{v \in V} \J(v)
\label{eqn:variational}
\end{equation}
is equivalent to the weak form of nonlinear elliptic equation \eqref{eqn:nonlinear}: Find $u\in V$, such that
\begin{equation}
 \langle A u, v\rangle=\int_{\Omega}a(x,\nabla u) \cdot \nabla v  =\int_{\Omega} f v , \quad \forall  v \in V 
\label{eqn:weakform}
\end{equation}
where $a(x, \bxi) = \kappa(x) \vphi'(|\bxi|)\bxi/|\bxi|$.

We introduce the so-called N-function, and associated function spaces. A continuous function $\vphi$ is said to be a \emph{N-function} if the derivative $\vphi'$ exists, and in addition, $\vphi'$ has the following properties: right continuous, $\vphi'(0)=0$, $\vphi'(t)>0 $ for $t>0$, and $\lim_{t \rightarrow \infty} \vphi'(t)=\infty$.

We denote $\mathcal{L}^{\varphi}(\Omega):=\left\{w: \Omega \rightarrow \mathbb{R}: w \text { measurable and } \int_{\Omega} \varphi(|w| ) d x<\infty\right\}$ as the classical \emph{Orlicz space}, and we say $f\in \Wphi$, the \emph{Sobolev-Orlicz space}, if and only if $f$, $\nabla f \in \Lphi$. See Appendix \ref{sec:app:Orlicz} for more details, such as the definition of the norms $\|\cdot\|_{L^\vphi}$ and $\|\cdot\|_{\Wphi}$. Let $V:=\Wphi_0$ denote the Banach space endowed with $\|\cdot\|_{\Wphi}$ and homogeneous boundary condition, and $V^*$ be its dual space with norm $\|\cdot\|_{W^{-1,\vphi}}$. $V$ is a reflexive and separable Banach space. If $\vphi$ is a N-function and $f\in V^*$, \eqref{eqn:variational}/\eqref{eqn:weakform} is well-posed and admits a unique solution in $V$. 

We say that $\vphi$ satisfies the $\Delta_2$-condition, if there exists $c>0$ such that for all $t\geqs 0$, it holds that $\vphi(2t)\leqs c \vphi(t)$, and we denote $\Delta_2(\vphi)$ as the smallest such constant $c$. Since $\vphi(t)\leqs \vphi(2t)$, the $\Delta_2$ condition is equivalent to $\vphi(2t)\sim \vphi(t)$. We define the function $(\vphi')^{-1}:\R^+\to \R^+$ by $(\vphi')^{-1}(t):=\sup\{u\in \R^+:\phi'(u)\leqs t\}$. If $\vphi'$ is strictly increasing then $(\vphi')^{-1}$ is the inverse function of $\vphi'$.  $\vphi^*:\R^+\to \R^+$ with $\vphi^*(t):=\int_0^t (\vphi')^{-1}(s)\ds$
is again an N-function. $\vphi^*$ is the complementary function of $\vphi$ such that $(\vphi^*)^*=\vphi$. For any $a\geqs 0$, we define the shifted N-function by $\vphi_a'(t):=\frac{\vphi'(a\vee t)}{a\vee t}$ and $\vphi_a(t):= \int_0^{t} \vphi'(s)ds$.

Uniformly in $t\geqs 0$, we have $\vphi(t) \sim \vphi'(t) t$, 
where the constants depends only on $\Delta_2(\vphi, \vphi^*) :=\max(\Delta_2(\vphi), \Delta_2(\vphi^*))$. In addition, we make the following assumption, 
\begin{assumption}[N-function]
Let $\vphi$ be an N-function with strictly increasing $\vphi'$, and $\Delta_2(\{\vphi, \vphi^*\})<\infty$. We assume that
\begin{enumerate}
    \item $\vphi''$ exists, is right continuous, and satisfies $\vphi'(t) \sim t\vphi''(t)$ uniformly in $t\geqs 0$. 
    \item $\vphi''$ is non-decreasing.
    %(or the Simonenko-index $p^-\geqs 2$). 
\end{enumerate}
\label{asm:vphi}
\end{assumption}

\begin{remark}
Assumption \ref{asm:vphi} is satisfied by p-Laplacian $\vphi(t) = t^p/p$ for $p\geqs 2$, and also the regularized p-Laplacian defined in \ref{sec:prelim:reg}.
\end{remark}

We have the following properties of the flux function associated with an N-function $\vphi$.
\begin{property}
For $a(x, \bxi) = \kappa(x) \vphi'(|\bxi|)\frac{\bxi}{|\bxi|}$, and $\kappa(x)>0$, it holds that
\begin{enumerate}[(1)]
\item $(a(x, \bxi)-a(x, \bzeta))\cdot (\bxi - \bzeta) \geqs c\kappa(x)\vphi''(|\bxi|+|\bzeta|)|\bxi-\bzeta|^2$, 
\item $|a(x, \bxi) - a(x, \bzeta)| \leqs C\kappa(x)\vphi''(|\bxi|+|\bzeta|)|\bxi-\bzeta|$,
\item $a(x,\bm{0})=0$.
\end{enumerate}
where $\vphi$ is an N-function, the constants $c,C$ only depends on $\Dphi<\infty$. 
\label{prp:vphi}
\end{property}

Property \ref{prp:vphi} has been proved for p-Laplacian in \cite{Barrett:1993,Ebmeyer:2005}. We will provide more properties of N-functions in the appendix \ref{sec:app:n-function}.

The functional $\J$ is strictly convex. Furthermore, $\J$ is differentiable and second order Gateaux differentiable. Formally, its first and second variation are given as follows
\begin{equation}
 \J^{\prime}(u)(v)=\int_{\Omega}\kappa(x)\frac{\vphi'(|\nabla u|)}{|\nabla u|}  \nabla u\cdot \nabla v -\int_{\Omega} f v, 
\label{eqn:firstvar}
\end{equation}
\begin{equation}
\J^{\prime \prime}(u)(v, w) 
=\int_{\Omega}\kappa(x)\frac{\vphi'(|\nabla u|)}{|\nabla u|}\nabla v \cdot \nabla w +\int_{\Omega} \kappa(x) \frac{\vphi''(|\nabla u|)|\nabla u|-\vphi'(|\nabla u|)}{|\nabla u|^3}(\nabla u \cdot \nabla w)(\nabla u \cdot \nabla v).
\label{eqn:secondvar}
\end{equation}

We define the Bregman distance \cite{Bregman:1967} of $\J$ as
\begin{equation}
    \bdist{u}{v} := \J(u)-\J(v) -\J'(v)(u-v).
\end{equation}
      
\begin{lemma}
    We assume that $\vphi$ satisfies Assumption \ref{asm:vphi}. For any $u, v\in V$, there exists constants $0<\cj<\Cj$, which depend on $\Delta_2(\vphi, \vphi^*)$, and are independent of $\kappa$ such that
    \begin{equation}
        \cj \int\kappa(x)\vphi''(|\nabla u| + |\nabla (u-v)|)|\nabla(u-v)|^2 \leqs \bdist{u}{v} \leqs \Cj \int\kappa(x)\vphi''(|\nabla u| + |\nabla (u-v)|)|\nabla(u-v)|^2
        \label{eqn:bdist}
    \end{equation}
    see \ref{sec:app:prf:bdist} for the proof.
\label{lem:bdist}
\end{lemma}
The term $\int\kappa(x)\vphi''(|\nabla u| + |\nabla (u-v)|)|\nabla(u-v)|^2$ is called the quasi-norm associated with $\vphi$, which has already been defined in \cite{Barrett:1993,Ebmeyer:2005} for p-Laplacian equation.

\begin{remark}
For $p$-laplacian problem with $\vphi(t) = t^p/p$ with $p>1$, we have
\begin{equation}
\J(u)=\frac{1}{p} \int_{\Omega}\kappa(x)|\nabla u|^{p}-\int_{\Omega} f u, \quad u \in V
\end{equation}
where the Orlicz space $\Wphi_0(\Omega)$ reduces to $W_{0}^{1, p}(\Omega)$, and   $f \in W^{-1,q}(\Omega)$ ($1/p + 1/q = 1$). %, $g \in W^{1/q,p}(\Omega)$ is Dirichlet boundary data.  

\begin{equation}
 \J^{\prime}(u)(v)=\int_{\Omega}\kappa(x)|\nabla u|^{p-2} \nabla u \cdot \nabla v -\int_{\Omega} f v. 
\label{eqn:first derivative}
\end{equation}
\begin{equation}
\J^{\prime \prime}(u)(v, w) 
=\int_{\Omega}\kappa(x)|\nabla u|^{p-2} \nabla v \cdot \nabla w  +(p-2) \int_{\Omega}\kappa(x)|\nabla u|^{p-4}(\nabla u \cdot \nabla v)(\nabla u \cdot \nabla w). 
\label{eqn:second derivative}
\end{equation}
where $u,v,w \in V$. 
\end{remark}

%=============================================
\subsection{Regularization}
\label{sec:prelim:reg}
%=============================================

Let $u\in \Wphi_0$ be the solution of \eqref{eqn:variational}/\eqref{eqn:weakform}, $\kappa(x)\vphi'(\nabla u)/|\nabla u|$ can be treated as the coefficient of the elliptic equation \eqref{eqn:weakform}. It is possible that $u$ has critical points such that $|\nabla u|=0$, also $\vphi'(\nabla u)/|\nabla u|$ may grow large. The interaction between $\kappa(x)$ and nonlinearity may amplify such a degeneracy/singularity.  
% \llz{see \url{https://math.stackexchange.com/questions/306320/singularity-and-degeneracy-of-solution-to-p-laplace-equation}.}

Regularization techniques have been developed in \cite{Diening:2020} for $1< p< 2$, and for optimal control problem governed by $p-$Laplacian in \cite{casas2016approximation}. We adapt those methods and propose a two parameter regularized energy functional $\J_{\epsilon}$. Let regularization parameters be $\epsilon:=\{\epsilon_-, \epsilon_+\}$ such that $0<\epsilon_-<\epsilon_+$. We take $\vphi(t)=t^p/p$ as an example, and introduce a regularized N-function $\vphi_{\epsilon}: \mathbb{R}^+ \longrightarrow \mathbb{R}$, such that  $\vphi(t)=\lim _{\epsilon \rightarrow\{0, \infty\}} \vphi_{\epsilon}(t)$ for all $t \geqslant 0$ and $\vphi_{\epsilon}(t) \approx \epsilon_{+/-}^{p-2} t^{2}$ for $t\geqs \epsilon_+$ and $0\leqs t\leqs \epsilon_-$. See \eqref{eqn:regvphi} and \eqref{eqn:regvphi2} for two concrete examples. Notice that $\mathcal{J}_{\epsilon}(v)<\infty$, if and only if $v \in$ $W_{0}^{1,2}(\Omega)$.

We define the regularized energy functional as
\begin{equation}
\J_{\epsilon}(u)= \int_{\Omega}\kappa(x)\vphi_{\epsilon}(|\nabla u|)-\int_{\Omega} f u, \quad u \in  H_0^1.
\end{equation}

The regularization error in energy can be defined as,
\begin{equation}
    \err^\reg:=\J(u) - \J_\epsilon(u_\epsilon), 
\end{equation}
where $u_\epsilon$ is the energy minimizer of $\J_\epsilon$. It is beyond the scope of this work to investigate the decay of the regularization error with respect to the regularization parameters $\epsilon$. Instead, we have the following Lemma \ref{lem:regularization} (\cite[Lemma 5.2]{casas2016approximation}) and make Assumption \ref{asm:regularization} for the regularization error.

\begin{lemma}
	$\J_{\epsilon}$ has a unique minimizer $u_{\epsilon}\in H^1_0$ for any $\{\epsilon_{-}, \epsilon_{+}\} \subset (0,\infty)$. Let the sequence $\{\epsilon_{-,k}, \epsilon_{+,k}\}\longrightarrow \{0,\infty\}$, we obtain a sequence of minimizers $u_{\epsilon_k }$ of the regularized variational problems. Let $u$ be the solution of \eqref {eqn:energy}. It holds true that	
	\begin{equation}
	u_{\epsilon_k } \rightarrow u \quad i n \quad H_{0}^{1}(\Omega), \quad \text{ as } k \rightarrow \infty,
	\end{equation}
	\begin{equation}
	\chi_{\Omega \backslash \Omega_{k}\left(u_{\epsilon_k } \right)} \nabla u_{\epsilon_k }  \rightarrow \nabla u, \quad \text{ strongly in } L^{p}(\Omega)^{d}.
	\end{equation}
	\label{lem:regularization}
\end{lemma}
% Appendix \ref{sec:app:regulairzation} or

\begin{assumption}[Regularization]
For some $q>p>0$, we have
\begin{equation}
    \err^\reg \sim {\epsilon_-}^p + {\epsilon_+}^{p-q}.
\end{equation}
\label{asm:regularization}
\end{assumption}

%=============================================
\subsection{Finite Element Approximation}
\label{sec:prelim:fem}
%=============================================

We set up the finite element approximation of the weak formulation \eqref{eqn:weakform}. Let $\Tc$ be a coarse simplicial subdivision of $\Omega$, with the coarse mesh size $H:=\max_{K\in\Tc}H_K$, and $H_K := \mathrm{diam}(K)$. We assume that $\Tc$ is shape regular in the sense that $\max_{K\in\Tc} \frac{h_K}{\rho_K} \leqs \gamma$, for a positive constant $\gamma>0$, where $\rho_K$ is the radius of the inscribed circle in $K$. Let $\Omega_i^0$ be a coarse simplex, e.g. a coarse node $x_i$ or a coarse element $T_i$, we define the $\ell$-th layer patch $\Omega^\ell_i = \cup\{T\in\Tc: T\cap\Omega^{\ell-1}_i\neq \emptyset\}$ for $\ell\geqs 1$ recursively. A fine mesh $\Th$ with mesh size $h$ can be obtained by uniformly subdivide $\Tc$ several times. See Figures \ref{fig:mesh} and \ref{fig:patch} for illustrations.

The finite element space $\Vh$ contains continuous piecewise linear functions with respect to $\Th$ which vanish at the boundary $\partial \Omega$. We only consider continuous piecewise linear functions since higher order regularity of $p$-Laplacian problem is not guaranteed \cite{RN55}. The discrete finite element problem is defined in the following: Find $u_h\in \Vh$ such that
\begin{equation}
\int_{\Omega} a(x,\nabla u_h) \cdot \nabla v = \int_{\Omega} f v_h, \qquad \forall v_h \in \Vh.
\label{eqn:Galerkin}
\end{equation}
or equivalently 
\begin{equation}
\min _{v_h \in \Vh} \J(v_h)
\label{eqn:Ritz}
\end{equation}
% ($\mathcal{P}_h$) 

By \cite{Diening:2007,Ebmeyer:2005}, we have that $\|u-u_h\|\sim \epsilon(h)$ and 
\begin{equation}
\err^h: = \J_\epsilon(u_\epsilon)-\J_{\epsilon, h}(u_{\epsilon, h}),
\label{eqn:errh}
\end{equation}
such that $\epsilon(h)\to 0$ as $h\to 0$. In this paper, we assume that the fine mesh error $\epsilon(h)$ is negligible. 

% nonlinear operator is already regularized, and therefore, the linearized operator $A[\un]$ has uniform upper and lower bound.  In the following, we refer to the regularized functional $\J_\epsilon$ or operator $A_\epsilon$ as $\J$ or $A$, and the upper and lower bound of $A$ may depends on the regularization parameters $\epsilon$.

%%%%%%%%%%%%%%%%%%%%%%%%%%%%%%%%%%%%%%%%%%%%%%%%%%%%%%%%%%%%%
\subsection{Numerical Homogenization for Linear Elliptic Equation}
\label{sec:prelim:numhom}
%%%%%%%%%%%%%%%%%%%%%%%%%%%%%%%%%%%%%%%%%%%%%%%%%%%%%%%%%%%%%

For linear elliptic equation $-\mathrm{div} A(x) \nabla u = f$ with coefficient $A(x)\in (L^\infty)^{d\times d}$, $0<{m^-}:= \inf_{x\in\Omega}\lambda_{\min} (A(x))\leqs \sup_{x\in\Omega} \lambda_{\max} (A(x))=:m^+$. We call $m:=\{m^-, m+\}$ the set of lower and upper bounds. The associated finite element formulation is  
\begin{equation}
 A(\nabla u_h, \nabla v_h) = (f, v_h),  \qquad \forall v_h \in \Vh.
\end{equation}

The goal of \emph{numerical homogenization} is to identify a coarse space $\Vc$, such that the coarse solution $u_H\in \Vc$ of  
\begin{equation}
 A(\nabla u_H, \nabla v_H) = (f, v_H),  \qquad \forall v_H \in \Vc.
 \label{eqn:numhom}
\end{equation}
achieves (quasi-)optimal convergence rate, and also the construction of $\Vc$ achieves (quasi-)optimal complexity. By now, there have been a wide range of literatures on the numerical homogenization of linear elliptic operators, see \cite{dur91,ab05,weh02,ee03,abdulle2014analysis, ming2005analysis, berlyand2013introduction,Arbogast_two_scale_04, egw10, eh09, Review, hughes98, bazilevs2007variational, berlyand2010flux,Owhadi:2011, owhadi2014polyharmonic, egh12, chung2014adaptiveDG,  chung2015residual} for an incomplete list of references. In the following, we briefly introduce the (generalized) rough polyharmonic splines (GRPS) \cite{OwhadiZhangBerlyand:2014, Liu:2018,OwhadiMultigrid:2017}, as a representative approach for the numerical homogenization of linear elliptic equation with coefficient $A(x)$. 

Given $N_H$ suitable measurement functions $\psi_i$, $i=1, \dots, N_H$, for example characteristic function of coarse patches, we define the spaces $\Vi:=\{\phi\in \Vh| \int_{\Omega} \phi(x)\psi_j(x)\dx = \delta_{i,j}, j=1,\dots ,N_H \}$. The GRPS basis is given by the solution of the following constrained minimization problem which is strictly convex and admits a unique minimizer $\phi_i\in \Vi$,
\begin{equation}
	\phi_i = \argmin \limits_{\phi\in \Vi}\|\phi\|_{,\Omega}^2,
\label{eqn:phi}
\end{equation}	
for an appropriate norm $\|\cdot\|_{,\Omega}$, for example, the energy norm $\|\cdot\|_{A,\Omega}:=\int_\Omega A(x) |\nabla \cdot|^2 \dx$. 

The GRPS approach naturally induces a two-level decomposition of $\Vh$: We define the coarse space as $\Vc: = {\rm span}\{\phi_{i}\}$, and the fine space as $\Vf: = \{v\in \Vh |\int_\Omega \psi_i v \dx = 0,  i=1,\dots ,N_H\}$.
Let $\<\cdot, \cdot\>$ be the inner product induced by the norm $\|\cdot\|$, then $\Vc\perp \Vf$ with respect to the $\<\cdot, \cdot\>$ product. Furthermore, let $w_{\mathrm{I}}: = \sum_i (\int_\Omega \psi_i w \dx) \phi_i(x)$ be the interpolation of $w\in\Vh$ in $\Vc$, we have the optimal recovery property: 
\begin{displaymath}
\|w\|^2 = \|w_{\mathrm{I}}\|^2 + \|w-w_{\mathrm{I}}\|^2,
\end{displaymath}
and the optimal approximation property of $\Vc$, 
\begin{equation}
	\|w- w_{\mathrm{I}}\|_{L^2(\Omega)}\leqs C_{m} H \|w\|,
	\label{eqn:optapprox}
\end{equation}
the constants $C_m$ depends on the bounds $m$, the details of the proof can be found in \cite{OwhadiZhangBerlyand:2014,OwhadiMultigrid:2017}.

The basis in \eqref{eqn:phi} has a global support, which can be localized to a small patch of size $O(H\log(1/H))$, and still maintains the optimal accuracy of $O(H)$. Let $\Omega^{\ell}_i$ be a $\ell$-th layer coarse patch, we define the localized basis $\phi_i^\ell$ by the following constrained minimization problem, 
\begin{equation}
	\phi^\ell_i = \argmin \limits_{\phi\in \Vi\text{ and }\mathrm{supp}(\phi)\subset \Omega^\ell_i}\|\phi\|_{,\Omega^\ell_i}^2,	
\label{eqn:philoc}
\end{equation}

\begin{theorem}
We have the following properties of the localized basis,
\begin{enumerate}
    \item exponential decay of the truncation error: 
			\begin{displaymath}	
				\|\phi_i - \phi^{\ell}_i\| \leqs C \exp(-C'\ell),
			\end{displaymath}
			where the constants $C$, $C'$ depend on $m$.
	\item finite element error for the localized basis: if $\ell\simeq \log(1/H)$, $u_H^{\ell}$ is the FEM solution of \eqref{eqn:numhom} in $\Vc^\ell :={\rm span} \{\phi_i^{\ell}\}$, then
  \begin{equation}
  	\|u-u_H^{\ell}\|_{H^1(\Omega)}\leqs CH\|f\|_{L^2(\Omega)}.
  	\label{eqn:femerror}
  \end{equation} 
\end{enumerate}
\end{theorem}

\begin{remark}
For the nonlinear problem \eqref{eqn:variational}, we can still define global or local variation basis as in \eqref{eqn:phi} and  \eqref{eqn:philoc}. For example, the global basis can be defined by
\begin{equation}
	\phi_i = \argmin \limits_{\phi\in W^{1,\vphi}(\Omega), (\phi, \psi_j) = \delta_{ij}}\int_{\Omega} \kappa(x)\vphi(|\nabla \psi|)dx,	
	\label{eqn:nonlinearbasis}
\end{equation}	
and it still decays exponentially. However, the basis does not have the optimal approximation property as in \eqref{eqn:optapprox} due to the loss of orthogonality, see Figure \ref{fig:nonlinearbasis}. 
\begin{figure}[H]
	\centering
	{\includegraphics[width=0.35\textwidth]{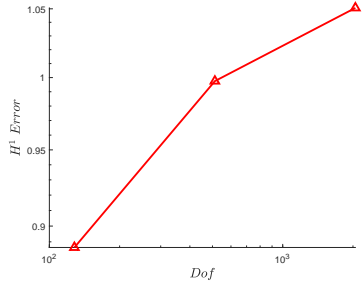}}
	\caption{For the multiscale trigonometric example in \ref{sec:numerics:mstrig}, we solve \eqref{eqn:weakform} using the nonlinear basis defined in \eqref{eqn:nonlinearbasis}. The $H^1$ error does not converge with respect to the degrees of freedom. }
	\label{fig:nonlinearbasis}
\end{figure}
\end{remark}

%=============================================
\section{Numerical Methods} 
\label{sec:methods}
%=============================================

We formulate the numerical methods in this section. For a nonlinear problem, we usually need to approximate the solution iteratively. We will first  introduce an `ideal' iterative method with exponential decay of energy error and 
convergence rate independent of the heterogeneity, which is based on the quasi-norm formulation. Next, we connect the widely used preconditioned gradient descent (PGD) and Newton's method with the quasi-norm based approach, and show the exponential energy decay with some reasonable assumptions. As a consequence, we will formulate the iterated numerical homogenization method, and its $L^2$ residual regularization for the multiscale nonlinear equations.

%=============================================
\subsection{Iterative Methods}
\label{sec:methods:iterative}
%=============================================

%=============================================
\subsubsection{Quasi-norm based Implicit Method} 
\label{sec:methods:quasinorm}
%=============================================

We formulate an `ideal' iterative method based on the quasi-norm introduced in Lemma \ref{lem:bdist}. Suppose that $\un\in \Wphi$ is the n-th iterative approximation to $u$, and $\wn$ is an increment such that $u^{(n+1)} = \un + \wn$, which will be specified later. 

By Lemma \ref{lem:bdist}, for any $\un$ and $\wn$, we have 
\begin{equation}
    \J(\un+\wn)-\J(\un)  \leqs   \J^{\prime}(\un)(\wn)+ \Cj \int_{\Omega}\kappa(x)\vphi''\left(|\nabla \un|+|\nabla \wn|\right)|\nabla \wn|^2.
\label{eqn:upaw}
\end{equation}	

We can define the increment $\wn$ by  
\begin{equation}
	C_q \int_{\Omega}\kappa(x)\vphi''\left(|\nabla \un|+|\nabla \wn|\right)\nabla \wn\cdot \nabla v =-\J^{\prime}(\un)(v),\, \forall v\in  W^{1,\vphi},
	\label{eqn:direction-wn}
\end{equation}
where $C_q = \Cj + 1$, such that the following lemma holds. 

\begin{lemma}
    Under Assumption \ref{asm:vphi}, also assume that $\un \in W_0^{1,\vphi}$, there exists a unique $\wn \in W_0^{1,\vphi} $ satisfying \eqref{eqn:direction-wn}. Furthermore, there exists a constant $C_1 > 1$ depending on $\vphi$ but independent of $\kappa$, such that 
	\begin{equation}
	\J(\un)-\J(u)  \leqs C_1 \int_{\Omega}\kappa(x)\vphi''\left(|\nabla \un|+|\nabla \wn|\right)|\nabla \wn|^2,
	\label{eqn:upperbound-qn}
	\end{equation}	
	and  
	\begin{equation}
		\J(\un+ \wn)-\J(\un) \leqs -\int_{\Omega}\kappa(x)\vphi''\left(|\nabla \un|+|\nabla \wn|\right)|\nabla \wn|^2. 
		\label{eqn:lowerbound-qn}
	\end{equation}
	See \ref{sec:app:existence-wn} for proof.
	\label{lem:quasinorm}
\end{lemma}

Let the energy error be defined as 
$$e^{(n)}:=\J(\un)-\J(u),$$ 
and we have the following theorem for its exponential decay.
\begin{theorem}
    For the direction $\wn$ defined in \eqref{eqn:direction-wn}, let $u^{(n+1)}: = \un +  \wn$, the energy error decays exponentially with 
    \begin{equation}
            e^{(n+1)} \leqs (1-\frac{1}{C_1}) e^{(n)}
    \end{equation}
    where $C_1$ only depends on $\Delta_2(\vphi, \vphi^*)$, and independent of $\kappa(x)$.
    \label{thm:quasinorm}
\end{theorem}

\begin{proof}
By lemma \ref{lem:quasinorm}, we have 
\begin{equation}
    \J(\un)-\J(u^{(n+1)})\geqs \frac{1}{C_1} (\J(\un)-\J(u))
\end{equation}
which leads to the desired result. 
\end{proof}

%=============================================
\subsubsection{Explicit Iterative Methods}
\label{sec:methods:explicit}
%=============================================

The quasi-norm based iterative method has the exponential convergence for energy error, and the convergence rate is independent of the oscillatory coefficients $\kappa$, by Theorem \ref{thm:quasinorm}. However, at each iteration, the nonlinear equation \eqref{eqn:direction-wn} needs to be solved to find the search direction $\wn$. Usually, we have to resort to some explicit iterative methods for $\wn$ in practice. This can be done, for example, by using a linearized/approximate operator $A[\un]$ to find $\wn$, and update $u^{(n+1)}=\un+\wn$. 

To be more precise, we define the first variation of $\J$ at $\un$,
\begin{equation}
  \J^{\prime}(\un)(v) =\int_{\Omega}\kappa(x)\frac{\vphi'(|\nabla \un|)}{|\nabla \un|} \nabla \un \cdot \nabla v -\int_{\Omega} f v.
\end{equation}

Let 
\begin{equation}
A_\sd[\un](\wn_\sd, v): = \int_{\Omega}\kappa(x) \frac{\vphi'(|\nabla \un|)}{|\nabla \un|} \nabla \wn_\sd \cdot \nabla v,
\label{eqn:N-func-direction-sd}
\end{equation}
be the linear operator corresponding to the preconditioned steepest descent method. 
For p-Laplacian it writes \cite{Huang:2007}
\begin{equation}
    A_\sd[\un](\wn_\sd, v): = \int_{\Omega}\left(\kappa(x)\left|\nabla \un\right|^{p-2}\right) \nabla \wn_\sd \cdot \nabla v.  
\label{eqn:direction-sd}
\end{equation}

Also, let  
\begin{equation}
\begin{aligned}
A_\nt[\un](\wn_{\nt}, v): =& \int_{\Omega}\kappa(x) \frac{\vphi'(|\nabla \un|)}{|\nabla \un|} \nabla \wn_{\nt} \cdot \nabla v   \\
&+\int_{\Omega} \kappa(x) \frac{\vphi''(|\nabla \un|)|\nabla \un|-\vphi'(|\nabla \un|)}{|\nabla \un|^3}(\nabla \un \cdot \nabla \wn_{\nt})(\nabla \un \cdot \nabla v),
\end{aligned}
\label{eqn:N-func-direction-newton}
\end{equation}
be the linear operator corresponding to the Newton's method,
% in particular for p-Laplacian, it writes as 
% \begin{equation}
%     A_\nt[\un](\wn_{\nt}, v):  =\int_{\Omega}\kappa(x)|\nabla \un|^{p-2} \nabla \wn_{\nt} \cdot \nabla v  +(p-2) \int_{\Omega}\kappa(x)|\nabla \un|^{p-4}(\nabla \un \cdot \nabla \wn_{\nt})(\nabla \un \cdot \nabla v)
%     \label{eqn:direction-newton}
% \end{equation}

For $A[\un]=A_\sd[\un]$ or $A[\un]=A_\nt[\un]$, the search direction $\wn$ is then defined analogous to \eqref{eqn:direction-wn}, through the following equation,
\begin{equation}
    C_n A[\un](\wn, v) = -\J'(\un)(v), \text{ for } \forall v \in \Vh. 
    \label{eqn:direction-wn-reg}
\end{equation}
where $C_n$ is a scaling factor analogous to $C_q$ in \eqref{eqn:direction-wn}. 

We make the following assumption,
\begin{assumption}
Assume that for $\un \in H_0^1$, there exists an unique $\wn \in H_0^1 $ satisfying \eqref{eqn:direction-wn-reg}. Furthermore, There exists $C_n>0$ such that 
\begin{equation}
    C_nA[\un](\wn, \wn) \geqs (\Cj+1)\int_{\Omega}\kappa(x)\vphi''\left(|\nabla \un|+|\nabla \wn|\right)|\nabla \wn|^2.
    \label{ineq:upperBoundQuasinorm}
\end{equation}
\label{asm:cn}
\end{assumption}

\begin{remark}
For the well-posedness of \eqref{eqn:direction-wn-reg} in $H^1_0$, we need $A[\un]$ is bounded from above and below. For example, we can regularize the operator $A[\un]$ at each iteration, or alternatively, assume that we are solving a regularized variational problem as defined in \ref{sec:prelim:reg} in the first place. In the latter case, the above assumption \ref{asm:cn} can be satisfied uniformly by letting $C_n \simeq (\epsilon_+/\epsilon_-)^{p-2}$ for any $\un,\wn \in H_0^1$. For general $\vphi$-Laplacian problems without regularization, \eqref{ineq:upperBoundQuasinorm} can only be satisfied for small $\wn$, and the constant $C_n$ may depend on $\un$. 
\end{remark}

\begin{lemma}
	Under Assumptions \ref{asm:vphi} and \ref{asm:cn}, there exists a constant $C_2 > 1$ depending on $\Delta_2(\vphi)$, $\Delta_2(\vphi^*)$, $C_n$, but independent of $\kappa$, such that 
	\begin{align}
	\J(\un)-\J(u)  \leqs C_2 \int_{\Omega}\kappa(x)\vphi''\left(|\nabla \un|+|\nabla \wn|\right)|\nabla \wn|^2,
	\label{eqn:upperbound-qn-reg} \\
	\J(\un+ \wn)-\J(\un) \leqs -\int_{\Omega}\kappa(x)\vphi''\left(|\nabla \un|+|\nabla \wn|\right)|\nabla \wn|^2.
	\label{eqn:upperbound-nstep-qn-reg} 
	\end{align}
	\label{lem:quasinorm-reg}
	See \ref{sec:app:existence-wn} for proof.
\end{lemma}

\begin{theorem}
Under the conditions in Lemma \ref{lem:quasinorm-reg}, and furthermore we assume that $C_n$ in \eqref{eqn:direction-wn-reg} is uniformly bounded, namely, $C_n\leqs M_C$, $\forall n\in \mathbb{N}$. For the direction $\wn$ defined in \eqref{eqn:direction-wn-reg}, let $u^{(n+1)}: = \un +  \wn$, there exists a constant $0<\theta<1$ such that 
\begin{equation}
    e^{(n+1)} \leqs \theta e^{(n)},
\end{equation}
    where $\theta$ depends on $M_C$, $\Delta_2(\vphi)$, $\Delta_2(\vphi^*)$, and independent of $\kappa$.
    \label{thm:explicit}
\end{theorem}

We can formulate the iterative algorithm here.
\begin{algorithm}[H]
	\caption{ {\textbf{Iterative Methods}}}
	\label{alg:iterative}
	\begin{algorithmic}
		\STATE{
			\textbf{STEP 1}: Initialize $u^{(0)}$, $n=0$, and tolerance  $\varepsilon>0$.}
		
		\textbf{STEP 2}:
        Find the increment $\wn$ by solving \eqref{eqn:direction-wn} for the quasi-norm based approach, or \eqref{eqn:direction-wn-reg} for either PGD direction with \eqref{eqn:N-func-direction-sd} or Newton's direction with \eqref{eqn:N-func-direction-newton}. \\
        \textbf{STEP 3}:
		Update $u^{(n+1)}:=\un+ \wn$.\\
		\textbf{STEP 4}:
		If $(\J(\un)-\J(u^{(n+1)}))/|\J(u^{(n)})|< \ve$, \textbf{STOP}; else goto \textbf{STEP 2} and set $n:=n+1$.
	\end{algorithmic}
\end{algorithm}	

\begin{remark}
In practice, the constants $C_q$ and $C_n$ are not known \emph{a priori}. Instead, we can solve the equation 
 \begin{equation}
     A[\un](\wn, v) = -\J'(\un)(v), \text{ for } \forall v \in \Vh,
\end{equation}
and use a line search algorithm \cite{Nocedal:2006} to find $\alpha_n = \argmin_{\alpha\geqs 0}\J(\un + \alpha \wn)$, and update $u^{(n+1)} = \un + \alpha_n \wn$.
\end{remark}

%=============================================
\subsection{Iterated Numerical Homogenization}
\label{sec:methods:numhom}
%=============================================

In Algorithm \ref{alg:iterative}, each iteration requires solving elliptic type equation \eqref{eqn:direction-wn-reg} which combines the highly oscillatory coefficient $\kappa(x)$ and the nonlinearity from $\vphi$. Besides, the gradient $|\nabla \un|$ may possibly approach zero or grow large during the iteration. Therefore, a proper multi-scale solver can be employed to reduce the computational cost and maintain the coarse mesh accuracy. We are going to iteratively use the numerical homogenization method introduced in Section \ref{sec:methods:numhom} to solve \eqref{eqn:direction-wn-reg}. 

Let the norm associated with $A[\un]$ be $\|\cdot\|_{A[\un]}$, the corresponding coarse space basis $\phi^{(n)}_i$ can be defined by  
\begin{equation}
	\phi^{(n)}_i = \argmin \limits_{\phi\in \Vi}\|\phi\|_{\Ar[\un]}^2,	
\label{eqn:variationalbasis}
\end{equation}	
and the coarse space $\Vcn:=\mathrm{span}\{\phi^{(n)}_i\}$. 

We can find the search direction $\wnH\in \Vcn$ by solving \eqref{eqn:direction-wn-reg} in $\Vcn$, with $A[\un]$ given by \eqref{eqn:N-func-direction-sd} or \eqref{eqn:N-func-direction-newton}.
\begin{equation}
C'_n A[\un](\wnH , v)=-\J'(\un)(v), \quad \forall v \in \Vcn.
\label{eqn:direction-wnH}
\end{equation}		

% Let the coarsening error defined by 
% \begin{equation}
%     \err_H^{(n)}: = \J(\un+\wnr) - \J(\un+\wnrH)
% \end{equation}
% \xl{Has it been used?}

Similar to Assumption \ref{asm:cn}, we make the following assumption,
\begin{assumption}
There exists $C'_n>0$ such that, 
\begin{equation}
    C'_n A[\un](\wnH, \wnH) \geqs (\Cj+1)\int_{\Omega}\kappa(x)\vphi''\left(|\nabla \un|+|\nabla \wnH|\right)|\nabla \wnH|^2.
    \label{ineq:upperBoundQuasinormH}
\end{equation}
\label{asm:cnH}
\end{assumption}
For simplicity, we take $C'_n = C_n$ in both \eqref{ineq:upperBoundQuasinorm} and \eqref{ineq:upperBoundQuasinormH}. Please see Figure \ref{fig:coarseerrorpgd} for numerical results and discussions related to Assumptions \ref{asm:cn} and \ref{asm:cnH}. 

\begin{lemma}
Suppose that $\un \in \Vh$. There exists positive constants $C_{\Th}$ depending on the shape regularity of the mesh $\Th$, such that, 
\begin{equation}
    \|\J^{\prime}(\un)\|_{L_h^2} \leqs \frac{C_{\Th}}{h}\|a(x, \nabla\un)\|_{L^2}  + \|f\|_{L^2},
\end{equation}
where $\|\J^{\prime}(\un)\|_{L_h^2}$ is defined as 
$
\|\J^{\prime}(\un)\|_{L_h^2}:= \displaystyle \min_{v_h \in \Vh} \frac{< \J^{\prime}(\un), v_h >}{\|v_h\|_{L^2}}
$
\begin{proof}
The proof is based on the scaling argument that for any $v_h \in \Vh$, $\|\nabla v_h\|_{L^2} \leqs C_{\Th}/h \|v_h\|_{L^2}$, where the constant $C_{\Th}$ depends only on the mesh regularity. Therefore, we have 
$\|\J'(\un)\|_{L_h^2}  \leqs C_{\Th}/h \|a(x, \nabla\un)\|_{L^2} +\|f\|_{L^2}$.
\end{proof}
\end{lemma}

\begin{corollary}
	For the finite element solution $\wnH\in \Vcn$ to \eqref{eqn:direction-wn-reg}, where $\Vcn$ is the space of "good" basis (global basis or local basis with $\ell\simeq \log H$ such that \eqref{eqn:femerror} holds). We have
	\begin{equation}
		\|\wn_h-\wnH\|_{A[\un]}\leqs \frac{1}{C_n\lambda_{\min}(A[\un])^{1/2}} H  \|\J'(\un)\|_{L_h^2}.
	\label{eqn:errhom}
	\end{equation}
	where $\lambda_{\min}(A[\un]):=\inf_{x\in\Omega}\inf_{\|v\|=1}v^T a[\un](x) v$, and $a[\un](x)$ denotes the elliptic coefficient in the operator $A[\un]$.
    \label{thm:homoerror}
\end{corollary}

\begin{theorem}
    Under the conditions in Lemma \ref{lem:quasinorm-reg}, and furthermore we assume  that $1 \leqs C_n \leqs M_C$, $\|\J'(\un)\|_{L^2_h}^2 \leqs M_L$, and $\lambda_{\min}(A[\un])\geqs M_A >0$ uniformly for $n\geqs 0$. For the direction $\wnH$ defined in \eqref{eqn:direction-wnH}, and $u^{(n+1)}: = \un +  \wnH$, there exists a constant $C_H>0$ such that 
	\begin{equation}
	(e^{(n+1)}-C_H H^2) \leqs (1-\frac{\theta}{C_{\J}+1}) (e^{(n)}-C_H H^2),
	\end{equation}
	where $C_H$ depends on  $m$, $\Delta_2(\vphi, \vphi^*)$, $M_C$, $M_L$ and $M_A$.
	$\theta$ is the constant in Theorem \ref{thm:explicit}.
	\label{thm:inh}
\end{theorem}

\begin{proof}

By Assumption \ref{asm:cnH}, we have  
\begin{equation}
C_nA[\un](\wnH, \wnH) \geqs (\Cj+1)\int_{\Omega}\kappa(x)\vphi''\left(|\nabla \un|+|\nabla \wnH|\right)|\nabla \wnH|^2,
\label{ineq:upperBoundQuasinorm-homo}
\end{equation}

Combining Lemma \ref{lem:bdist}, \eqref{eqn:direction-wnH}, \eqref{eqn:errhom}, and  \eqref{ineq:upperBoundQuasinorm-homo}, we derive the following inequalities,
\begin{displaymath}
\begin{aligned}
\J(\un)-\J(\un+\wnH) & \geqs C_n A[\un](\wnH, \wnH) - \Cj\int_{\Omega}\kappa(x)\vphi''\left(|\nabla \un|+|\nabla \wnH|\right)|\nabla \wnH|^2\\
&  \geqs \frac{C_n}{\Cj+1} A[\un](\wnH, \wnH) \\
& = \frac{C_n}{\Cj+1} A[\un](\wn_h, \wn_h) + \frac{C_n}{\Cj+1} A[\un](\wn_h-\wnH, \wn_h-\wnH) \\
& \geqs  \frac{C_n}{\Cj+1} A[\un](\wn_h, \wn_h)-\frac{1}{(\Cj+1)C_n\lambda_{\min}(A[\un])}\|\J'(\un)\|_{L^2_h}^2 H^2  \\
& \geqs \frac{1}{\Cj+1} (\J(\un)-\J(\un+\wn_h)) - \frac{M_L}{(\Cj+1) M_A}  H^2.
\end{aligned}
\end{displaymath}

Let $e^{(n+1)}_H:=\J(\un+ \wnH)-\J(u)$, and $\tilde{e}^{(n+1)}:=\J(\un + \wn_h)-\J(u)$, by Theorem \ref{thm:explicit}, we have
\begin{displaymath}
        e^{(n)}_H-e^{(n+1)}_H  \geqs  \frac{1}{\Cj+1}(e^{(n)}_H-\tilde{e}^{(n+1)})- \frac{M_L}{(\Cj+1) M_A}  H^2  \geqs \frac{\theta}{\Cj+1}e^{(n)}_H- \frac{M_L}{(\Cj+1) M_A} H^2. 
\end{displaymath}
It follows by rearranging terms that,
\begin{displaymath}
(1-\frac{\theta}{C_{\J}+1})(e^{(n)}_H-\frac{ M_L}{\theta M_A} H^2)  \geqs  (e^{(n+1)}_H-\frac{ M_L}{\theta M_A} H^2),
\end{displaymath}
namely, the energy error decays exponentially up to order $O(H^2)$, which concludes the theorem.
\end{proof}

We summarize the iterated numerical homogenization method in the following algorithm.
\begin{algorithm}
	\caption{Iterated Numerical Homogenization}
	\begin{algorithmic}			
		\STATE{\textbf{STEP 1}: Initialize
            $u^{(0)}$, $n=0$, and tolerance $\varepsilon>0$.}\\		
            %=(-\DIV(\kappa(x)\nabla \cdot))^{-1}f
        \textbf{STEP 2}: Construct a coarse space $\V{n}{H}:= {\rm span}\{\phin{n}{i},i=1,...,N_{H}\}$, where each $\phin{n}{i}$ is obtained by solving the minimization problem \eqref{eqn:variationalbasis} associated with the quadratic form $A[\un]$.\\
        \textbf{STEP 3}: Find the search direction $\wnH$ by solving  \eqref{eqn:direction-wnH} for either the PGD direction in \eqref{eqn:N-func-direction-sd} or the Newton's direction in \eqref{eqn:N-func-direction-newton}.\\
		\textbf{STEP 4}: Update $u^{(n+1)}:=\un+\wnH$.\\
		\textbf{STEP 5}:
		If $(\J(\un)-\J(u^{(n+1)}))/|\J(u^{(n)})|< \varepsilon$, \textbf{STOP}; else goto \textbf{STEP 2} and set $n:=n+1$.
	\end{algorithmic}
	\label{alg:inh}
\end{algorithm}

%=============================================
\subsection{Residual Regularization}
\label{sec:methods:reg}
%=============================================

As we mentioned in previous sections, the term 
\begin{displaymath}
\J'(\un) := -\DIV\left(a(x, \nabla \un)\right)+f = -\DIV\left(\kappa(x)\vphi'(|\nabla \un|)/|\nabla \un| \nabla \un\right) + f
\end{displaymath} 
only has $W^{-1,\vphi}$ regularity for general $\un\in\Wphi$. 
However, the approximation error in \eqref{eqn:errhom} depends on $ \|\J'(\un)\|_{L_h^2}^2$, which may blow up as $h\to 0$.
This can be observed from the following Figure \ref{fig:localbehavior}. Let  $\displaystyle \mathcal{E}_n(\alpha) := \J\left(\un+\alpha \wnH\right)$ and $\mathcal{R}_n(\alpha) := \|\J^{\prime}(\un+\alpha \wnH)\|_{L_h^2}^2$.  We show the local behavior of $\mathcal{E}_n(\alpha)$ and $\mathcal{R}_n(\alpha)$ in Figure \ref{fig:localbehavior}. In this example, we take the multiscale trigonometric example from \ref{sec:numerics:mstrig}, and $\wnH$ is obtained by Newton's method \eqref{eqn:N-func-direction-newton}. $\mathcal{E}_n(\alpha)$ attains its minimum at $\alpha \sim 1$, but the residual $\mathcal{R}_n(\alpha)$ has already blown up before this point. 
\begin{figure}[H]
	\centering
	{\includegraphics[width=0.45\textwidth]{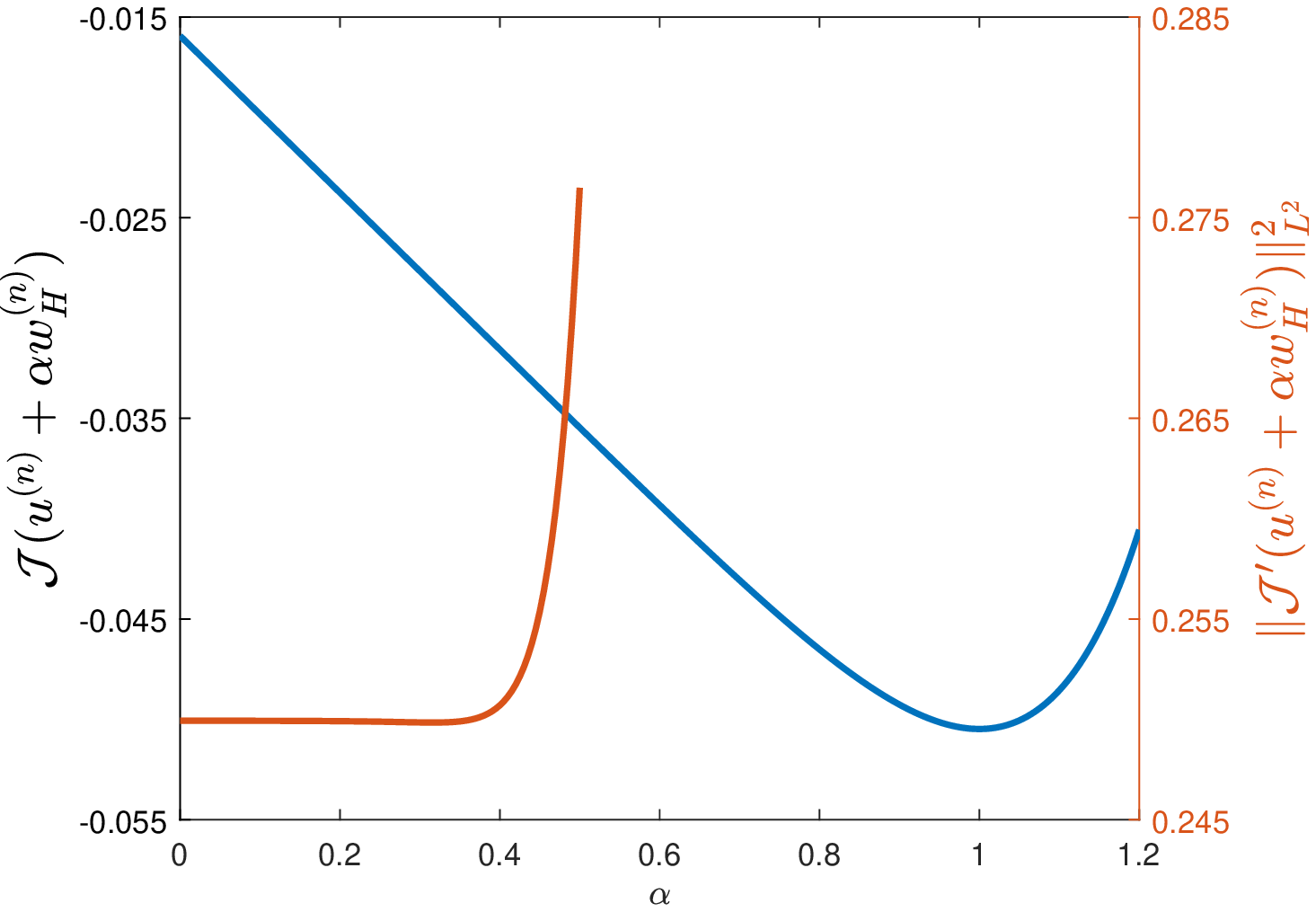}}
	\caption{Descent behavior of $\J\left(\un+\alpha_n \wnH\right)$ and  $\|\J^{\prime}(\un+\alpha_n \wnH)\|_{L^2}^2$.}
	\label{fig:localbehavior}
\end{figure}

In order to control the residual term $\|\J^{\prime}(\un)\|_{L_h^2}^2$ during iterations, we can prescribe an upper bound $L>0$, and relax the \textbf{STEP 4} of the Algorithm \ref{alg:inh} as, 

\textbf{STEP 4'}: Choose $0<\alpha_{n}\leqs 1$, such that 
\begin{displaymath}
    \alpha_n := \left\{
        \begin{array}{cc}
            1, & \quad \|\J'(\un+\wn_H)\|_{L_h^2} \leqs M_L;\\
            \argmin_{0<\alpha<1} \|\J'(\un+\alpha\wn_H)\|_{L_h^2}, &\quad \text{else}.
        \end{array}
    \right. 
\end{displaymath}
update $u^{(n+1)}:=\un+\alpha_{n} \wnH$. 

Alternatively, we can take $\|\J^{\prime}(\un+\alpha \wnH)\|_{L_h^2}^2$ as a penalty to regularize the energy minimization, and adopt the following line search formulation,

\textbf{STEP 4'}: Choose $\alpha_{n},$ such that 
$
\alpha_{n}:=\argmin _{\alpha \geqs 0} \J\left(\un+\alpha \wnH\right)+\lambda_n {\|\J^{\prime}(\un+\alpha \wnH)\|_{L_h^2}^2}
$, and update $u^{(n+1)}:=\un+\alpha_{n} \wnH$. 

A possible choice for the penalty parameter is $\lambda_n := \mathcal{E}_n'(0)/ |\mathcal{R}_n'(0)| $, such that the energy decrease and the residual change are comparable.

The $L^2$ residual regularization requires additional computational cost. To switch on/off the residual regularization, we propose the following indicator 
\begin{equation}
\rho^{(n+1)}:=\frac{\J(\un+ \alpha_n \wnH)-\J(\un)-\left<\J'(\un),\alpha_n \wnH\right>}{\left<\J'(\un),\alpha_n \wnH\right>}.
\label{eqn:rho}
\end{equation}
following the trust region idea \cite{Nocedal:2006}. If $\rho^{(n+1)}$ is close to 1/2, it indicates that $\J(\un+\alpha\wnH)\sim \J(\un)+\alpha  \left<\J'(\un), \wnH\right> + C(\alpha_n) \alpha^2/2 $ , then we can drop the residual regularization for the following iterations. 

We summarize the method the following Algorithm \ref{alg:inhrrls}.
\begin{algorithm}[H]
	\caption{Iterated Numerical Homogenization with Residual Regularized Line Search}
	\begin{algorithmic}			
		\STATE{\textbf{STEP 1}: Initialize
            $u^{(0)}$, $n=0$, $\I_{reg}=1$, tolerance $\varepsilon$, and a threshold $0.5<\delta<1$.} \\		
        \textbf{STEP 2}: Construct a coarse space $\V{n}{H}:= {\rm span}\{\phin{n}{i},i=1,...,N_{H}\}$, where each $\phin{n}{i}$ is obtained by solving the constrained minimization problem \eqref{eqn:variationalbasis} associated with the quadratic form $A[\un]$.\\
        \textbf{STEP 3}: Find the search direction $\wnH$ by solving  \begin{equation}
            A[\un](\wn, v) = -\J'(\un)(v), \text{ for } \forall v \in \V{n}{H}, 
        \end{equation}
        for either the PGD direction in \eqref{eqn:N-func-direction-sd} or the Newton's direction in \eqref{eqn:N-func-direction-newton}.\\
		\textbf{STEP 4.1}: \IF{$\I_{reg}=1$}
		                    \STATE Line search for  
		                    $$
		                    \alpha_{n} :=\argmin _{\alpha \geqs 0} \J\left(\un+\alpha \wnH\right)+\lambda_n\|\J^{\prime}(\un+\alpha_n \wnH)\|_{L_h^2}^2,
		                    $$ where $\lambda_n :=  \mathcal{E}_n'(0)/|\mathcal{R}_n'(0)|$.  Calculate $\rho^{(n+1)}$ in \eqref{eqn:rho}. \\
		                    \IF{$\rho^{(n)}\leqs\delta$}
		                    \STATE Set $\I_{reg}:=0$.
		                    \ENDIF \\
		                    \ELSE 
		                    \STATE Line search for 
		                    $$\alpha_n := \argmin _{\alpha \geqs 0} \J\left(\un+\alpha \wnH\right).$$
		                    \ENDIF \\
		\textbf{STEP 4.2}: Update $u^{(n+1)}:=\un+\alpha_{n} \wnH$. 
		
		\textbf{STEP 5}:
		If $(\J(\un)-\J(u^{(n+1)}))/|\J(u^{(n)})|< \varepsilon$, \textbf{STOP}; else goto \textbf{STEP 2} and set $n:=n+1$.
	\end{algorithmic}
	\label{alg:inhrrls}
\end{algorithm}

\section{Numerical Experiments} 
\label{sec:numerics}
%=============================================

In this section, we validate our numerical methods with some numerical experiments. We use the N-function $\vphi(t) = t^p/p$ with different exponents $p$ to characterize the nonlinearity. In practice, we use the regularized N-functions, $\vphi_\epsilon(t)$ as introduced in Section \ref{sec:prelim:reg}, and the specific form will be given and further investigated in Section \ref{sec: Reg}. For the heterogeneous coefficient $\kappa(x)$, we use: (i) the oscillatory multi-scale trigonometric function in Section \ref{sec:numerics:mstrig}, and (ii) the high-contrast heterogeneous permeability field which contains several channels, in the SPE10 benchmark for reservoir simulation (\url{http://www.spe.org/web/csp/}), in Section \ref{sec:numerics:spe10}. In Section \ref{sec:numerics:sparseupdate}, we address the issue of sparse updating to save the computational cost. All the simulations are carried out in a desktop computer with Intel(R) Core(TM) i7-7700 CPU @ 3.60GHz using MATLAB 2020a.

\subsection{Multiscale Trigonometric Example} 
\label{sec:numerics:mstrig}

The multiscale trigonometric (mstrig) coefficient $\kappa(x)$ is given by,
\begin{eqnarray}
\kappa(x)&=&\frac{1}{6}\big(\frac{1.1+\sin(2\pi x/\epsilon_1)}{1.1+\sin(2\pi y/\epsilon_1)}+\frac{1.1+\sin(2\pi y/\epsilon_2)}{1.1+\cos(2\pi x/\epsilon_2)}+\frac{1.1+\cos(2\pi x/\epsilon_3)}{1.1+\sin(2\pi y/\epsilon_3)}+\nn\\
&&\frac{1.1+\sin(2\pi y/\epsilon_4)}{1.1+\cos(2\pi x/\epsilon_4)}+\frac{1.1+\cos(2\pi x/\epsilon_5)}{1.1+\sin(2\pi y/\epsilon_5)}+\sin(4x^2y^2)+1\big). \label{9.3}
\end{eqnarray}
where $\epsilon_1=1/5, \epsilon_2=1/13, \epsilon_3=1/17, \epsilon_4=1/31,\epsilon_5=1/65$. $\kappa(x)$ is highly oscillatory with non-separable scales.  We show $\kappa(x)$, and also $\kappa(x)|\nabla u(x)|^{p-2}$ for $p=5$, which is the coefficient for the nonlinear equation \eqref{eqn:nonlinear}, in Figure \ref{fig:trig}.
\begin{figure}[H]
	\centering
	\subfigure[coefficient $\kappa(x)$ with contrast $37.10$.]{
		\includegraphics[width=0.29\textwidth]{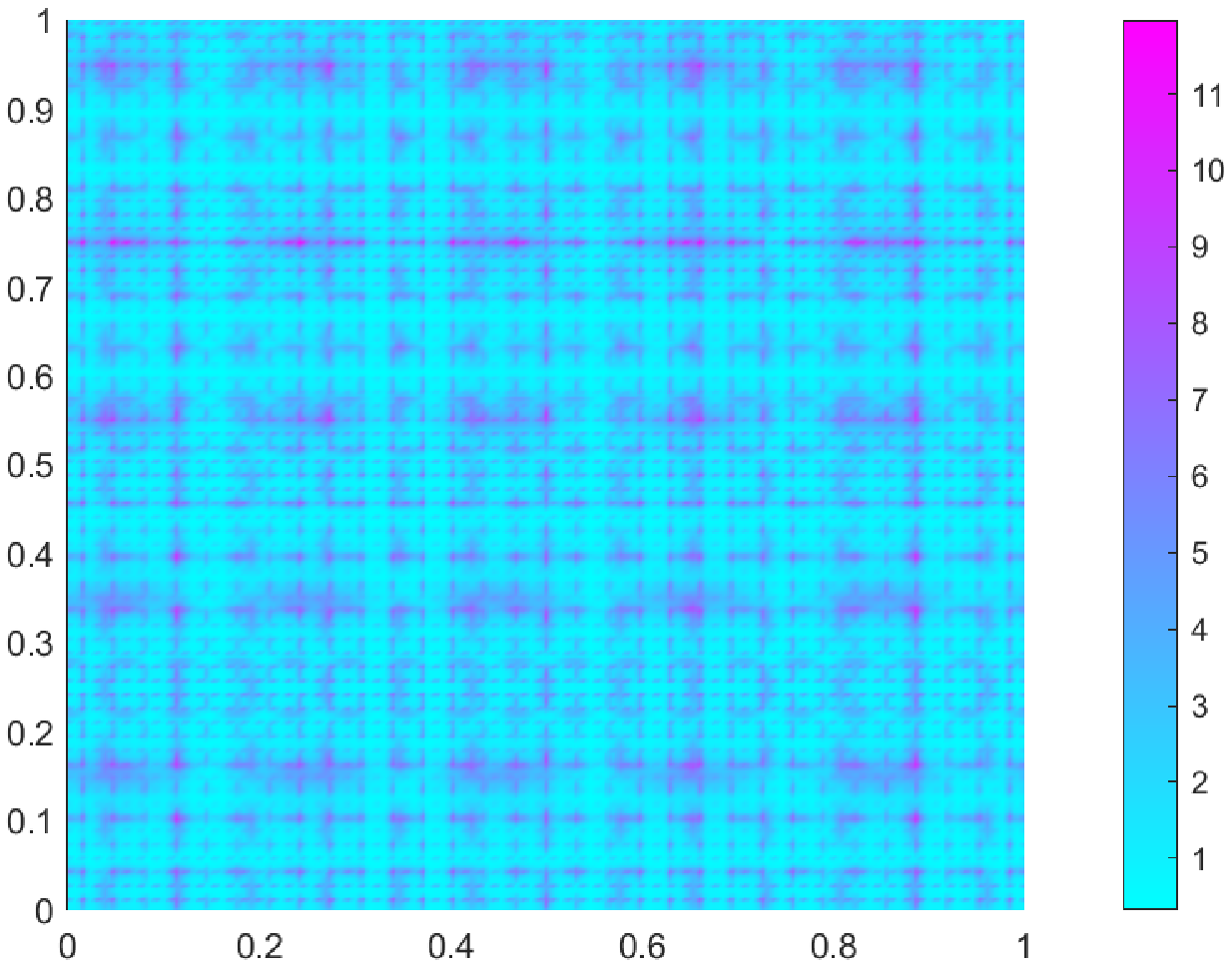}
	}
	\quad
	\subfigure[$|\nabla u|^{p-2}$, $\max(|\nabla u(x)|^{p-2})=0.8017$, $\min(|\nabla u(x)|^{p-2})=0$]{
		\includegraphics[width=0.29\textwidth]{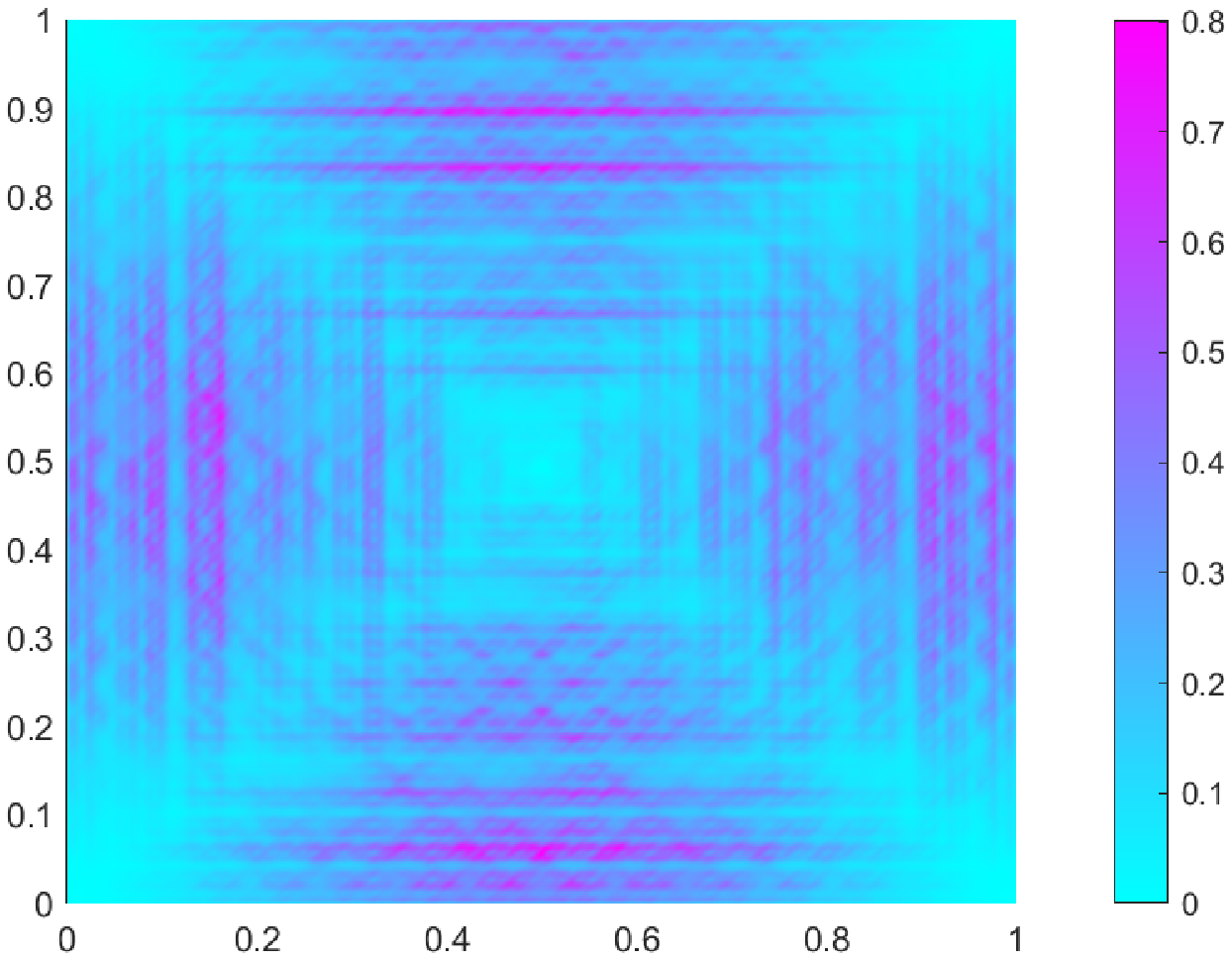}}
	\quad
	\subfigure[coefficient $\kappa(x)|\nabla u|^{p-2}$ for the nonlinear equation \eqref{eqn:nonlinear}, $\max(|\nabla u(x)|^{p-2})=1.722$, $\min(|\nabla u(x)|^{p-2})=0$]{
	\includegraphics[width=0.29\textwidth]{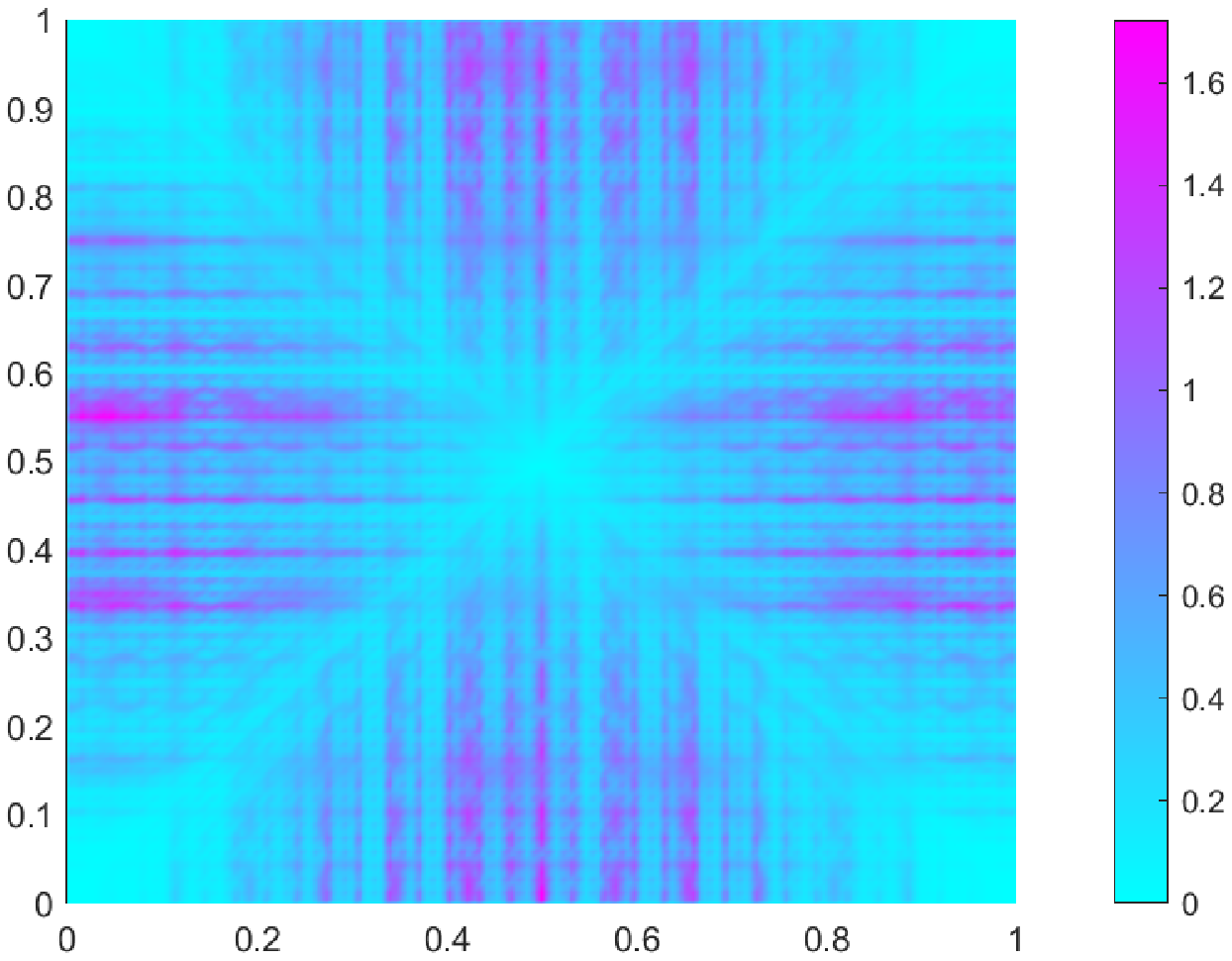}}
	\caption{coefficients $\kappa(x)$ and $\kappa(x)|\nabla u|^{p-2}$.}
	\label{fig:trig}
\end{figure}

We take the unit square $\Omega=[0,1]\times[0,1]$ as the computational domain. The coarse mesh $\mathcal{T}_H$ is obtained by first subdividing $\Omega$ uniformly into $\Nc\times \Nc$ squares, then partitioning each square into two triangles along the
$(1,1)$ direction. We can further refine the coarse mesh uniformly by dividing each triangle into four congruent subtriangles. We refine ${\mathcal{T}_H}$ $J$ times to obtain the fine mesh $\mathcal{T}_h$, with $h = 2^{-J} H$.  We refer to Figure \ref{fig:mesh} for an illustration of the partition. The degrees of freedom of the global GRPS basis with volume measurement functions are $N_H = 2N_c^2$. In the numerical experiment, we use a fixed fine mesh with $h=2^{-7}$ and coarse mesh sizes $H=2^{-2},\,2^{-3},\,2^{-4},\,2^{-5}$, respectively. 

\begin{figure}[H]
	\centering
	\subfigure[Coarse mesh, $N_c$=2]{
		\includegraphics[width=0.25\textwidth]{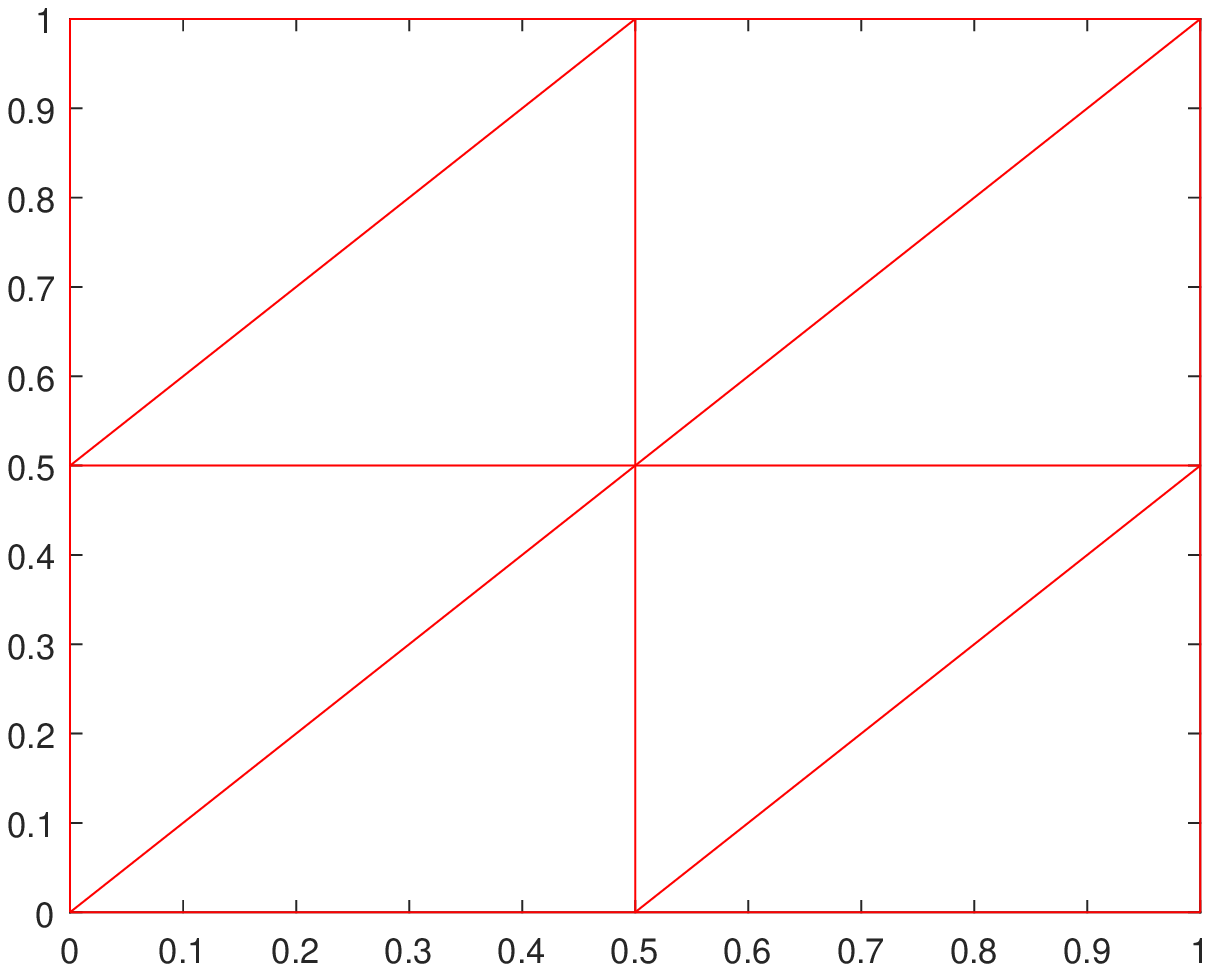}}
	\subfigure[Fine mesh, $N_c$=1,J=2]{
		\includegraphics[width=0.25\textwidth]{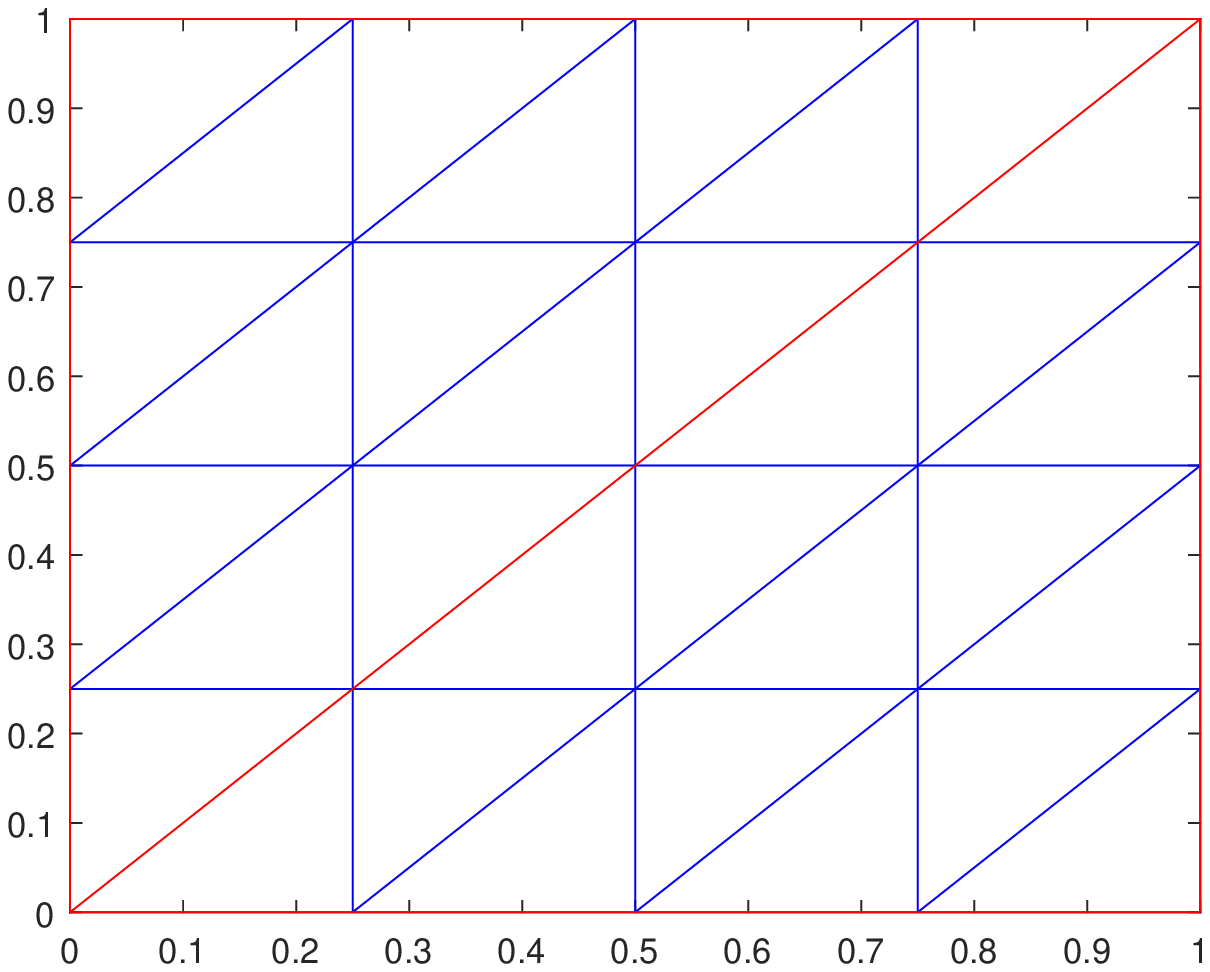}}
	\subfigure[Fine mesh, $N_c$=1,J=3]{
		\includegraphics[width=0.25\textwidth]{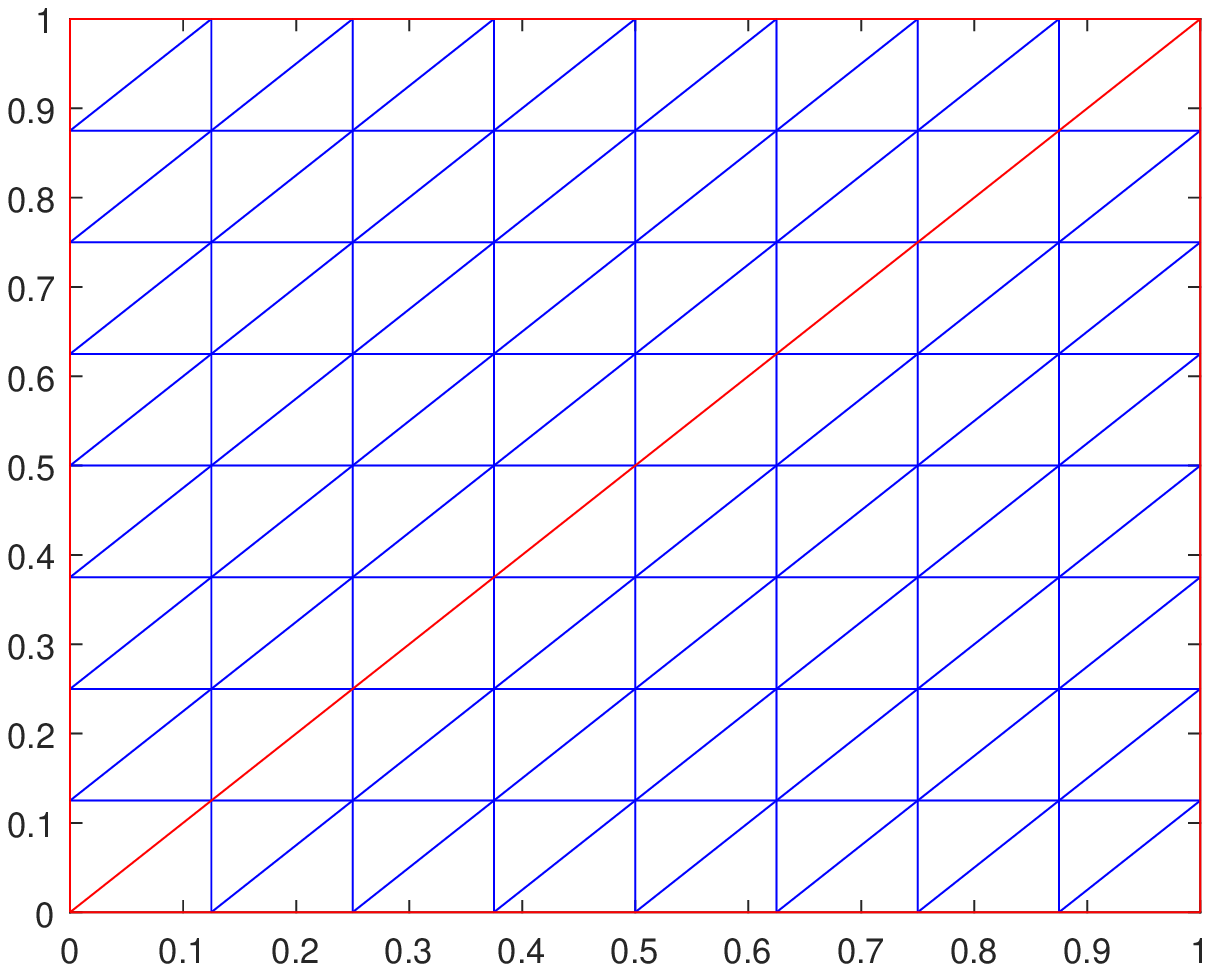}}
	\caption{Coarse and fine mesh of the unit square.}
	\centering
	\label{fig:mesh}
\end{figure}

We define the localized patch $\Omega_i^l$ by letting $\Omega_i^0:=T_i$ for coarse triangle $T_i\in \Tc$, and $\Omega_i^{l+1}: = \cup \{T\in \Tc: T\cap \bar\Omega_i^l \neq \emptyset\}$. We refer to Figure \ref{fig:patch} for an illustration of the patches $\Omega_i^l$ with $l=1,2,3$. 

\begin{figure}[H]
	\centering
	\subfigure[$\bar\Omega_i^1$]
	{\includegraphics[width=0.21\textwidth]{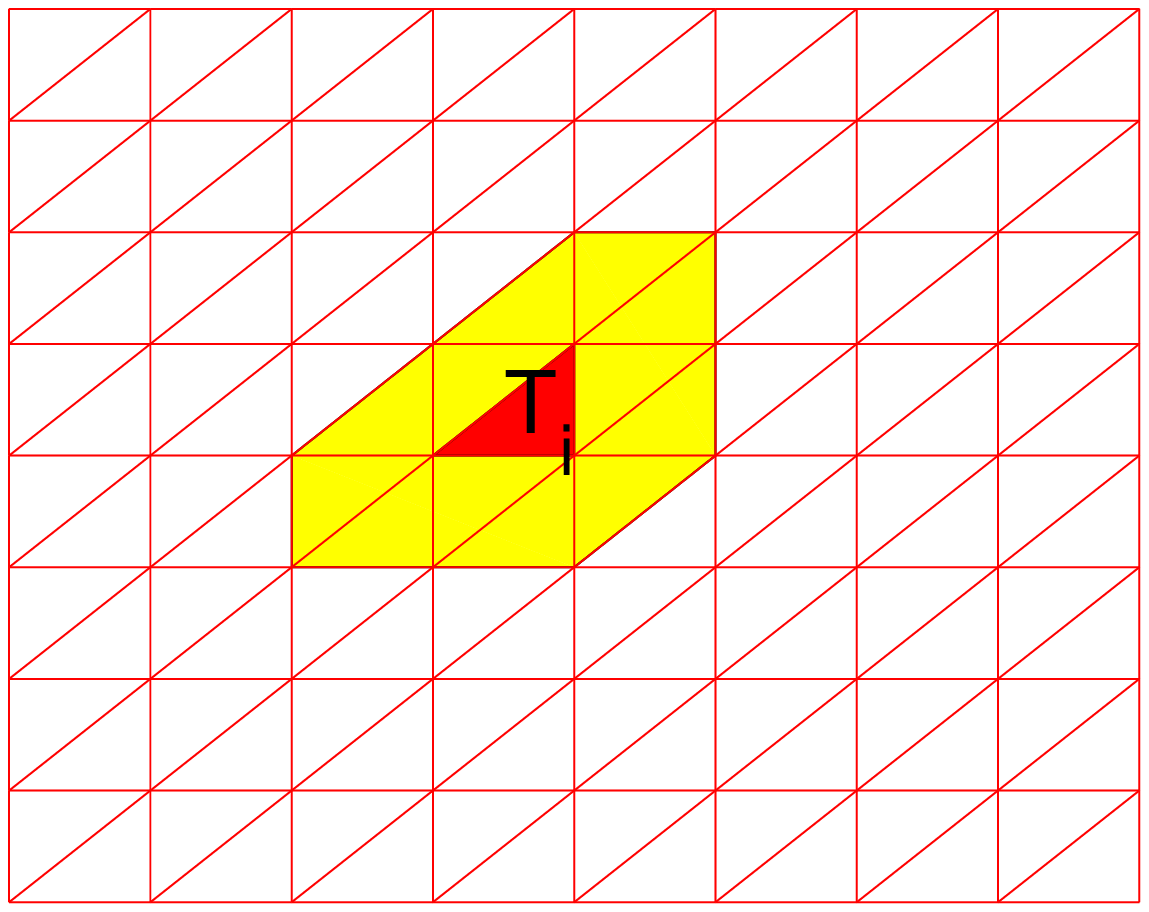}}
	\quad
	\subfigure[$\bar\Omega_i^2$]
	{\includegraphics[width=0.21\textwidth]{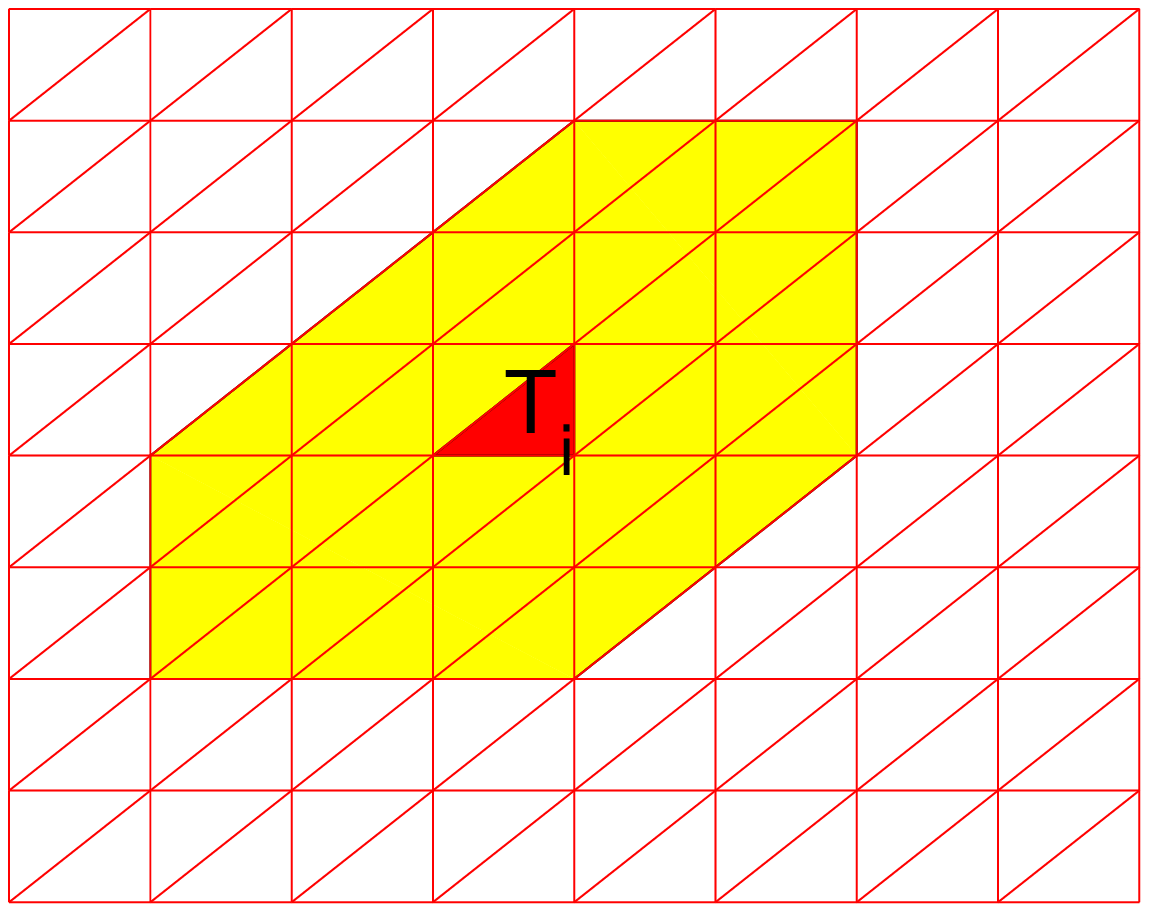}}
	\quad
	\subfigure[$\bar\Omega_i^3$]
	{\includegraphics[width=0.21\textwidth]{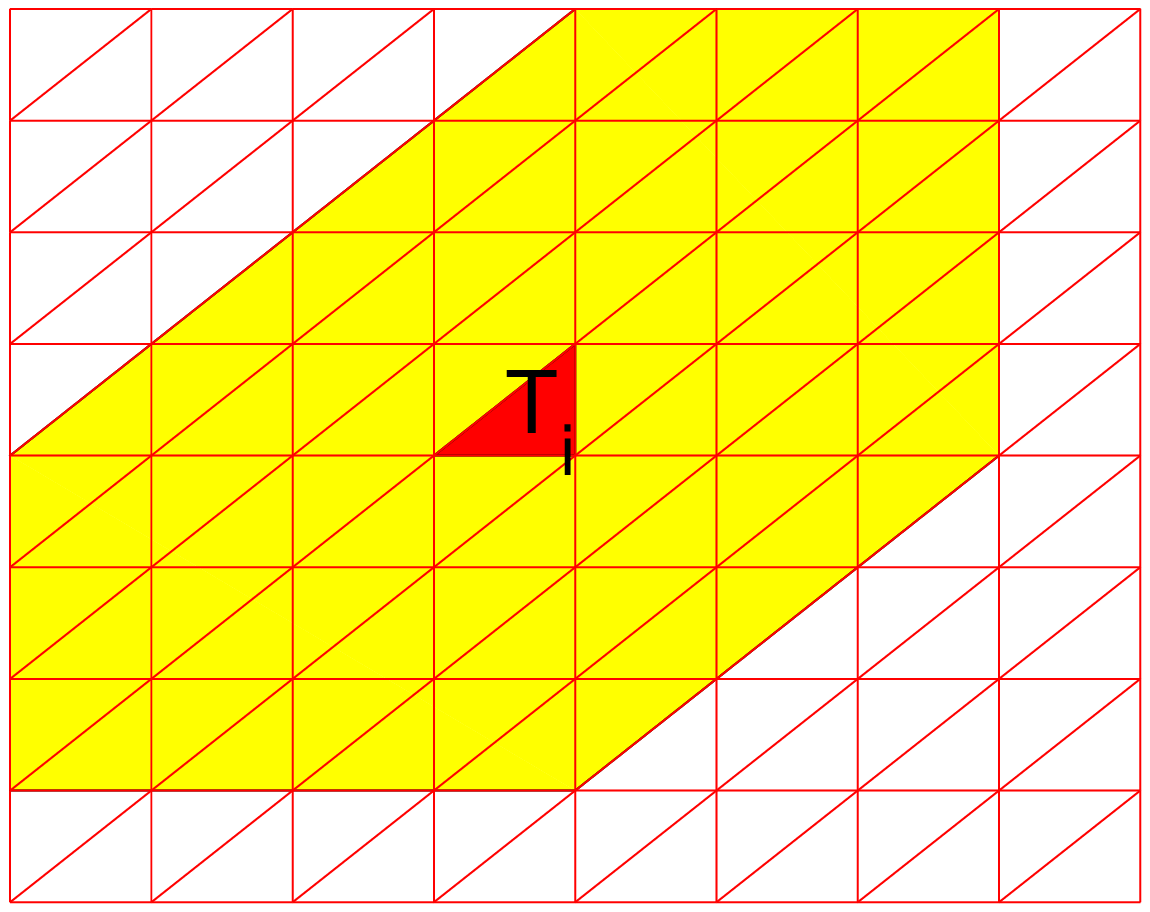}}
	\caption{Local patches for GRPS basis.}
	\centering
	\label{fig:patch}
\end{figure}

\subsubsection{Regularization error}
\label{sec: Reg}
We investigate the effect of regularization parameters $\epsilon$ in this section. In Figure \ref{fig:regerror}, we compare the regularization errors for the regularized models with the following N-functions $\vphi_{\epsilon}$, where $\vphi_{\epsilon,1}\in C^1$ is defined as,
\begin{equation}
		\vphi_{\epsilon,1}(t):=\left\{\begin{array}{ll}
		\frac{1}{2} \epsilon_{-}^{p-2} t^{2}+\left(\frac{1}{p}-\frac{1}{2}\right) \epsilon_{-}^{p} & \text { for } t \leqslant \epsilon_{-} \\
		\frac{1}{p} t^{p} & \text { for } \epsilon_{-} \leqslant t \leqslant \epsilon_{+} \\
		\frac{1}{2} \epsilon_{+}^{p-2} t^{2}+\left(\frac{1}{p}-\frac{1}{2}\right) \epsilon_{+}^{p} & \text { for } t \geqslant \epsilon_{+}
		\end{array}\right.
		\label{eqn:regvphi}
\end{equation}

and a smoother $\vphi_{\epsilon,2}\in C^2$ is defined as, 
\begin{equation}
		\vphi_{\epsilon,2}(t):=\left\{\begin{array}{ll}
		\frac{1}{p}\epsilon_-^{p-2} t^2+ \frac{p - 2}{p^2 + 2p} \epsilon_-^{-2}t^{p+2} -\frac{p - 2}{p (p + 2)} \epsilon_-^p & \text { for } t \leqslant \epsilon_{-} \\
		\frac{1}{p} t^{p} & \text { for } \epsilon_{-} \leqslant t \leqslant \epsilon_{+} \\
		\frac{p(p-1)}{2}\epsilon_+^{p-2} t^2+ (2-p) \epsilon_+^{p-1}t -\frac{p^2-3p+2}{2p} \epsilon_+^p & \text { for } t \geqslant \epsilon_{+}.
		\end{array}\right.
		\label{eqn:regvphi2}
\end{equation}

\begin{figure}[H]
	\centering
	\subfigure[Energy error]{\includegraphics[width=0.4\textwidth]{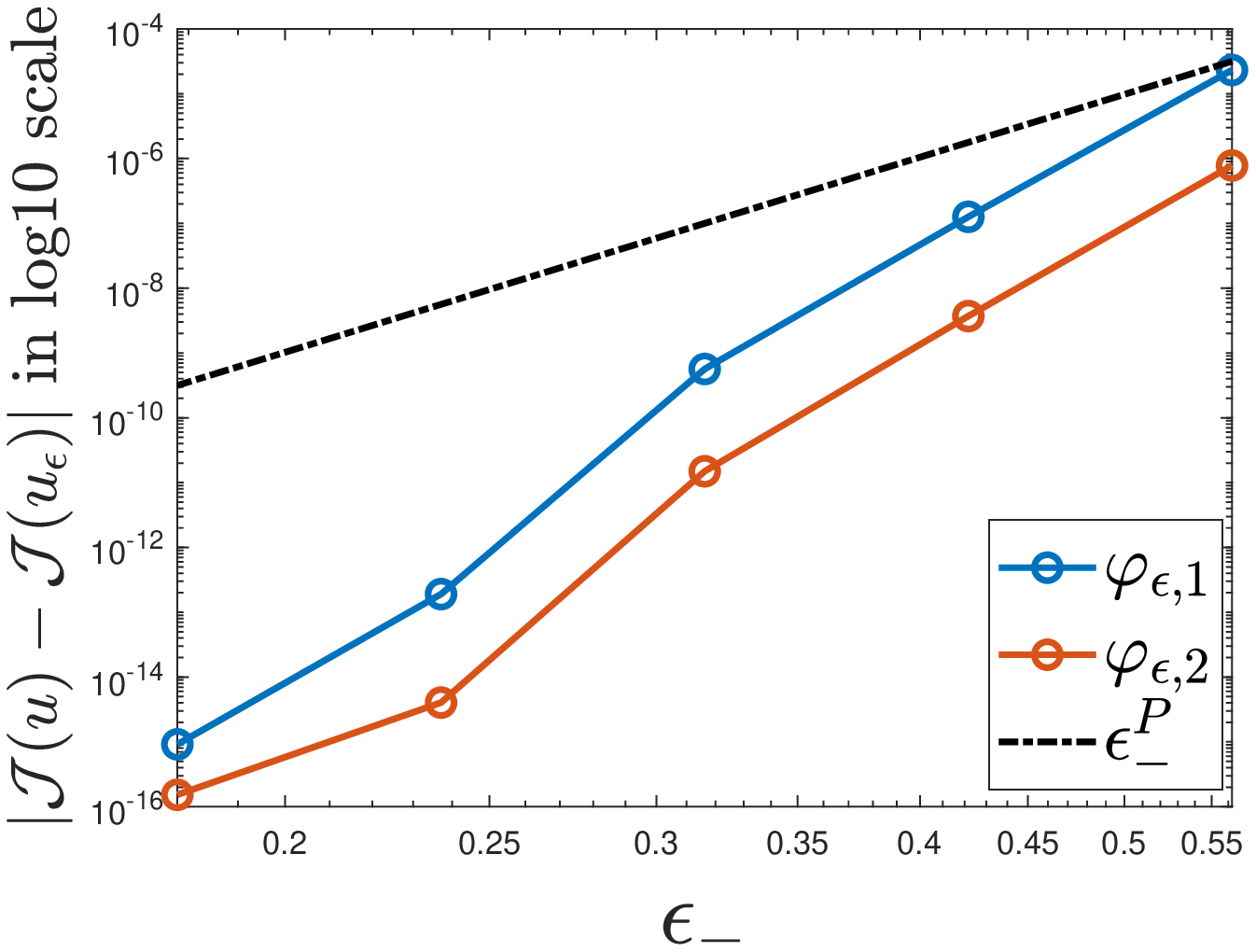}\label{fig:regModelerror}}
	\subfigure[$W^{1,p}$ error]{\includegraphics[width=0.4\textwidth]{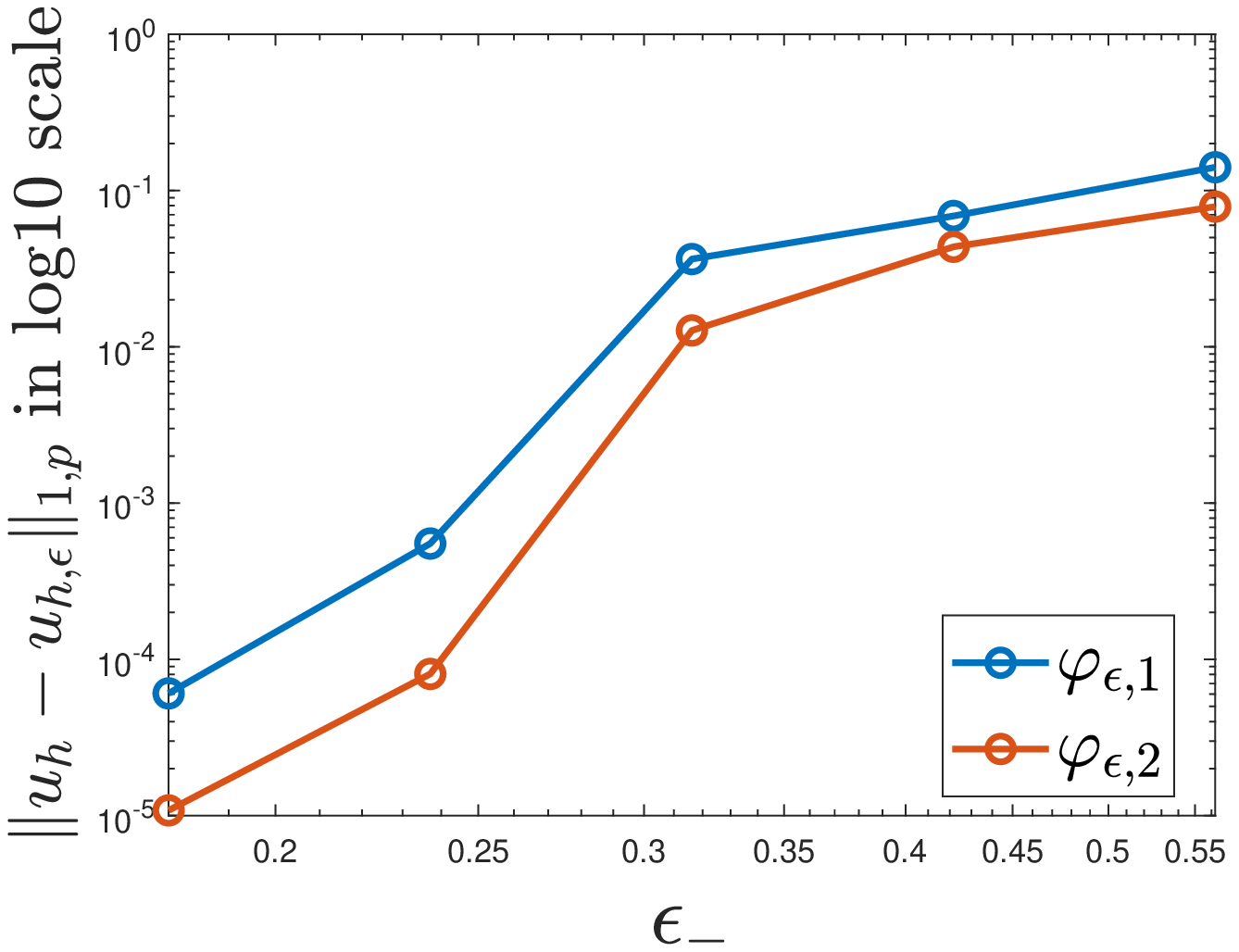}\label{fig:regModelerror2}}
	\caption{Regularization errors for $p=10$.}
	\label{fig:regerror}
\end{figure}

The fine mesh solutions for the original model and regularized models are discretized using piecewise linear finite elements, and solved by Matlab built-in  function \emph{fminunc} (trust region) with first optimality condition below $10^{-7}$. 

From our numerical results in Figure \ref{fig:trig} and Figure \ref{fig:regerror}, it seems for $p\geqs 2$, the coefficients $\kappa(x) |\nabla u|^{p-2}$ can be degenerate but still bounded from above. Therefore, only the regularization parameter $\epsilon_-$ from below is effective in this case. The regularization error has faster decay than $\epsilon_-^p$ postulated in Assumption \ref{asm:regularization}. We observe that when $\epsilon_-^{p-2} \simeq 10^{-6}$, the error in energy has reached the level of about $10^{-15}$. 

In the following numerical studies, we simply use the regularized model with $\epsilon_-^{p-2}=10^{-6}$ which guarantees sufficiently small regularization error as well as the well-posedness of the iterative scheme. For simplicity, we drop the subscript $\epsilon$ from the quantities such as $\J_\epsilon$, $u_\epsilon$, etc. 

\subsubsection{Comparison of iterative methods}

We compare different iterative methods on fine mesh in Figure \ref{fig:Implicit-Explicit}: the implicit quasi-norm based method, the gradient descent method (abbreviated as GD, $w_{GD}^{n} = \Delta^{-1}\J'(\un)$), the preconditioned gradient descent (PGD) method and Newton's method. We consider the regularized $\J_{\epsilon}$ with $\epsilon_-^{p-2}=10^{-6}$. The stopping criterion is $(\J(u^{(n)})-\J(u^{(n+1)}))/|\J(u^{(n)})|\leqs 10^{-15}$. We use the Matlab build-in function \emph{fminunc} for the line search. We choose the initial guess $u^{(0)}$ as the solution of the Poisson's equation with the coefficient $\kappa(x)$ unless otherwise specified. 

\begin{figure}[H]
    \centering
    \includegraphics[width=0.45\textwidth]{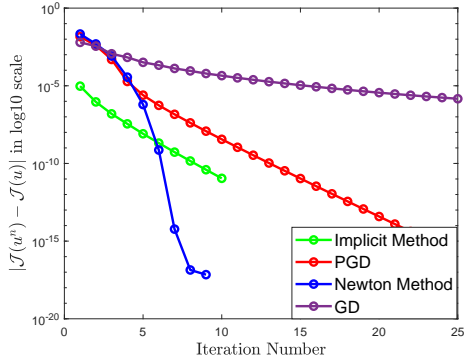}
    \caption{Comparison of iterative methods, $p=5$.  }
    \label{fig:Implicit-Explicit}
\end{figure}

As shown in Figure \ref{fig:Implicit-Explicit}, the implicit quasi-norm based method has the best start and exponential convergence in  energy as expected by Theorem \ref{thm:quasinorm}, and Newton's method achieves best convergence rate as a second order method. PGD has similar convergence rate as the implicit method, which seems to be independent of the heterogeneous coefficient $\kappa(x)$. The gradient descent method suffers from the heterogeneity of $\kappa(x)$, and converges very slow.
Therefore, an optimal iterative method may combine a starting step with the implicit quasi-norm based method and the consequent iterations with Newton's method.

% \begin{remark}
% \begin{equation}
% 	 \begin{aligned}
%          & A_\ch[\un](\wn_\ch, v): = \\
%          &\int_{\Omega}\kappa(x)(|\nabla \un|+ |\nabla (\un_{\nt})|)^{p-2} \nabla \wn_{\ch} \cdot \nabla v \\
%          & +(p-2) \int_{\Omega}\kappa(x)|\nabla \un|^{p-4}(\nabla \un \otimes \nabla \un) : (\nabla \wn_{\ch} \otimes \nabla v)\\
%          & +(p-2) \int_{\Omega}\kappa(x)|\nabla (\un_{\nt}) |^{p-4}(\nabla (\un_{\nt}) \otimes \nabla (\un_{\nt})) : (\nabla \wn_{\ch}\otimes\nabla v)\\
% 	 &= -\J^{\prime}(\un)(v) ,\ \forall v\in \Vh.
% 	 \end{aligned}
% 	 \label{eqn:direction-chebyshev}
% \end{equation}
% for the second derivative free modified Chebyshev's method, \cite{Hernandez:2000}. 
% \end{remark}

\subsubsection{Iterated numerical homogenization and residual regularization}
We present the implementation details and numerical results for the $L^2$ residual regularized iterated numerical homogenization in Section \ref{sec:methods:reg}.  Here, we fix the fine mesh size $h=2^{-7}$, the coarse mesh size $H=2^{-4}$, and $p=10$. We implement PGD and Newton's method on both the coarse mesh and the fine mesh. We take the tolerance $\varepsilon=10^{-15}$ and the threshold $\delta = 0.68$ in Algorithm \ref{alg:inhrrls}.

As shown in Figure \ref{fig:resreg}(a), the residual regularized iterated homogenization converges, while the iterated homogenization without residual regularization (green curve in \ref{fig:resreg}(a)) does not converge because of the residual behavior shown in Figure \ref{fig:localbehavior}. The convergence of the coarse mesh solution is slower than its fine mesh counterpart due to the error of the search direction $\wnH$, as well as the residual regularization. For the iterated homogenization with residual regularization, the error of the coarse mesh solution will eventually saturate due to the $O(H^2)$ homogenization error expected by Theorem \ref{thm:inh}. In Figures \ref{fig:resreg}(b)(c), the indicator $\rho_n$ approaches $0.5$ and the penalty parameter $\lambda_n$ decreases during the first 7 or 8 iterations, as a consequence, the residual regularization switches off thereafter. We also plot numerical estimates of $C_n$ and $C_n'$ in Figure \ref{fig:resreg}(d) as a numerical verification of Assumptions \ref{asm:cn} and \ref{asm:cnH}. $C_n$ is estimated by the solution $\widetilde{C}_n$ of the following nonlinear equation,
\begin{displaymath}
 \widetilde{C}_nA[\un](w_0^{(n)},w_0^{(n)}) = \int_{\Omega}\kappa(x)\vphi''\left(|\nabla \un|+\frac{1}{\widetilde{C}_n}|\nabla w_0^{(n)}|\right)|\nabla w_0^{(n)}|^2.
\end{displaymath}
where $w_0^{(n)}$ is given by $A[\un](w_0^{(n)}, v) = -\J'(\un)(v), \text{ for } \forall v \in \Vh$. $\widetilde{C}'_n$ is defined similarly. We observe that both $\widetilde{C}_n$ and $\widetilde{C}'_n$ have an upper bound controlled by the contrast of the linearized operators $A[\un]$, and decays to $O(1)$ as $n$ grows. 

\begin{figure}[H]
	\centering	
	\subfigure[Energy error ]{\includegraphics[width=0.38\textwidth]{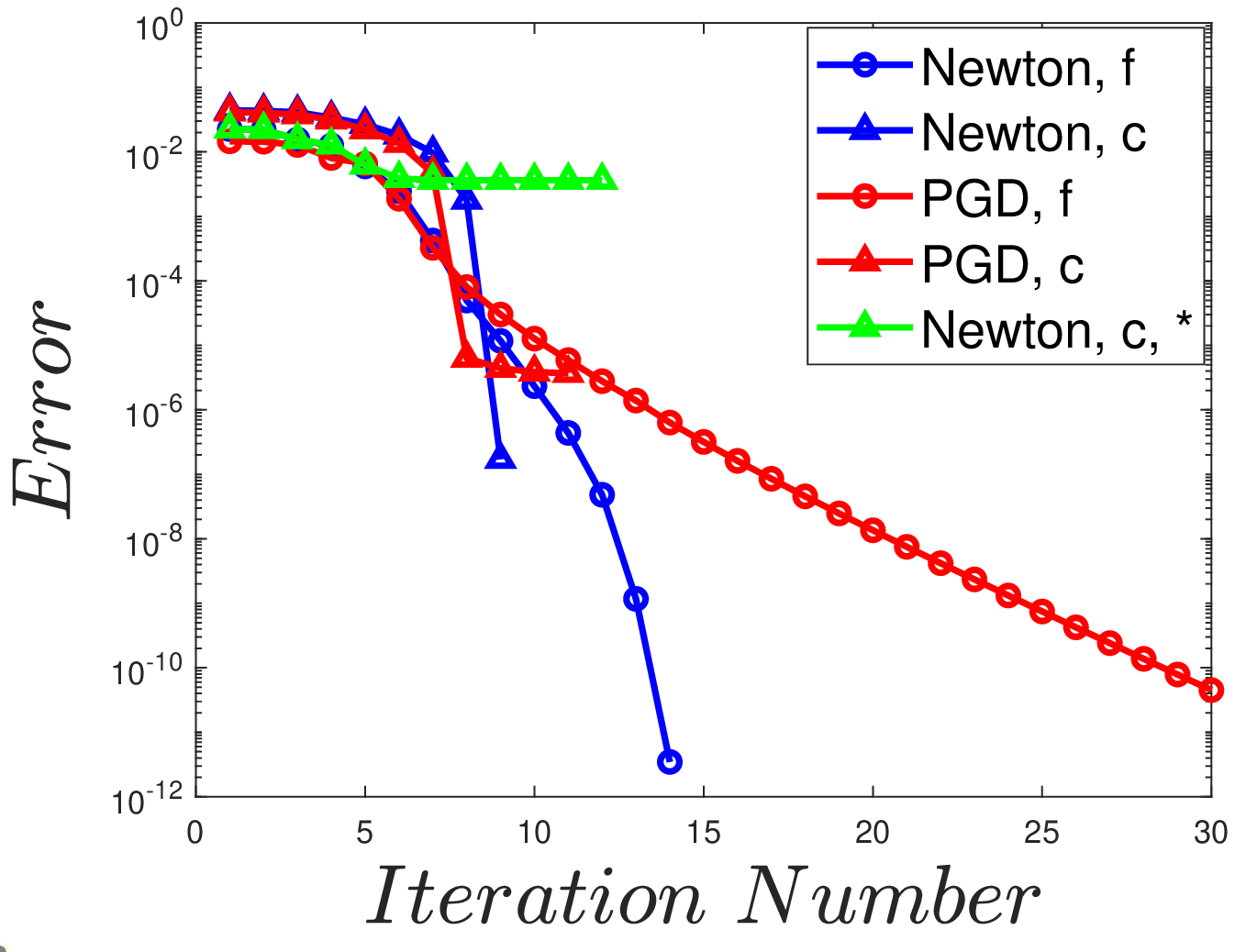}}
	\quad
	\subfigure[$\rho_n$]{\includegraphics[width=0.38\textwidth]{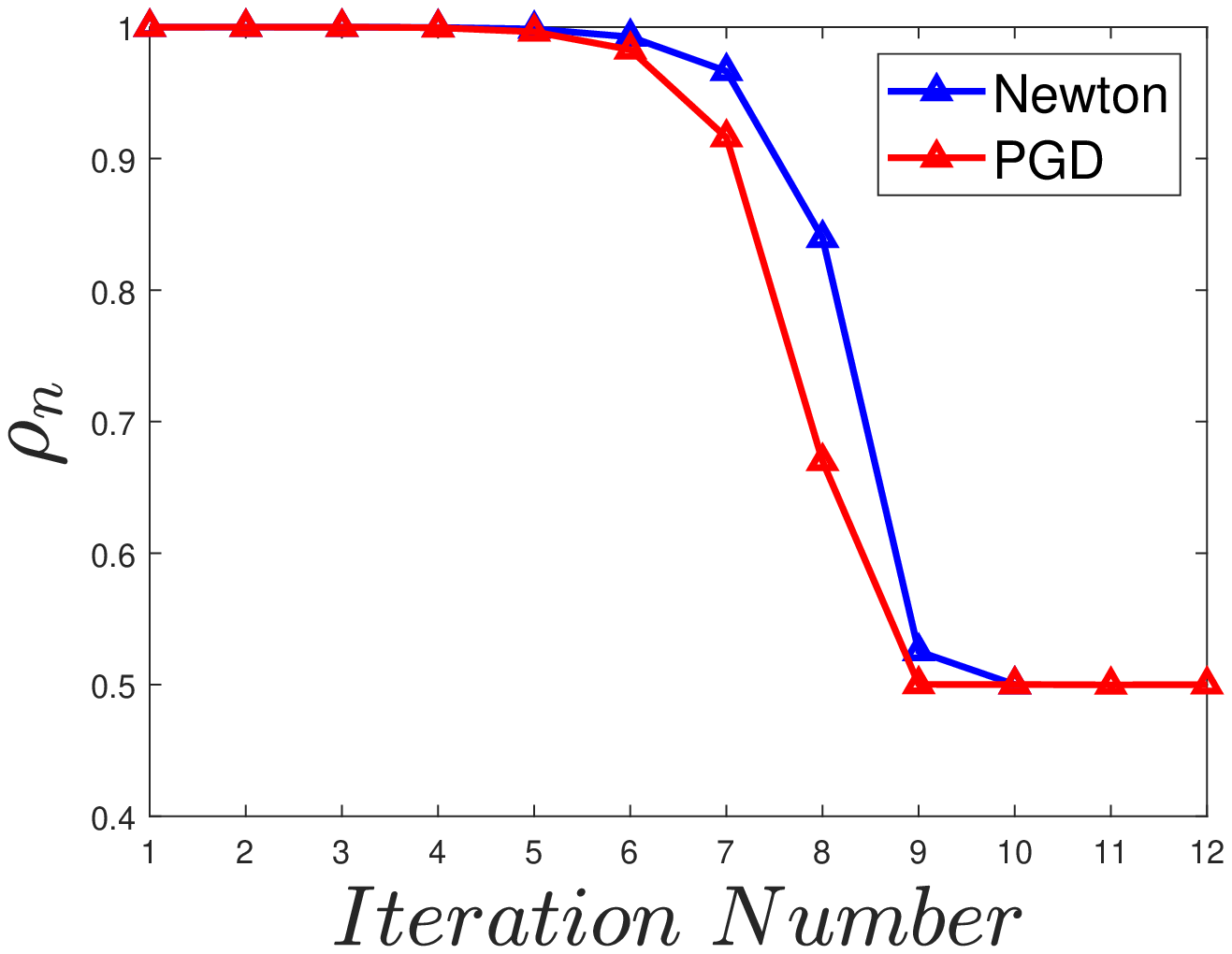}}
	\subfigure[$\lambda_n$]{\includegraphics[width=0.38\textwidth]{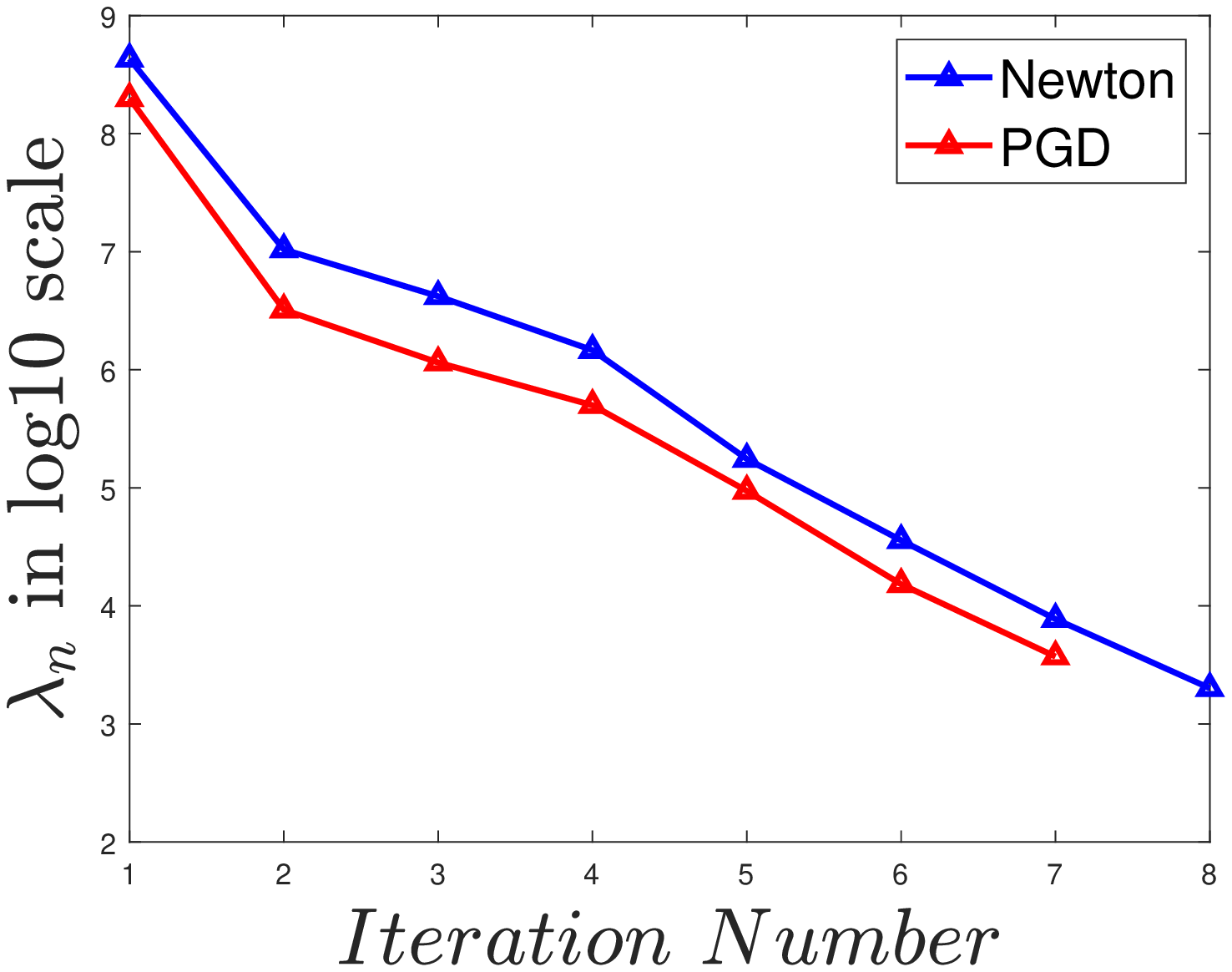}}
	\quad
	\subfigure[$\widetilde{C}_n$ and $\widetilde{C}'_n$]{\includegraphics[width=0.38\textwidth]{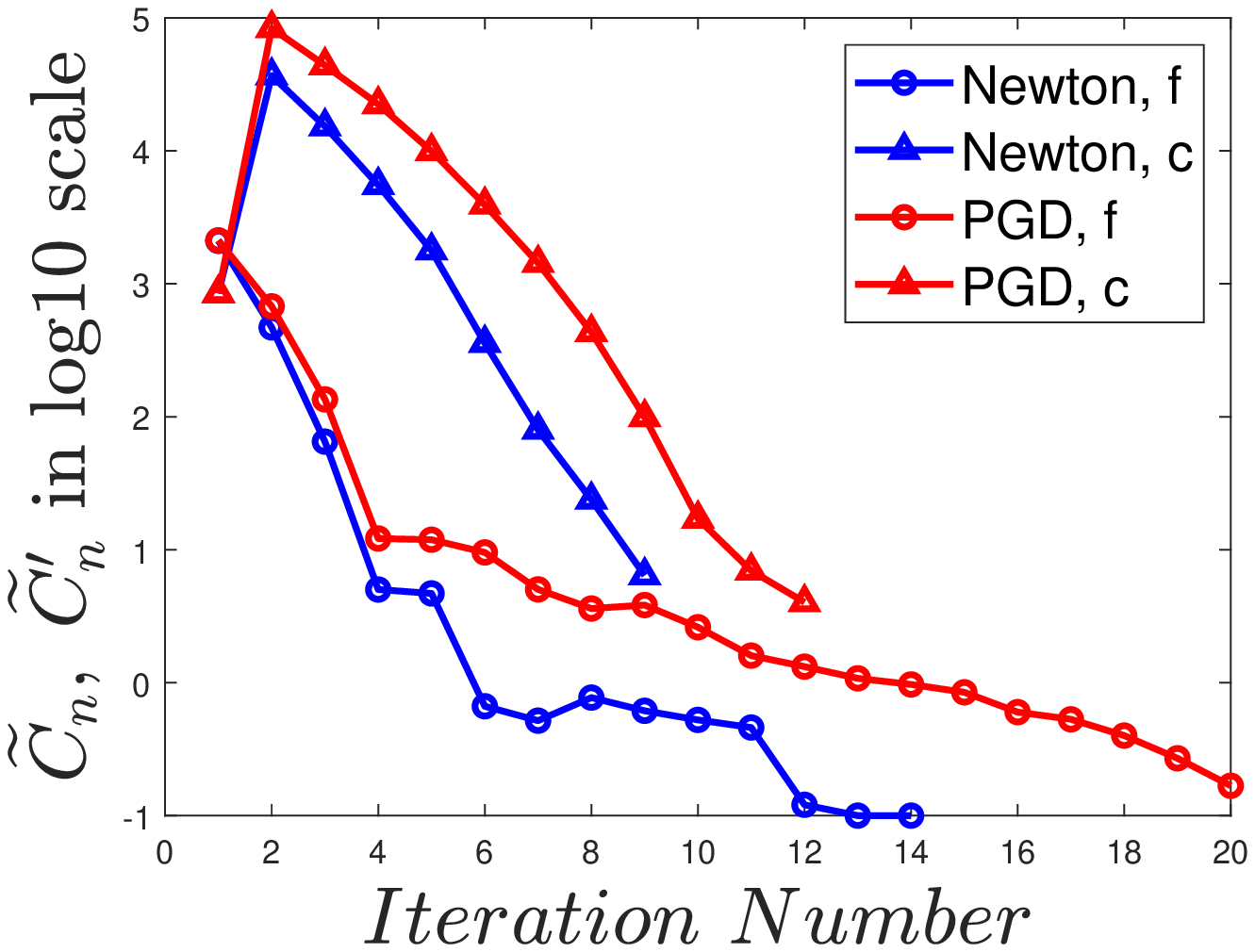}}
	\caption{Energy convergence, residual regularization indicator $\rho_n$, and penalty parameter $\lambda_n$. In the legend of (a), the letter 'c' denotes coarse mesh solution, 'f' denotes fine mesh solution; in (d), 'c' denotes $\widetilde{C}_n$, and 'f' denotes 
	$\widetilde{C}'_n$.}
	\label{fig:resreg}
\end{figure}

\subsubsection{Homogenization error}

We demonstrate the numerical homogenization errors of coarse mesh solutions in Figure \ref{fig:coarseerrorpgd} and Figure \ref{fig:coarseerrornewton}, with respect to the coarse degrees of freedom, for PGD and Newton methods, respectively. We show the $H^1$ error, the $W^{1,p}$ error, and the energy error of $u_{H}$ with respect to the minimizer $u_{h}\in \Vh$. In this example, we use a fixed fine mesh with $h=2^{-7}$ and varying coarse mesh sizes $H=2^{-2},\,2^{-3},\,2^{-4},\,2^{-5}$, respectively.

\begin{figure}[H]
	\centering
	\subfigure[$H^1$ error ]{\includegraphics[width=0.3\textwidth]{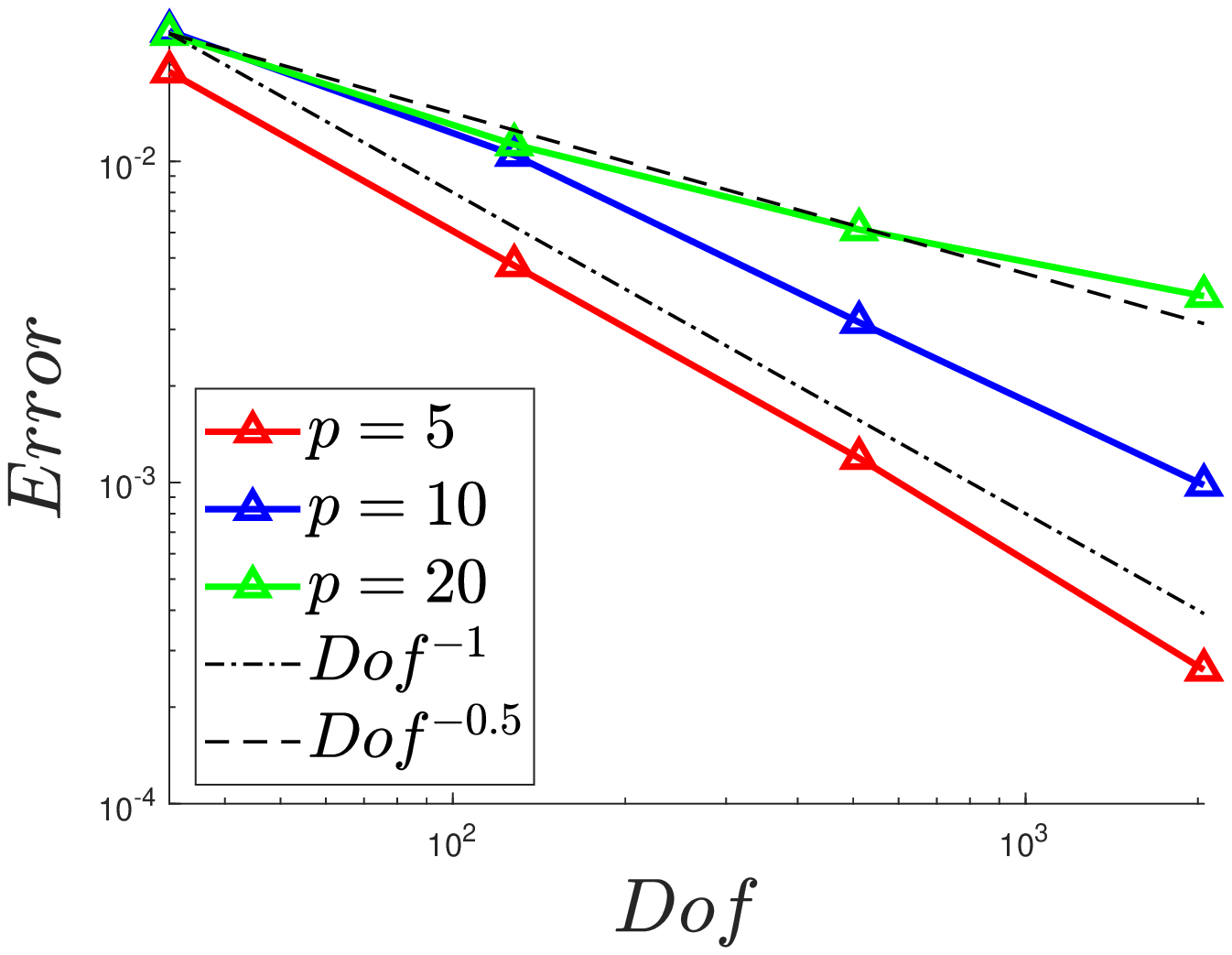}}
	\subfigure[$W^{1,p}$ error]{\includegraphics[width=0.3\textwidth]{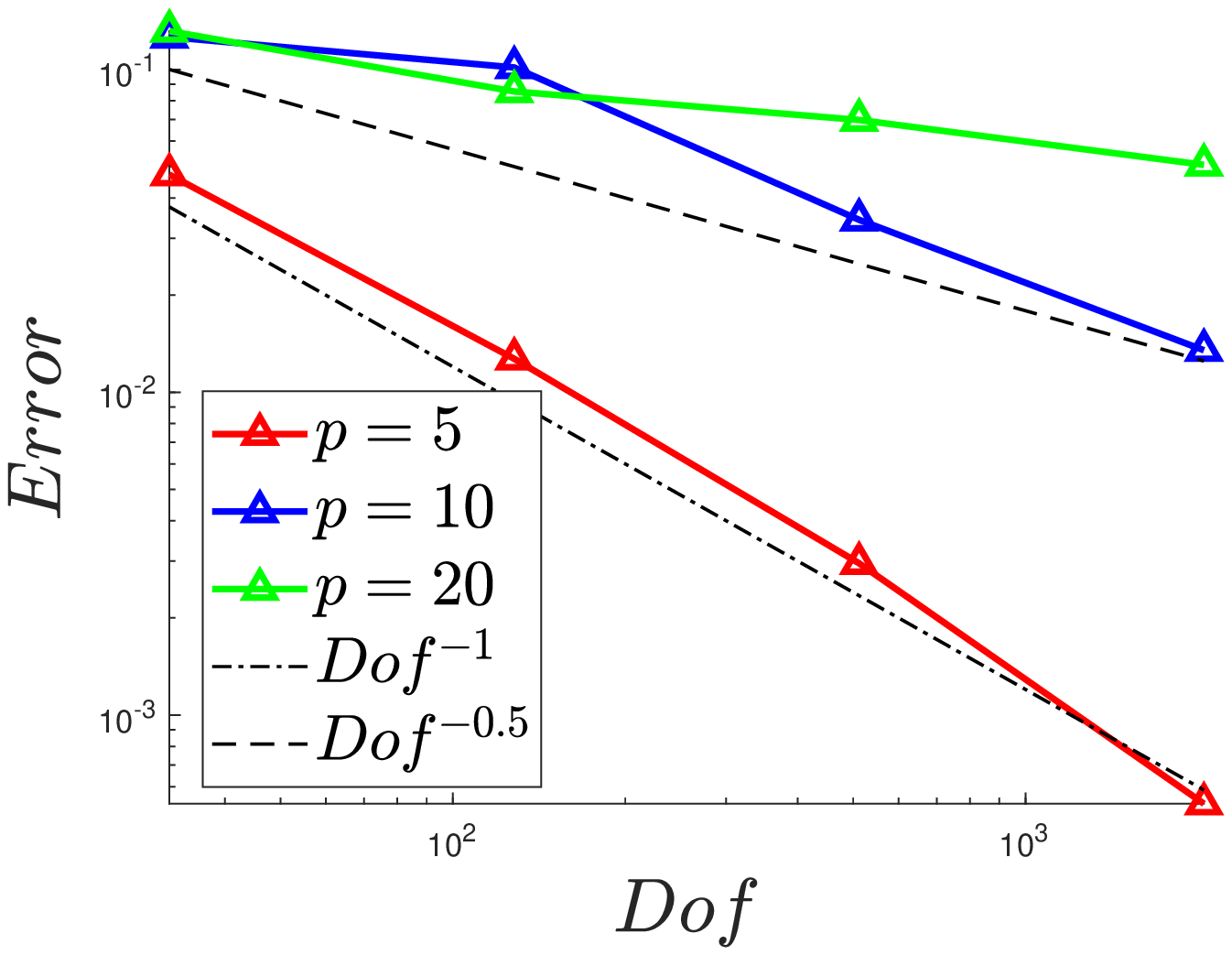}}
	\subfigure[energy error]{\includegraphics[width=0.3\textwidth]{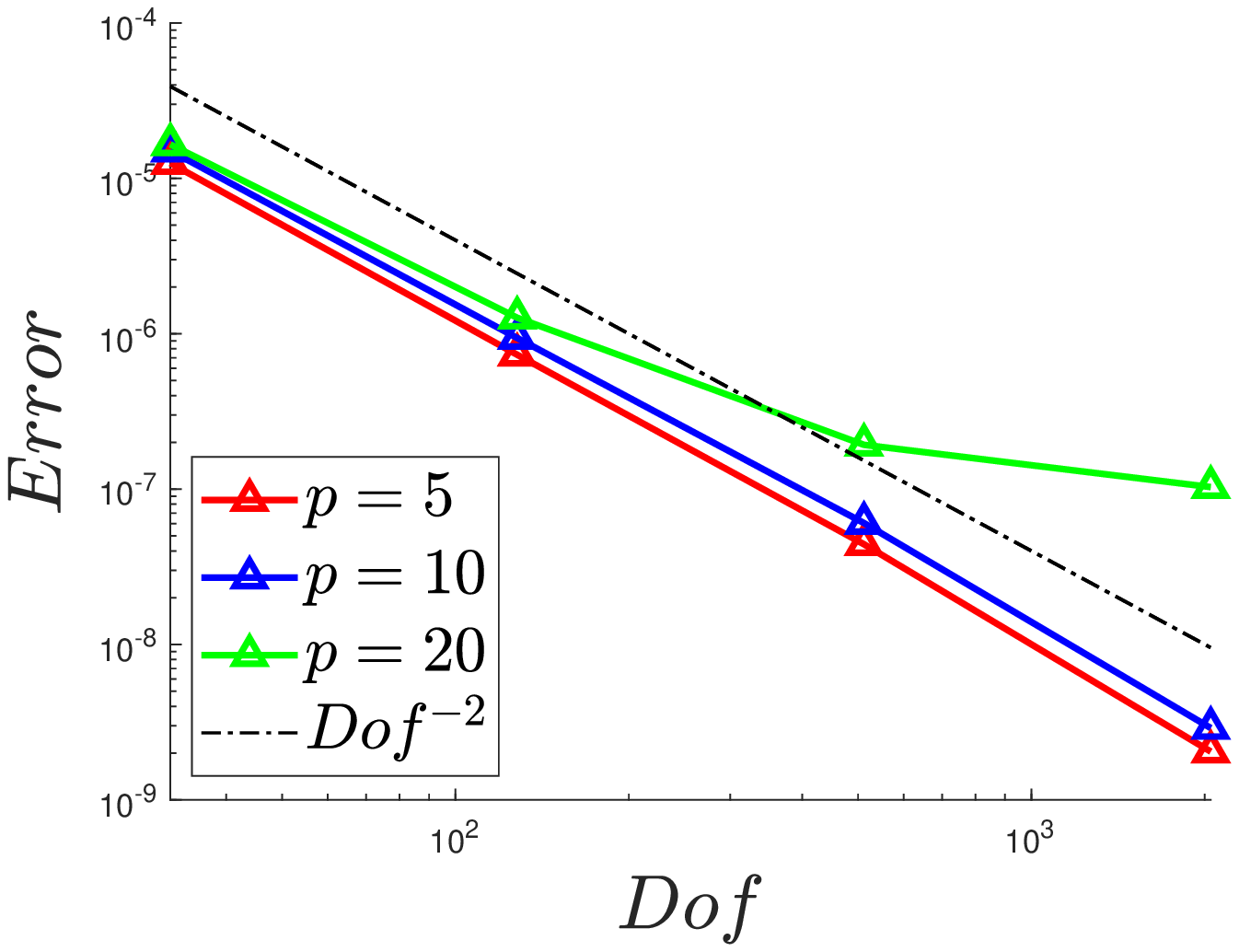}}
	\caption{Homogenization error for the preconditioned gradient descent method and the regularized model $\J_{\epsilon}$ with $\epsilon_-^{p-2}=1e-6$, $f=\sin(\pi x)\sin(\pi y)$, and $p=5,10,20$, respectively. $u^{(0)}$ is $0.5u_h$, where $u_h$ is the fine mesh reference solution.}
	\label{fig:coarseerrorpgd}
\end{figure}

\begin{figure}[H]
	\centering
	\subfigure[$H^1$ error]{\includegraphics[width=0.3\textwidth]{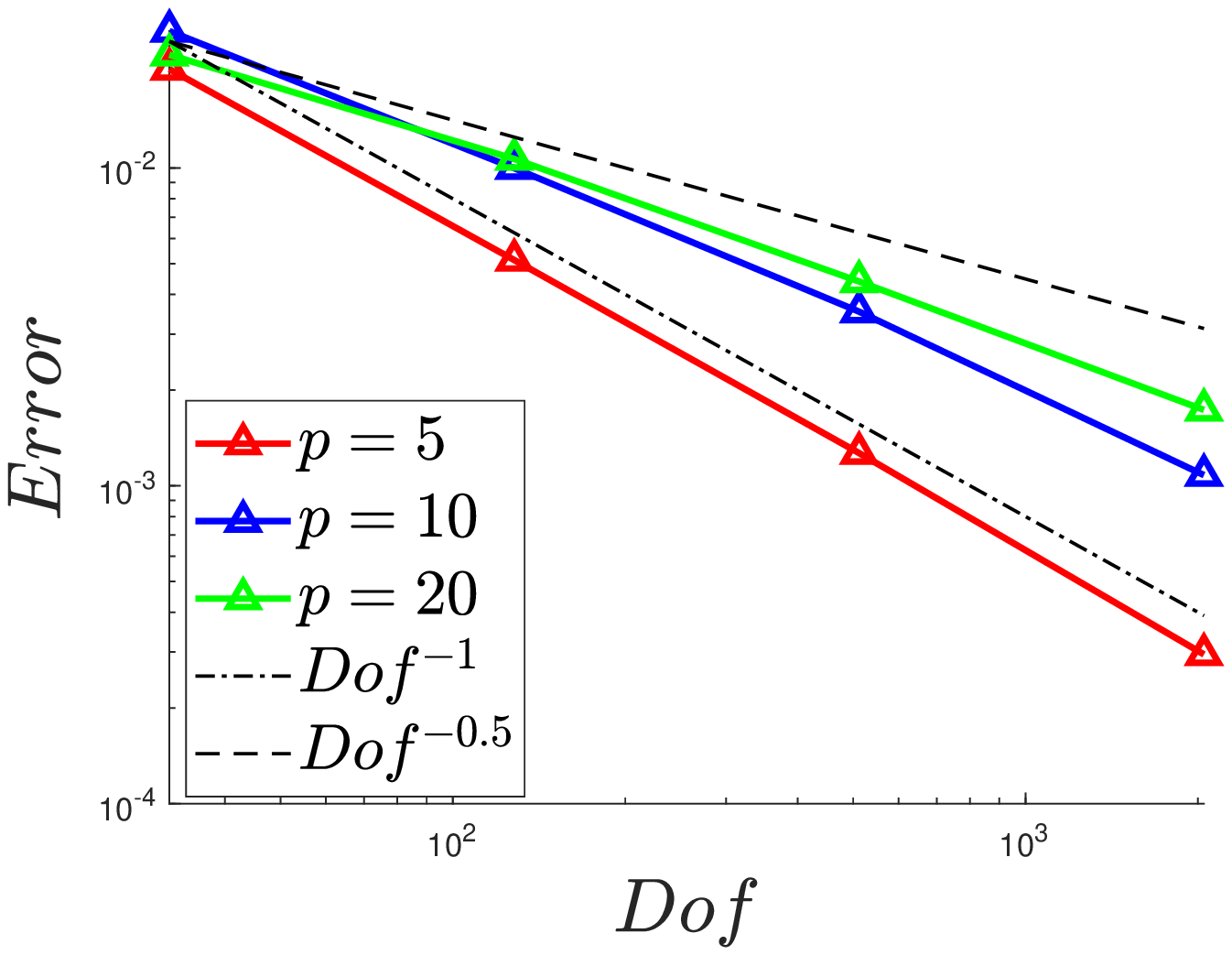}}
	\subfigure[$W^{1,p}$ error]{\includegraphics[width=0.3\textwidth]{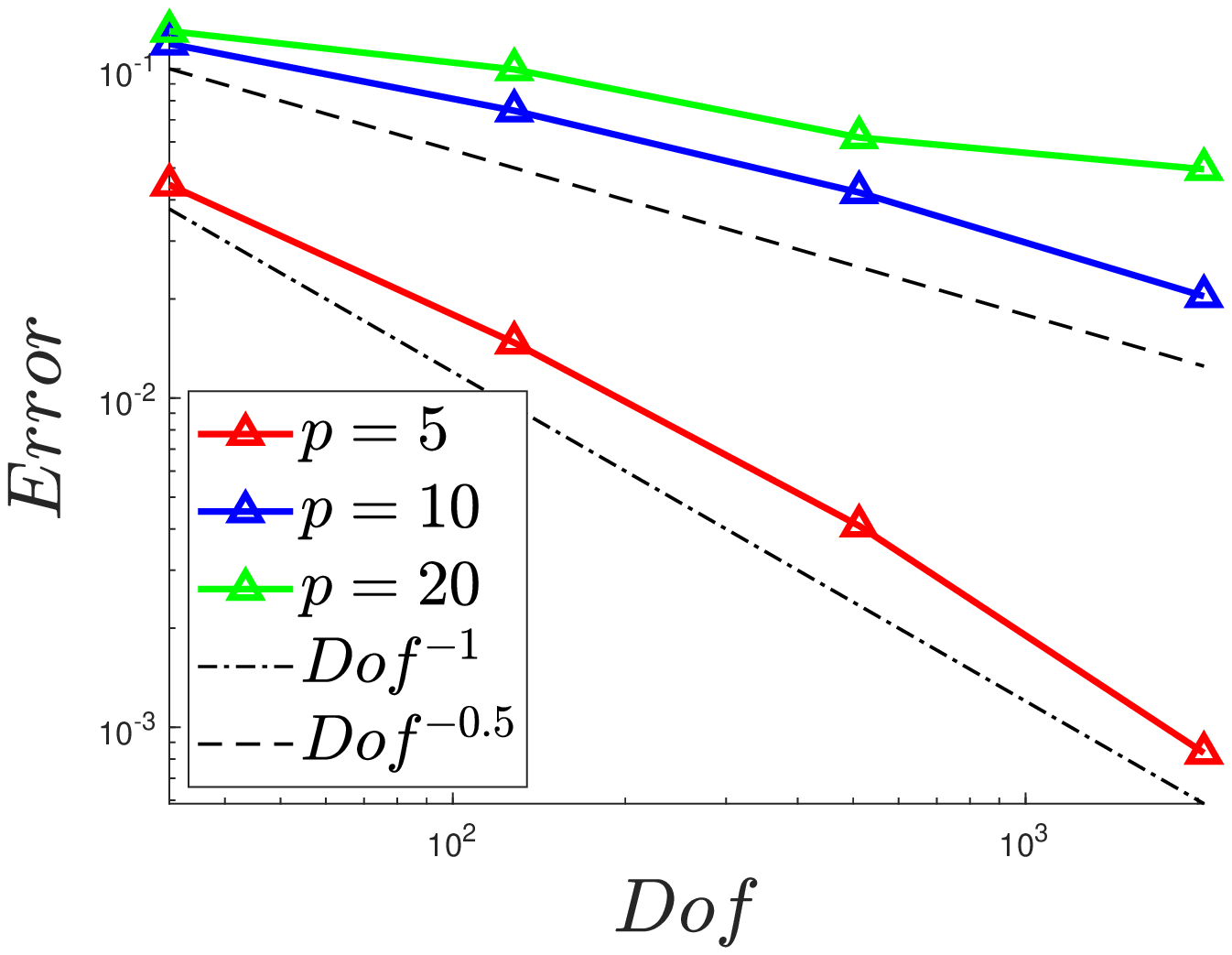}}
	\subfigure[energy error]{\includegraphics[width=0.3\textwidth]{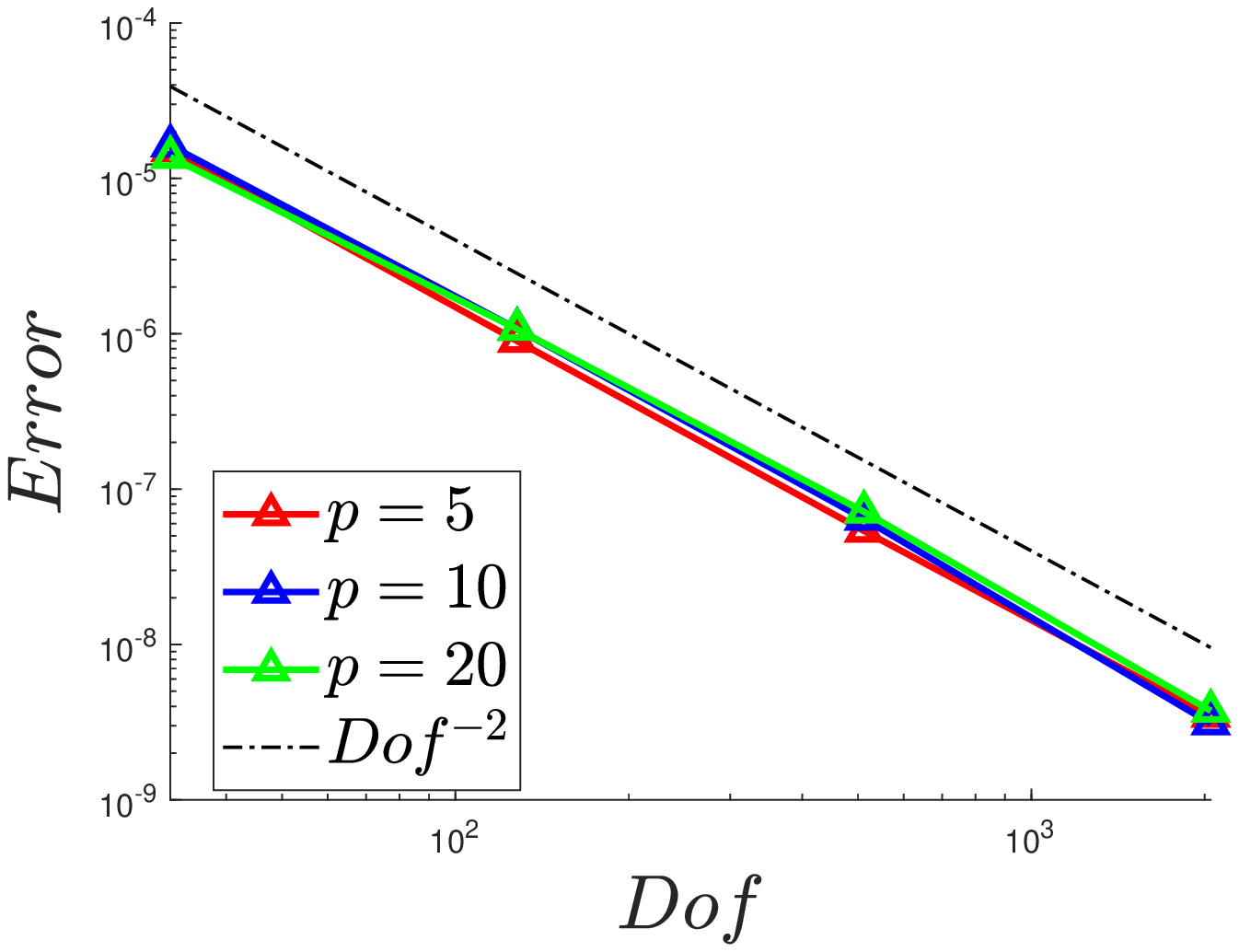}}
	\caption{Homogenization error for the Newton's method and the regularized model $\J_{\epsilon}$ with $\epsilon_-^{p-2}=1e-6$, $f=\sin(\pi x)\sin(\pi y)$, and $p=5,10,20$, respectively. $u^{(0)}$ is the solution of $-\nabla \cdot \kappa(x) \nabla u^{(0)} = f$.}
	\label{fig:coarseerrornewton}
\end{figure}

%=============================================
\subsection{SPE 10}
\label{sec:numerics:spe10}
%=============================================

In this section, we show the performance of the iterated numerical homogenization method for a more challenging example, the SPE10 benchmark problem, which is the set of industry benchmark problems from the Society of Petroleum Engineers (SPE). The coefficients $\kappa(x,y,z)$ are piecewise constant in the domain $[0,220]\times [0,60]\times [0,85]$. We select $\kappa(x, y, 39)$ as the coefficient of the two dimensional problem, with contrast $1.74956\times 10^{7}$. We rescale the domain to $\Omega=[0,2.2]\times [0,0.6]$. We take $f=\sin((1/2.2)\pi x)\sin((1/0.6)\pi y)$, and $p=20$. In the numerical experiment, we choose $H=1/10$, $1/20$, $1/40$, respectively. The fine mesh size is fixed as $h=1/160$.  

\begin{figure}[H]
	\centering
	\subfigure{\includegraphics[width=0.45\textwidth]{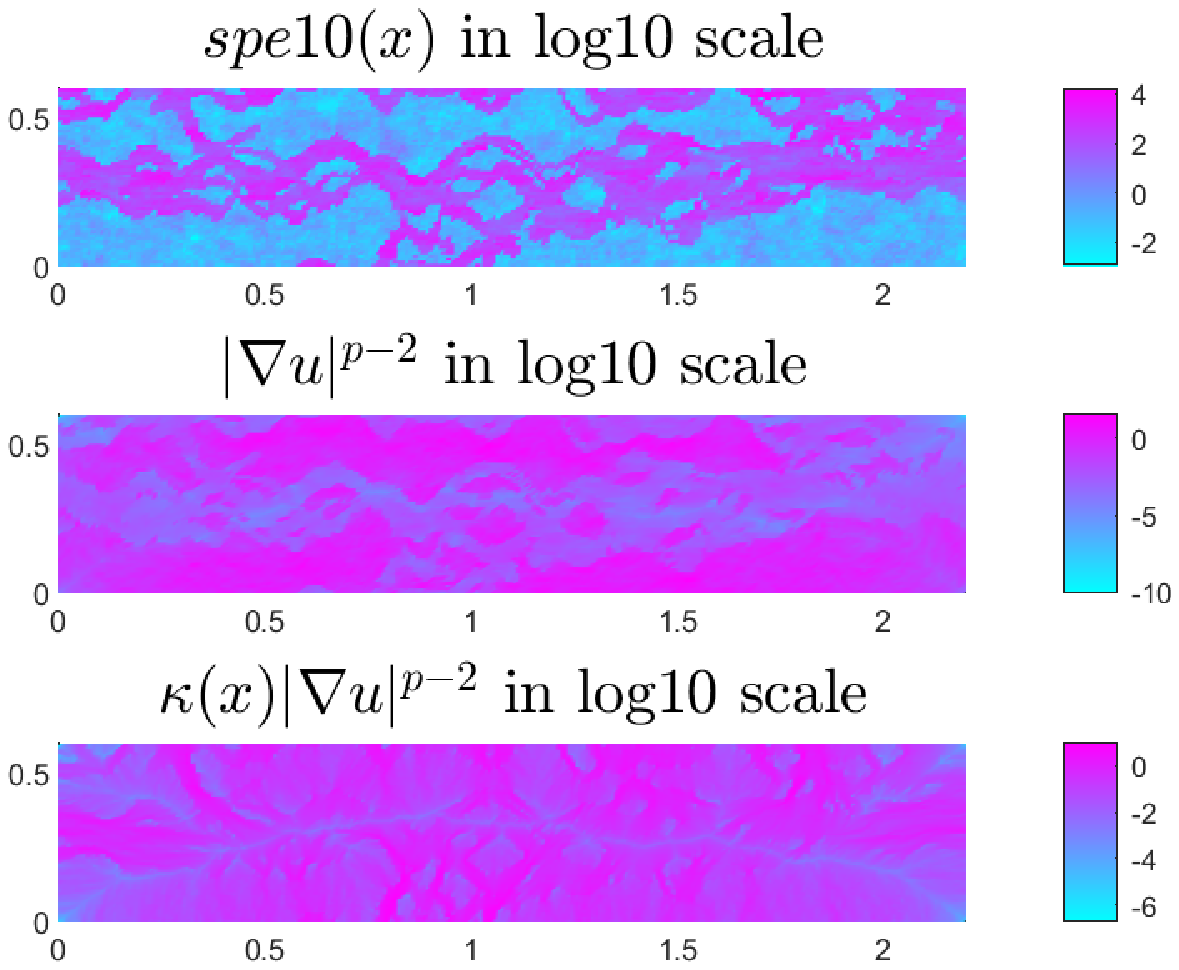}}
	\subfigure{\includegraphics[width=0.45\textwidth]{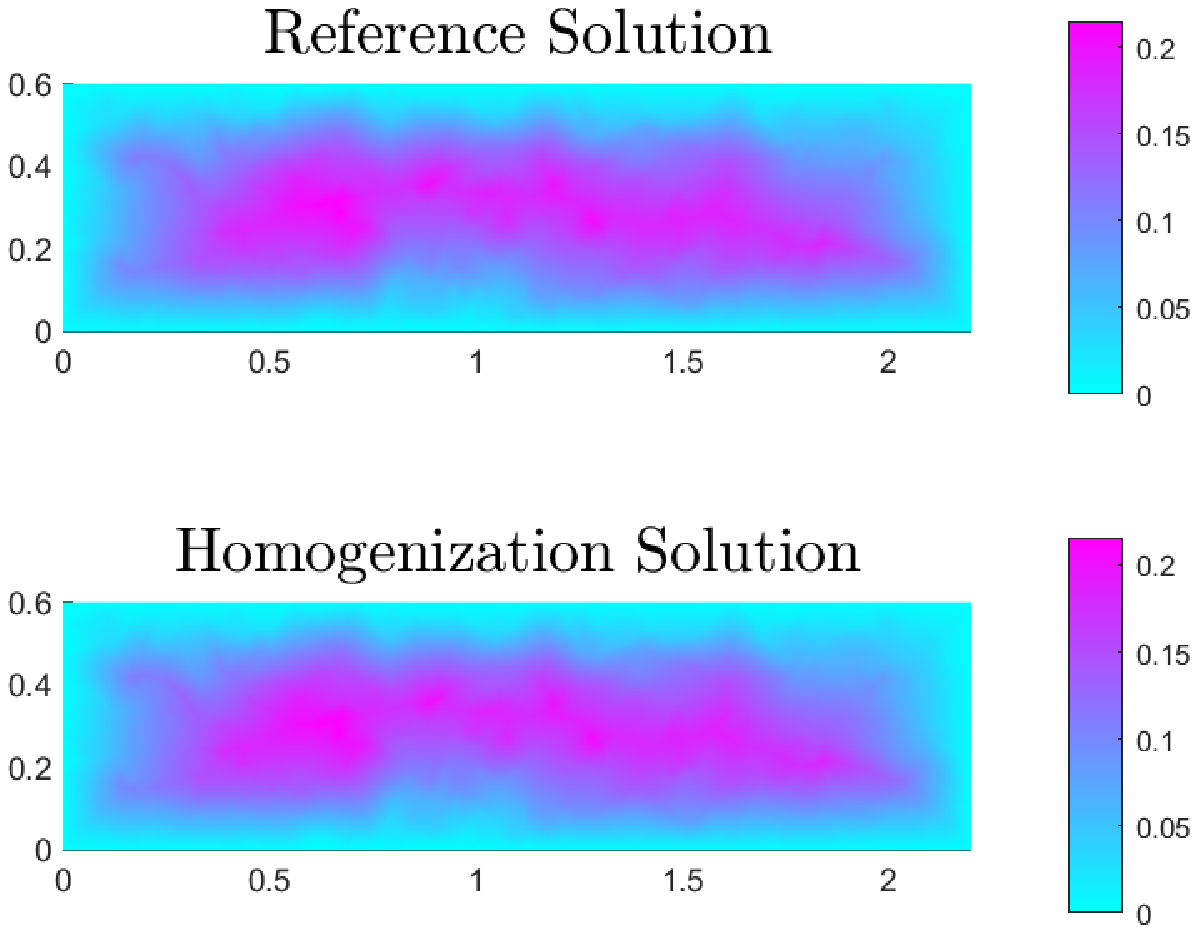}}
	\caption{SPE10 coefficients $\kappa(x)$, $|\nabla u|^{p-2}$, and $\kappa(x)|\nabla u|^{p-2}$ for the fine mesh reference solution $u$. $\min (|\nabla u|^{p-2})=0$,  $\max (|\nabla u|^{p-2})=31.8992$, $\max(\kappa|\nabla u|^{p-2})=8.7326$. Here we let $\epsilon_{-}=10^{-5}$, so the contrast of coefficients$ a[\un](x)$ is $10^6$ approximately. }
	\label{fig:spe10coef}
\end{figure}

\begin{figure}[H]
\centering
	\subfigure[$H^1$ error]{
		\includegraphics[width=0.3\textwidth]{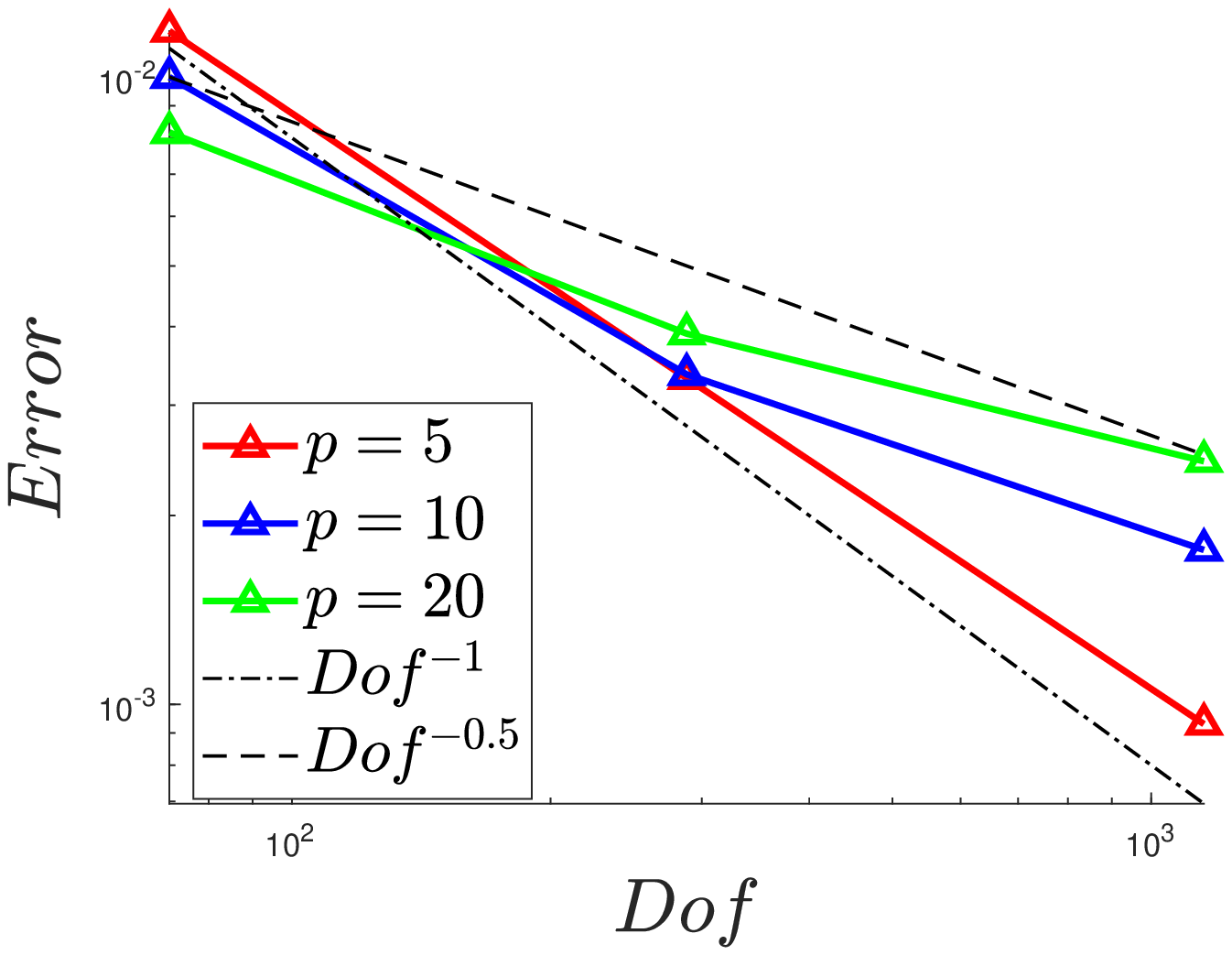}}
	\subfigure[$W^{1,p}$ error]{
	\includegraphics[width=0.3\textwidth]{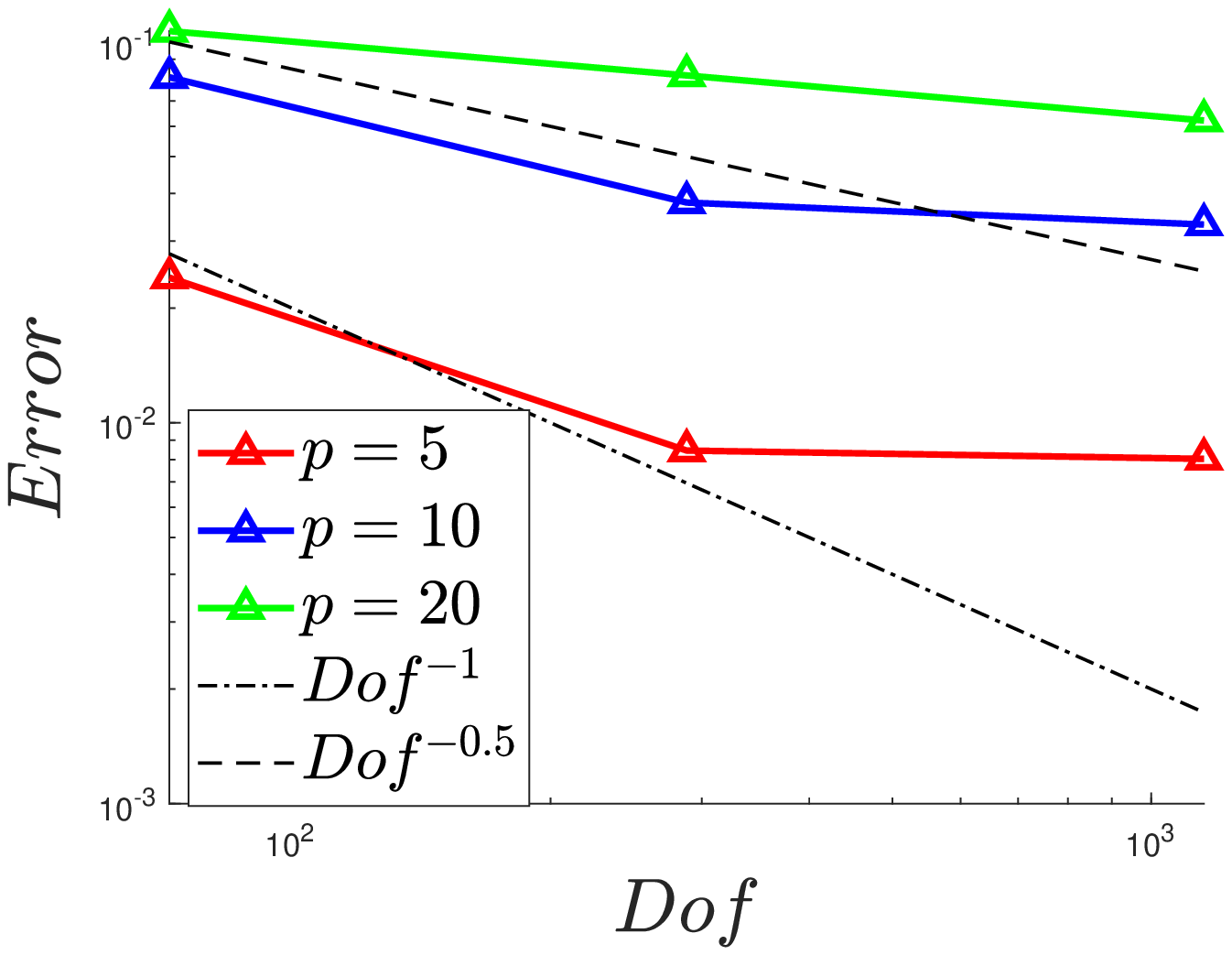}}
	\subfigure[energy error]{
	\includegraphics[width=0.3\textwidth]{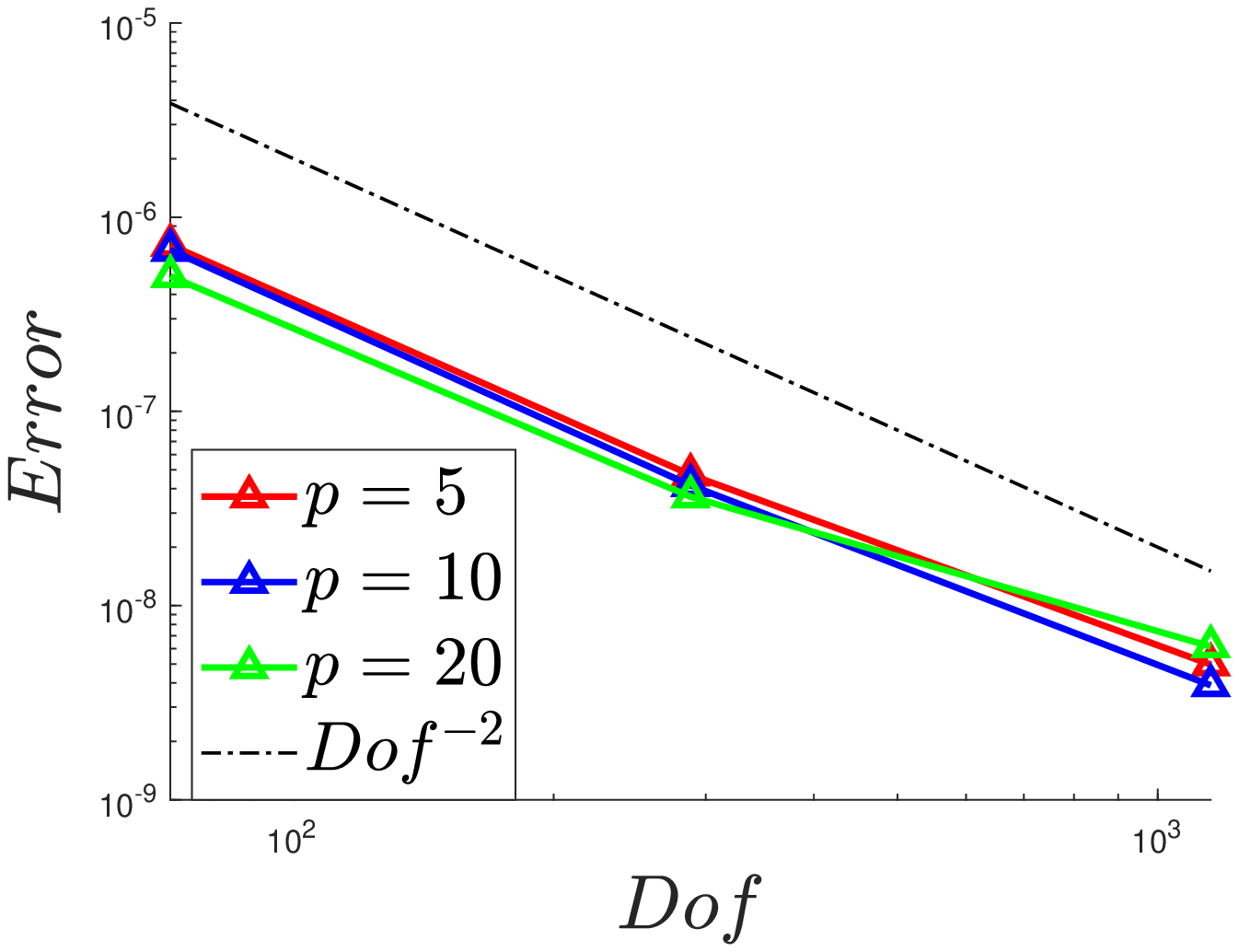}}
	\caption{Homogenization error of $u_{H}$ (obtained by Newton's method) for $p=5,10,20$.}
	\label{fig:spe10err}
\end{figure}

We illustrate the above SPE10 coefficient $\kappa$, also $|\nabla u|^{p-2}$ and $\kappa|\nabla u|^{p-2}$ in Figure \ref{fig:spe10coef}(a). We show the reference solution and homogenized solution in Figure \ref{fig:spe10coef}(b), which visually match with each other. 

We further show the homogenization error in Figure \ref{fig:spe10err}. Although the energy error has similar decay, but the $H^1$ and $W^{1,p}$ errors are not as good as the mstrig example in Section \ref{sec:numerics:mstrig}.

%=============================================
\subsection{Sparse Updating}
\label{sec:numerics:sparseupdate}
%=============================================

In the iterated homogenization method, we need to update the coarse space $\Vcn$ in every iteration, which may contribute to a significant part of the computational cost. However, apart from the first several steps of the iteration, as $\un$ converges to the true solution $u$, the linearized operator $A[\un]$ also converges to $A[u]$. Therefore, we propose a sparse updating scheme for the efficient construction of the coarse space $\Vcn$ in the iterated numerical homogenization.

Here we propose a simple indicator to switch on/off basis update, and we refer to \cite{hellman2019numerical} for more rigorous analysis.
We recall that a single basis of $\Vcn$ only depends locally on the coefficient of $A[\un]$ through $\kappa(x)$ and $|\nabla \un|$. Denote by $\phi_i^{(n,n+1)}(\tau)$ the corresponding basis of the operator $A[\un+\tau (u^{(n+1)}- \un)]$, 
and by $\bB^{(n,n+1)}(\tau)$ their stiffness matrix. The components of $\bB^{(n,n+1)}(\tau)$  write  $\bB_{i,j}^{(n,n+1)}(\tau) = A[\un+\tau (u^{(n+1)}- \un)](\phi_i^{(n,n+1)}(\tau),\phi_j^{(n,n+1)}(\tau))$. Let
\begin{displaymath}		
\I_i^{(n)}:=\frac{d \bB_{i,i}^{(n,n+1)}(\tau)}{d\tau}\Big|_{\tau=0}=A[(u^{(n+1)}- \un)](\phi_i^{n},\phi_i^{n}),		
\end{displaymath}
we may take $\I_i^{(n)}$ as an indicator for basis updating. Using indicator $\I_i^{(n)}$ for $n\geqs 1$, we can reformulate the \textbf{STEP 2} of Algorithms \ref{alg:inh} and \ref{alg:inhrrls} as below,

\begin{algorithmic}
	\STATE{\textbf{STEP 2}:
		\IF{$n=0$} 
		\STATE Construct a coarse space $\V{0}{H}:= {\rm span}\{\phin{0}{i},i=1,...,N_{H}\}$, and each $\phin{0}{i}$ is obtained by solving a constrained minimization problem \eqref{eqn:variational} associated with the quadratic form $A[u^{(0)}]$. 
		\ELSE
		\FOR{$i=1..N_H$}
		\STATE Calculate $\I_i^{(n)}$.
		\IF{$\I_i^{(n)} \geqs \delta_{\I}$}
		\STATE Calculate new basis $\phi_{i}^{(n)}$ by solving  \eqref{eqn:variational} associated with the quadratic form $A[\un]$.
		\ELSE 
		\STATE $\phi_{i}^{(n)} := \phi_{i}^{(n-1)}$.
		\ENDIF 
		\ENDFOR
		\STATE $\V{n}{H}:= {\rm span}\{\phin{n}{i},i=1,...,N_{H}\}$.
		\ENDIF} \\
\end{algorithmic}

% We do full update for the first several steps and sparse update for the last steps. For each sparse update, we update the largest 1/3 proportion of the total basis in terms of the values of indicator $\int(a_{n+1}-a_n)\nabla\phi_i^{n}\cdot\nabla\phi_i^{n}$. The reason we keep full updating for the first several steps is empirical. The numerical results often show that during the last several iteration, the basis do not change much and the first several steps usually are essential for the final convergence error. 

\begin{table}[H]
    \parbox{.45\linewidth}{
	\caption{Efficiency Improvement for $p=15$, $\kappa(x)= spe10$}
	\begin{tabular}{ccc}
		\hline
		& Efficiency  & Accuracy($H^1$)  \\
		\hline
		Full Updating & 100\%   & 0.67e-2 \\
		\hline
		$\delta_{\I}=4e-8$ & 80.94\% & 0.73e-2 \\
		\hline
		$\delta_{\I}=2e-8$ & 80.23\% & 0.79e-2 \\
		\hline
		$\delta_{\I}=1e-8$ & 68.68\% & 1.08e-2 	\\
	\end{tabular}
	\label{tab:spe10p15}
    }
    \hfill
    \parbox{.45\linewidth}{
    \caption{Efficiency Improvement for $p=20$, $\kappa(x)= spe10$}
	\begin{tabular}{ccc}
		\hline
		& Efficiency  & Accuracy($H^1$)  \\
		\hline
		Full Updating & 100\%   & 0.44e-2 \\
		\hline
		$\delta_{\I}=4e-9$ & 84.89\% & 0.48e-2 	\\
		\hline
		$\delta_{\I}=6e-9$ & 65.42\% & 0.75e-2 \\
		\hline 
		$\delta_{\I}=1e-8$ & 55.75\% & 0.82e-2 \\	
	\end{tabular}}
	\label{tab:spe10p20}
\end{table}
% \begin{table}[H]
% 	\caption{Efficiency Improvement for $p=20$, $\kappa(x)= spe10$}
	
% 	\label{tab:spe10p20}
% \end{table}

In Table \ref{tab:spe10p15} and Table \ref{tab:spe10p20}, we show the total number of basis updating (Efficiency) versus the parameter $\delta_\I$, for spe10 example with $p=15,20$, respectively. It seems the sparse updating is more effective for more nonlinear problem ($p=20$).

%=============================================
\section{Conclusion}\label{sec:conclusion}
%=============================================

In this paper we propose the iterated numerical homogenization method for multi-scale elliptic equations with monotone nonlinearity. We combine the iterative methods inspired by the quasi-norm based approach and the numerical homogenization for each linearized operator in the nonlinear iteration. The residual regularized line search is proposed to ensure the convergence of the iterated numerical homogenization up to coarse resolution. We offer a number of representative numerical examples to illustrate and validate the proposed method.

We plan to study more complicated nonlinear multiscale problems, for example, problems with nonconvex flux, and nonlinear heterogeneous elasticity. We also plan to study the use of nonlocal upscaling techniques \cite{chung2018non,chung2018nonlinear}. 

%=============================================
\section{Acknowledgements}
%=============================================

X.L and L.Z  are partially supported by the National Natural Science Foundation of China (NSFC 11871339, 11861131004). 
E.C. is partially supported by the Hong Kong RGC General Research Fund (Project numbers 14304719 and 14302018).

\bibliographystyle{plain}
\bibliography{references}
\appendix

\section{Appendix}

\subsection{Property of N-functions}
\label{sec:app:n-function}

For a given N-function $\vphi$, we define the N-function $\psi$ by 
$\frac{\psi'(t)}{t}:=(\frac{\vphi'(t)}{t})^\frac12$, and $\bV(\bxi):=(\nabla \bPsi)(\bxi) = \psi'(|\bxi|)\frac{\xi}{|\bxi|}$. We collect the following results from \cite{Diening:2007,Diening:2008b}. 

\begin{lemma}
Let $a(x, \bxi) = \kappa(x) \vphi'(|\bxi|)\frac{\bxi}{|\bxi|}$, we have 
\begin{equation}
\begin{aligned} 
	(a(x, \bxi)-a(x, \bzeta))\cdot (\bxi - \bzeta) & \sim \vphi_{|\bxi|}(|\bxi-\bzeta|) \\
	        & \sim |\bV(\bxi) - \bV(\bzeta)|^2 \\
			& \sim \vphi''(|\bxi|+|\bzeta|)|\bxi-\bzeta|^2.
\end{aligned}
\end{equation}
\label{lem:nfunction}
\end{lemma}

\begin{lemma}[Young type inequality]
%\cite[Lemma 6.8]{Diening:2007}\cite{Diening:2008b}
Let $\vphi$ be an N-function with $\Delta_2({\vphi, \vphi^*})< \infty$.  
For all $\delta >0$, we have the following Young's inequality, where $C_\delta$ only depends on $\Delta_2({\vphi, \vphi^*})$, such that for all $t$, $s$, $a\geqs 0$, 
\begin{equation}
    t\vphi'_a(s)+s\vphi'_a(t)\leqs\delta \vphi_a(t) + c_\delta \vphi_a(s).
\end{equation}
\label{lem:young}
\end{lemma}

\begin{lemma}
%\cite[Lemma 6.2]{Diening:2007}
Let the N-function $\vphi$ satisfy Assumption \ref{asm:vphi}, then uniformly in $s$, $t\in\mathbb{R}$, we have 
    \begin{equation}
        \begin{aligned}
            \vphi''(|s|+|t|)|s-t| & \sim \vphi'_{|s|}(|s-t|), \\
            \vphi''(|s|+|t|)|s-t|^2 & \sim \vphi_{|s|}(|s-t|).
        \end{aligned}
    \end{equation}
    \label{lem:vphi2}
\end{lemma}

\begin{lemma}
%\cite[Lemma 6.4]{Diening:2007}
Let the N-function $\vphi$ satisfy Assumption \ref{asm:vphi}, then uniformly in $s$, $t\in \mathbb{R}$, we have 
\begin{equation}
    \vphi_{|s|}(|s-t|)\sim \vphi_{|t|}(|s-t|).
\end{equation}
\label{lem:ab}
\end{lemma}

The following inequality is the consequence of the Young type inequality, Property \ref{prp:vphi}, and Lemma \ref{lem:vphi2}, 
\begin{equation}
    \begin{aligned}
        (a(x, \nabla u) - a(x, \nabla v))(\nabla w-\nabla v) & \leqs 
        |(a(x, \nabla u) - a(x, \nabla v))|\cdot|(\nabla w-\nabla v)|\\
                    & \leqs c \vphi''(|\nabla u|+\nabla v|)|\nabla u -\nabla v|\cdot |\nabla w - \nabla v|\\
                    & \leqs c\vphi'_{|\nabla u|}(|\nabla u - \nabla v|)|\nabla w - \nabla v|\\
        &\leqs \delta  \vphi_{|\nabla u|}(|\nabla u - \nabla v|) + c_\delta  \vphi_{|\nabla u|}(|\nabla w - \nabla v|).
    \end{aligned}
    \label{eqn:young}
\end{equation}

%=============================================
\subsection{Orlicz Space}
\label{sec:app:Orlicz}
%=============================================

We give some characterizations of the Orlicz space and Sobolev Orlicz space according to \cite{krasnosel,rao2002applications}. 
\begin{theorem}
    Let $\varphi$ be an N-function with $\Delta_2(\vphi, \vphi^*)<\infty$, the following statements are true:
    \begin{enumerate}[(1)]
        \item 
         $L^{\varphi}(\Omega)$ is normed with $\|w\|_{L^{\vphi}(\Omega)}:=\inf \left\{\lambda>0: \int_{\Omega} \varphi\left(\frac{|w|}{\lambda}\right) d x \leqs 1\right\}$. We note that this norm coincides with $L^p$ norm for $\vphi = t^p/p$. %[see Chapter II, \S 9.7]
        \item 
          A sequence of functions $\left\{w_{n}\right\}$ in $L^{\varphi}(\Omega)$ converges to $w \in L^{\varphi}(\Omega)$ with respect to $\|\cdot\|_{L^{\varphi}(\Omega)}$, if and only if $\int_{\Omega} \varphi\left(\left|w_{n}-w\right|\right) d x \stackrel{n \rightarrow \infty}{\longrightarrow} 0,$ %see [Chapter II, \S Theorem 9.4].  
          \item 
          $L^{\vphi}(\Omega)$ is complete, separable and reflexive.% see [Chapter II, \S 9,2. 10.3, 9.5, 14]
    \end{enumerate}
\end{theorem}

% We say $w \sim \tilde{w}: \quad \Longleftrightarrow \quad w=\tilde{w} \text { almost everywhere}$.
% Then, we define the Orlicz class as $L^{\varphi}:=\tilde{\mathcal{L}}^{\varphi} / \sim$. 

\begin{definition}
For an N-function $\vphi$ satisfying the $\Delta_{2}$ condition we define the Orlicz Sobolev space via
$$
W^{1, \varphi}(\Omega):=\left\{w \in L^{\varphi}(\Omega): \forall i \in\{1, \ldots, d\}, \frac{\partial}{\partial x_{i}} w \in L^{\vphi}(\Omega)\right\}.
$$
where $\frac{\partial}{\partial x_{i}} w$ denotes the weak derivative in the $i$-th direction.
\end{definition}

\begin{definition}
	For an N-function $\varphi$, we define its Simonenko indices via
	$$
	p^{-}:=\inf _{t>0} \frac{t \varphi^{\prime}(t)}{\varphi(t)}<\sup _{t>0} \frac{t \varphi^{\prime}(t)}{\varphi(t)}=: p^{+}.
	$$
	If $p^{+}<\infty$, $\varphi$ satisfies the $\Delta_{2}$ condition.
\end{definition}

\begin{theorem}
	Let $\varphi$ be an N-function with $p^{-}, p^{+} \in(1, \infty) .$ Then, the following statements are true:
	\begin{enumerate}[(1)]
	\item $W^{1, \varphi}(\Omega)$ is normed with $\|w\|_{W^{1, \varphi(\Omega)}}:=\|w\|_{L^{\varphi}(\Omega)}+\sum_{i=1}^{d}\left\|\frac{\partial}{\partial x_{i}} w\right\|_{L^{\varphi}(\Omega)}$.
	\item A sequence of functions $\left\{w_{n}\right\}$ in $W^{1, \varphi}(\Omega)$ converges to $w \in W^{1, \varphi}(\Omega)$ with respect to $\|\cdot\|_{W^{1, \varphi(\Omega)}}$, if and only if $\int_{\Omega} \varphi\left(\left|w_{n}-w\right|\right) d x \stackrel{n \rightarrow \infty}{\longrightarrow} 0$, and for all
	$i \in\{1, \ldots, d\}$ also $\int_{\Omega} \varphi\left(\left|\frac{\partial}{\partial x_{i}}\left(w_{n}-w\right)\right|\right) d x \stackrel{n \rightarrow \infty}{\longrightarrow} 0$.
	\item $W^{1, \varphi}(\Omega)$ is complete, separable and reflexive.
\end{enumerate}
\label{thm:orlicz}
\end{theorem}

\begin{definition}
We define the homogeneous Orlicz Sobolev space as $W_0^{1, \varphi}(\Omega):=\overline{C_{0}^{\infty}(\Omega)}^{\|\cdot\|_{W^{1, \varphi}(\Omega)}}
$. Note that Poincar\'e's Inequality also holds for the homogeneous Orlicz Sobolev space.
\end{definition}

\begin{proposition}
	\label{lem:variation}
	Under Assumption \ref{asm:vphi}, $\vphi''(t)$ is a non-decreasing function, which implies that the Simonenko index of $\vphi$, $p^{-}\geqs 2$. Furthermore, $\J$ is differentiable and second order Gateaux differentiable. Its derivative and second order derivative are given in \eqref{eqn:firstvar} and \eqref{eqn:secondvar}, and well-defined.
% 	\begin{proof}
% 		\begin{equation}
% 		\begin{aligned}
% 		\J(u+tv)-\J(u)- t \left(\int_{\Omega}\kappa(x)\frac{\vphi'(|\nabla u|)}{|\nabla u|}  \nabla u\cdot \nabla v -\int_{\Omega} f v\right)= &  \int_{\Omega}\int_0^t (a(x,\nabla (u+sv) )-a(x,\nabla u ))\cdot s\nabla v \, \mathrm{ds}\dx \\
% 		\leqs & C \int_{\Omega}\int_0^t\vphi_{|\nabla u|}(s|\nabla v|)\mathrm{ds}\dx
% 		\end{aligned}
% 		\end{equation}
% 		It suffices to show $\int_{\Omega}\vphi_{|\nabla u|}(t|\nabla v|)$ is $o(t \int_{\Omega}\vphi(t|\nabla v|)) $, if  Simonenko-indices $p^{-}$ of $\vphi_{|\nabla u|}$ is greater than 1 by applying Lemma \ref{lem:Growth}. We can verify $p^{-} \geq2$ under the Assumption \ref{asm:vphi}, $\vphi''(t)$ is a non-decreasing function, in particular $\vphi(t)=t^p/p$ with $p>2$. Therefore, $\J$ is differentiable and the derivative is given as in \eqref{eqn:firstvar}. Second order Gateaux derivative can be obtained similarly. 
% 	\end{proof}
\end{proposition}

% \subsection{p-Laplacian}

% Throughout the paper, we equip $W^{1,p}(\Omega)$ with the energy norm 
% $$
% \|u\|_{{1,p}(\Omega)}=\left(\int_{\Omega} \kappa(x)|\nabla u|^{p}dx \right)^{1/p}.
% $$
% \begin{proposition}
% 	By simply integrate over domain $\Omega$, we have the strong monotonicity and continuity of the operator $A$, which are the following properties:
% 	\begin{equation}	
% 	\begin{aligned}
% 	&\forall u, v \in V, \quad(A u-A v)(u-v) \geqslant m \|u-v\|_{1,p}^p \\
% 	& \|Au-Av\|_*\leqslant M \|u-v\|_{1,p}(\|u\|_{1,p}+\|v\|_{1,p})^{p-2}
% 	\end{aligned}
% 	\label{eq:monotonicity&continuity}
% 	\end{equation}
% \end{proposition}
% The well-posedness of ($\mathcal{P}$) can be simply obtained by applying  theorem \ref{thm:wellpose4monotone} to operator $A$.

% \begin{proposition}
% 	With Application of Holder inequality on  \eqref{eq:first derivative} and \eqref{eq:second derivative} we have following bounds:
	
% 	\begin{equation}
% 	|\langle A u, v\rangle| \leqslant C_{1}\|u\|_{1, p}^{p-1}\|v\|_{1, p}+\|f\|_{-1, q}\|v\|_{1, p},
% 	\end{equation}
% 	\begin{equation}
% 	|\langle A^{\prime}(u) v, w\rangle| \leqslant C_{2}\|u\|_{1, p}^{p-2}\|v\|_{1, p}\|w\|_{1, p}.
% 	\label{eq:continuity}
% 	\end{equation}
% \end{proposition}

%=============================================
\subsection{Proof of Lemma \ref{lem:bdist}}
\label{sec:app:prf:bdist}
%=============================================

\begin{lemma}\label{lem:uv}
	For any $u$, $v$, and $s\in[0,1]$, we have
	\begin{equation}
	s/2(|\nabla v|+|\nabla (u-v)|)\leqs |\nabla(v+s(u-v))|+|\nabla v|\leqs 2(|\nabla v|+|\nabla (u-v)|).
	\end{equation}
	In particular, for $s=1$, we have
	\begin{equation}
	1/2(|\nabla v|+|\nabla (u-v)|)\leqs |\nabla v|+|\nabla u|\leqs 2(|\nabla v|+|\nabla (u-v)|).
	\end{equation}
\end{lemma}

\begin{proof}
	Assumption \ref{asm:vphi} implies that $\vphi(t) \sim t\vphi'(t) \sim t^2\vphi''(t)$, and by the $\Delta_2(\vphi)$ condition, we also have $\vphi(2t) \leqs c \vphi(t)$. Therefore, it holds true that $4t^2 \vphi''(2t)\leqs c t^2\vphi''(t)$ uniformly for $t\geqs 0$.
    Denoting $w=u-v$, we have
	\begin{equation}
		\begin{aligned}
			\bdist{u}{v}&=\J(u)-\J(v)-\J^{\prime}(v)(u-v)\\
					  &=\int_{0}^{1}\J^{\prime}(v+ sw)w-\J^{\prime}(v)w \ds \\
					  &=\int_{0}^{1} \frac{1}{s} (\J^{\prime}(v+s w)-\J^{\prime}(v))(s w)\ds \\
					  &\geqs c\int_{0}^1 s\int \kappa(x)\vphi''(|\nabla v+s\nabla w|+|\nabla v|)|\nabla w|^2\dx \ds \\
					  &\geqs c\int \kappa(x)\vphi''(|\nabla v| + |\nabla w|)|\nabla u-\nabla v|^2\dx,
		\end{aligned}.
	\end{equation}
	and 
	\begin{equation}
		\begin{aligned}
			\bdist{u}{v}&=\J(u)-\J(v)-\J^{\prime}(v)(u-v)\\
					  &=\int_{0}^{1}\J^{\prime}(v+ sw)w-\J^{\prime}(v)w \ds \\
					  &\leqs\int_{0}^{1} \int \kappa(x)\vphi''(|\nabla v+s\nabla w|+|\nabla v|)s|\nabla w|^2\ds \\
					  &\leqs C\int \kappa(x)\vphi''(|\nabla v|+|\nabla w|)|\nabla u-\nabla v|^2\dx.
		\end{aligned}.
	\end{equation}

\end{proof}

 %=============================================
 \subsection{Proof of Lemma \ref{lem:regularization}}
 \label{sec:app:regulairzation}
 %=============================================
 
 \begin{proof}
 It is clear that we regularize when $|\nabla u|$ vanishes or becomes large. 
 
 For $u \in W_{0}^{1, p}(\Omega)$, we define $\Omega_{k}(u):=\left\{x \in \Omega:|\nabla u(x)|>k\right\}$. By following inequalities
 $$
 	\left|\Omega_{k}(u)\right| :=\int_{\Omega_{k}(u)} 1  \leqs \frac{1}{k} \int_{\Omega_{k}(u)}|\nabla u|  
 	 \leqs \frac{1}{k}\left|\Omega_{k}(u)\right|^{\frac{1}{q}}\left(\int_{\Omega}|\nabla u|^{p} d\right)^{\frac{1}{p}}=\frac{\|u\|_{W_{0}^{1, p}(\Omega)}}{k}\left|\Omega_{k}(u)\right|^{\frac{p-1}{p}},
 $$
 we have $\left|\Omega_{k}(u)\right|   \leq\|u\|_{W_{0}^{1, p}(\Omega)}^{p} k^{-p}$, for all $u \in W_{0}^{1, p}(\Omega)$. Similarly, we have
 \begin{equation}
 	\left|\Omega_{k}(u)\right|   \leqs \frac{1}{k^p}  \left(\int_{ \Omega_{k}(u)}a_{\epsilon}(x, \nabla u) \cdot \nabla u\right), \quad \forall u \in W_{0}^{1,2}(\Omega),
 	\label{ineq:sigularity measure}
 \end{equation}
 where $a_{\epsilon}(x, \nabla u)=\vphi_{\epsilon_k}'(|\nabla u|)\nabla u/|\nabla u|$.
 Therefore, $	\left|\Omega_{k}(u)\right|  \rightarrow 0$, as $\epsilon_{+,k} \rightarrow \infty$.
		
 Step 1: \emph{The sequence $u_{\epsilon_k}$ is bounded in  $H^{1}_0$}.
 
 Taking $v=u_{\epsilon_k}$ in \begin{equation}
 	\int_{\Omega} a_{\epsilon_k}(x, \nabla u_{\epsilon_k}) \cdot \nabla v =(f,v), \quad \forall  v \in H_0^1,
 	\label{eqn:reginek}
 \end{equation}

 and combining the fact that 
 \begin{equation}
 	\begin{aligned}
 		\|\nabla u_{\epsilon_k}\|_{L^{2}\left(\Omega \backslash \Omega_{k}(u_{\epsilon_k})\right)^{d}} & \leq|\Omega|^{\frac{p-2}{2 p}}\|\nabla u_{\epsilon_k}\|_{L^{p}\left(\Omega \backslash \Omega_{k}(u_{\epsilon_k})\right)^{d}} \\
 		& \leqs |\Omega|^{\frac{p-2}{2 p}}\left(\int_{\Omega \backslash \Omega_{k}(u)}a_{\epsilon}(x, \nabla u_{\epsilon_k}) \cdot \nabla u_{\epsilon_k}\right)^{\frac{1}{p}},
 	\end{aligned}
 \end{equation} 
 and 
 \begin{equation}
 	\|\nabla u_{\epsilon_k}\|_{L^{2}\left(\Omega_{k}(u_{\epsilon_k})\right)^{N}} \leqs |\Omega|^{\frac{p-2}{2 p}}\left(\int_{  \Omega_{k}(u_{\epsilon_k})}a_{\epsilon}(x, \nabla u_{\epsilon_k}) \cdot \nabla u_{\epsilon_k}\right)^{\frac{1}{2}},
 \end{equation}

 we have
 \begin{equation}
 	\begin{aligned}
 		\int_{\Omega} a_{\epsilon_k}(x, \nabla u_{\epsilon_k}) \cdot \nabla  u_{\epsilon_k} &=(f, u_{\epsilon_k}) \leqs \|f\|_{L^2}\|\nabla u_{\epsilon_k}\|_{L^{2}} \\
 		& \leqs C \|f\|_{L^2}\left(\int_{\Omega} a_{\epsilon_k}(x, \nabla u_{\epsilon_k}) \cdot \nabla  u_{\epsilon_k}\right)^{1/2}.
 	\end{aligned}
 \end{equation}
 Therefore, $\int_{\Omega} a_{\epsilon_k}(x, \nabla u_{\epsilon_k})\cdot \nabla u_{\epsilon_k}$ (hence $\|u_{\epsilon_k}\|_{H_0^1}$) is also uniformly bounded. 
		
 Step 2 : \emph{(Lemma 5.1 in Casas \cite{casas2016approximation}) weak limit $u$ of $u_{\epsilon_k}$ belongs to $W^{1,p}$}.

 Since $\{u_{\epsilon_k}\}$ is uniformly bounded in $H^1_0$, we can extract a subsequence which converge weakly to $u\in H^1_0$. Without loss of generality, we assume the subsequence is $\{u_{\epsilon_k}\}$ itself.

 % To establish this, let us take a subsequence $\left\{u_{\epsilon_i}\right\}_{i\in \mathrm{N}}$  (here, $\epsilon_{-,i} \rightarrow 0$ and $\epsilon_{+,i} \rightarrow \infty$ as $i \rightarrow \infty $) and a function $u \in H_{0}^{1}(\Omega)$ such that ${u_{\epsilon_i}} \rightharpoonup u$ in $H_{0}^{1}(\Omega)$ as $i \rightarrow \infty$. 

 We fix an index $k \in \mathrm{N}$ and associate it with the set: $B_{k}:=\bigcup_{j=k}^{\infty} \Omega_{j} \left(u_{\epsilon_j}\right)$, where $\Omega_{j}\left(u_{\epsilon_j}\right):=\left\{x \in \Omega:\left|\nabla u_{\epsilon_j}(x)\right|>j\right\}$.

 By the estimate \eqref{ineq:sigularity measure} and uniformly boundedness of $\int_{\Omega} a_{\epsilon_k}(x, \nabla u_{\epsilon_k})\cdot \nabla u_{\epsilon_k}$, we see that
 $$
 		\left|B_{k}\right|  \leqs   C \sum_{j=k}^{\infty} \frac{1}{j^{p}}<+\infty,
 $$
 and, therefore $\lim _{k \rightarrow \infty}\left|B_{k}\right|=0$. We also have
 $$
 \begin{aligned}
 	\int_{\Omega \backslash B_{k}}\kappa(x)\left|\nabla u_{\epsilon_{j}}\right|^{p} d x 
 	& \leqs \int_{\Omega} a_{\epsilon_j}(x, \nabla u_{\epsilon_j})\cdot \nabla  u_{\epsilon_j}  \leqs C, \, \forall j\geqs k,
 \end{aligned}
 $$
 hence $\left\{\nabla u_{\epsilon_{j} }\right\}$ is bounded in $L^{p}\left(\Omega \backslash B_{k}\right)^{N}$. Since $\nabla u_{\epsilon_{j}} \rightharpoonup \nabla u$ in $L^{2}(\Omega)^{d}$, we have $\nabla u_{\epsilon_{j}} \rightharpoonup \nabla u$ in $L^{p}(\Omega\backslash{B_k})^{d}$.  Hence,
 \begin{equation}
 \begin{aligned}
 	\int_{\Omega} \frac{1}{p}\kappa(x)|\nabla u|^p & = \lim_{k\rightarrow \infty} \int_{\Omega \backslash B_{k}}\frac{1}{p}\kappa(x)|\nabla u|^p \leqs \lim _{k \rightarrow \infty} \liminf _{\substack{j \rightarrow \infty \\ j\geqs k}} \int_{\Omega \backslash B_{k}} \frac{1}{p}\kappa(x)\left|\nabla u_{\epsilon_{j}}\right|^{p}\\
 	& \leqs \liminf _{j \rightarrow \infty }\int_{\Omega\backslash \Omega_{j}} \kappa(x)\vphi_{\epsilon_{j}} (|\nabla u_{\epsilon_{j}}|) \leqs\liminf _{j \rightarrow \infty }\int_{\Omega} \kappa(x)\vphi_{\epsilon_{j}} (|\nabla u_{\epsilon_{j}}|). 
 \end{aligned}
 \end{equation}

 Moreover, 
 \begin{equation}
 	\J(u)=\int_{\Omega}  \frac{1}{p} \kappa(x) |\nabla u|^p - f u \leqs \liminf _{j \rightarrow \infty }\int_{\Omega}\kappa(x)\vphi_{\epsilon_{j}} (|\nabla u_{\epsilon_{j}}|) -  fu_{\epsilon_{j}} = \liminf _{j \rightarrow \infty }\J_{\epsilon}(u_{\epsilon_{j}}).
 \end{equation}

 Let $u^*$ be the minimizer of $\J$ in $W^{1,p}$, we have 
 \begin{equation}
 	\J(u^*) \leqs \J(u) \leqs \liminf _{j \rightarrow \infty }\J_{\epsilon}(u_{\epsilon_{j}}).
 	\label{ineq:upperlimit}
 \end{equation}
		
 Using the definition \eqref{eqn:regvphi} or \eqref{eqn:regvphi2} of $\vphi_{\epsilon}$, for any $u\in W_0^{1,p}$, we have
 \begin{displaymath}
 	\int_{\Omega} \kappa(x)\vphi_{\epsilon_{j}} (|\nabla u|) \leqs \int_{\Omega} \kappa(x) |\nabla u|^p + C \epsilon_{j-}^p.
 \end{displaymath}
 
 Thus, $\J_{\epsilon}(u_{\epsilon_{j}}) \leqs \J(u^*) + C \epsilon_{j-}^p$. Let $j\rightarrow \infty$,
 \begin{equation}
 \liminf _{j \rightarrow \infty }\J_{\epsilon}(u_{\epsilon_{j}}) \leqs \J(u^*).
 \label{ineq:lowerlimit}
 \end{equation}
 
 Combining \eqref{ineq:upperlimit} and \eqref{ineq:lowerlimit} we have
 \begin{equation}
 \J(u^*)  = \J(u) = \liminf _{j \rightarrow \infty, }\J_{\epsilon}(u_{\epsilon_{j}}),
 \end{equation}
 which implies $u_{\epsilon_{j}} \rightharpoonup  u^*$ in $H^1_0$ and  $\nabla u_{\epsilon_{j}} \rightharpoonup \nabla u^*$ in $L^{p}(\Omega\backslash{B_k})^{d}$. Furthermore, we have
 \begin{equation}
	\lim_{k \rightarrow \infty}\int_{\Omega}\kappa(x) \vphi_{\epsilon_{k}}(|\nabla u_{\epsilon_{k}}|) = \lim_{k \rightarrow \infty}\int_{\Omega\backslash \Omega_{k}}\kappa(x) \vphi_{\epsilon_{k}}(|\nabla u_{\epsilon_{k}}|) = \int_{\Omega}\kappa(x) |\nabla u|^p ,
\end{equation}
and,
\begin{align}
  	&\lim_{k \rightarrow \infty}\int_{\Omega}\kappa(x) \chi_{\Omega \backslash \Omega_{k}}|\nabla u_{\epsilon_k }|^p= \int_{\Omega}\kappa(x) |\nabla u|^p , \label{eqn:normCov}\\
  	&\lim_{k \rightarrow \infty}\int_{\Omega_k}\kappa(x) |\nabla u_{\epsilon_k }|^2 = 0,
  	\label{eqn:normCov2}
\end{align} 	
since $\vphi_{\epsilon_{k}}(|\nabla u_{\epsilon_{k}}|) = |\nabla u_{\epsilon_k }|^p $ for $x\in\Omega \backslash \Omega_{k}$. The weak convergence $\nabla u_{\epsilon_{j}} \rightharpoonup \nabla u$ in $L^{p}(\Omega\backslash{B_k})^{d}$ and the norm convergence in  \eqref{eqn:normCov} imply the strong convergence in $L^{p}(\Omega\backslash{B_k})^{d}$. Combining this and \eqref{eqn:normCov2} we also have strong convergence in $H^1_0(\Omega)$. 
 \end{proof}

% %=============================================
% \subsection{FEM error}
% \label{sec:app:fem}
% %=============================================
%
% For the weak form \eqref{eqn:weakform} and the finite element approximation \eqref{eqn:Galerkin}, we have the following C\'ea's type estimates for $p>1$ \cite{Diening:2007},  
% \begin{equation}
%   \|\bV(\nabla u) - \bV(\nabla u_h)\|_2^2\leqs C \min_{\psi\in \Vh}\|\bV(\nabla u) - \bV(\nabla \psi)\|_2^2.
% \end{equation}
%
% This appears in \cite{ciarlet2002finite} for p-Laplacian problem, 
% \begin{equation}
% \|u-u_h\|_{1,p} \leqslant   C \inf _{v_{h} \in V_{h}}\left\|u-v_{h}\right\|_{1,p}
% \end{equation}
%
% Using those estimates, it is possible to derive that $\|u-u_h\|_{\Wphi}  \rightarrow 0$ and $\J(u_h) - \J(u) \rightarrow 0$ as $h\rightarrow 0 $. 

%=============================================
\subsection{Proof of Lemma \ref{lem:quasinorm} and \ref{lem:quasinorm-reg}}
\label{sec:app:existence-wn}
%=============================================

\subsubsection{Proof of Lemma \ref{lem:quasinorm}}
\begin{proof}	
We first show the existence and uniqueness of $\wn$ defined by \eqref{eqn:direction-wn}.	
Let $\phi(t)= \int_0^t \vphi''(|\nabla \un|+\tau)\tau d\tau$. It is straightforward to show that $\phi(t)$ is an N-function and strictly convex under the assumption that $\vphi''$ is non-decreasing. 
Since the Simonenko indices of $\vphi$ $p_{\vphi}^{-}>1$ and $p_{\vphi}^{+}<\infty$, it suffice to show that Simonenko indices of $\phi$, $p_{\phi}^{-}>1$ and $p_{\phi}^{+}<\infty$. If $|\nabla \un|=0$, we can use the relation $\vphi'(t)\sim \vphi''(t)t$. Therefore, we only need to verify the case $|\nabla \un|\neq 0$. It is obvious that $\vphi''(|\nabla \un|+t) >0$ since $\vphi(t)$ is strictly convex. By
$$
p_{\phi}(t):=\frac{t \phi^{\prime}(t)}{\phi(t)}= \frac{\vphi''(|\nabla \un|+t)t^2}{\int_0^t \vphi''(|\nabla \un|+\tau)\tau d\tau},
$$
it implies that $\lim_{t\rightarrow0} p_{\phi}(t) = 2 $. We have, $\phi(t)=\int_0^t \vphi''(|\nabla \un|+\tau)\tau d\tau\geqs \int_0^t \vphi''(\tau)\tau d\tau \geqs c \vphi(t)$, $\vphi''(|\nabla \un|+t)t^2 \leqs \vphi(|\nabla \un|+t)$, and $\vphi(|\nabla \un|+t) \leqs 2^{p^{+}}\vphi(t)$, as $t\geq|\nabla \un|$. Therefore, $p_{\phi}^+(t)\leqs 2^{p^{+}}$, as $t\geq|\nabla \un|$. Since $p_{\phi}(t)$ is continuous in $t$,  $p^{+}_{\phi}(t) < \infty$ for $0 < t \leqs \infty $. Also, the non-decreasing property of $\vphi''$ implies that $p^{-}_{\phi}(t) >1$. 

By Theorem \ref{thm:orlicz}, we deduce that $\Delta_{2}({\phi,\phi^*})\leqs \infty$, $W^{1,\phi}$ is a separable and reflexive Banach space. Moreover, since $\phi(t)\geqs \vphi(t)$, we have $\|w\|_{W_{1,\phi}}\geqs \|w\|_{W_{1,\vphi}}$, for any $w\in W_{1,\phi}$. By strict convexity of $F(w):=\int_{\Omega} \phi(|\nabla(w)|)+ \J^{\prime}(\un)(w)$ for $w\in W^{1,\phi}$, there exists a unique minimizer $\wn\in W^{1,\phi}$ satisfying the equation  \eqref{eqn:direction-wn}, which further implies $\wn \in W^{1,\vphi}$.

A combination of \eqref{eqn:direction-wn}, Assumption \ref{asm:vphi}, Lemma \ref{lem:bdist}, Young type inequality Lemma \ref{lem:young},  Lemma \ref{lem:ab}, and inequality \eqref{eqn:young} lead to the following inequalities,
\begin{equation}
    \begin{aligned}
        \J\left(\un\right)-\J(u) & \leqs \J^{\prime}(\un)\left(\un-u\right) \\
           & = -C_q \int_{\Omega}\kappa(x)\vphi''\left(|\nabla \un|+|\nabla \wn|\right)\nabla \wn\cdot \nabla (\un-u)  \\
                & \lesssim C_q \int_{\Omega}\kappa(x)\vphi'_{|\nabla \un|}\left(|\nabla \wn|\right) \left|\nabla \un-\nabla u\right|  \\ 
                & \leqs C_q \int_{\Omega} C_\delta \kappa(x)\vphi_{|\nabla \un|}(|\nabla \wn|) + \delta \kappa(x)\vphi_{|\nabla \un|}(|\nabla (\un - u)|)\\
                & \leqs C_q \int_{\Omega} C_\delta \kappa(x) \vphi_{|\nabla \un|}(|\nabla \wn|) + CC_q \delta \kappa(x)\vphi_{|\nabla u|}(|\nabla (\un - u)|).
    \end{aligned}
    \label{ineq:qusi-upperbound}
\end{equation}
    
By the fact $\J'(u) = 0$, Lemma \ref{lem:bdist}, and Lemma \ref{lem:nfunction}, we have
\begin{equation}
    \J(\un) - \J(u) \geqs \cj \int_\Omega \kappa(x)\vphi_{|\nabla u|}(|\nabla (\un - u)|).
    \label{ineq:qusi-lowerbound}
\end{equation}

Plugging \eqref{ineq:qusi-lowerbound} to the right hand side of \eqref{ineq:qusi-upperbound} and taking an appropriate $\delta>0$,  we derive the following inequality 
\begin{equation}
    \J\left(\un\right)-\J(u)\leqs C_1\int_\Omega \kappa(x) \vphi_{|\nabla \un |}(|\nabla \wn|),
    \label{eqn:en}
\end{equation}
which leads to the first inequality \eqref{eqn:upperbound-qn} by applying Lemma \ref{lem:nfunction} again, and the constant $C_1$ only depends on $\Dphi$. 

The second inequality \eqref{eqn:lowerbound-qn} comes directly from a combination of \eqref{eqn:upaw} and \eqref{eqn:direction-wn}.
\end{proof}

\subsubsection{Proof of Lemma \ref{lem:quasinorm-reg}}

\begin{proof}

By Lemma \ref{lem:bdist}, \eqref{eqn:direction-wn-reg} and     \eqref{ineq:upperBoundQuasinorm}, we have 
\begin{displaymath}
\begin{aligned}
    \J(\un+\wn)-\J(\un) &\leqs  \J^{\prime}(\un)(\wn)+ \Cj \int_{\Omega}\kappa(x)\vphi''\left(|\nabla \un|+|\nabla \wn|\right)|\nabla \wn|^2\\
    &= - C_n A[\un](\wn, \wn) +  \Cj \int_{\Omega}\kappa(x)\vphi''\left(|\nabla \un|+|\nabla \wn|\right)|\nabla \wn|^2\\
    &\leqs - \int_{\Omega}\kappa(x)\vphi''\left(|\nabla \un|+|\nabla \wn|\right)|\nabla \wn|^2.
\end{aligned}
\label{eqn:upaw2}
\end{displaymath}	

Similar to the above proof of Lemma \ref{lem:quasinorm}, we have
\begin{equation}
    \begin{aligned}
        \J\left(\un\right)-\J(u) & \leqs \J^{\prime}(\un)\left(\un-u\right) \\
        & =-C_n A[\un](\wn, \un-u) \\
        & \leqs C_n\int_{\Omega}\kappa(x)\vphi''\left(|\nabla \un|+|\nabla \wn|\right)|\nabla \wn|\cdot |\nabla (\un-u)|  \\
        & \lesssim C_n\int_{\Omega}\kappa(x)\vphi'_{|\nabla \un|}\left(|\nabla \wn|\right)  |\nabla (\un-u)| \\ 
        & \leqs \int_{\Omega} C_n C_\delta \kappa(x)\vphi_{|\nabla \un|}(|\nabla \wn|) + \int_{\Omega}C_n \delta \kappa(x) \vphi_{|\nabla \un|}(|\nabla (\un - u)|)\\
        & \lesssim  C_n C_\delta  \int_{\Omega}\kappa(x)\vphi_{|\nabla \un|}(|\nabla \wn|) + C_n \delta\int_{\Omega} \kappa(x)\vphi_{|\nabla u|}(|\nabla (\un - u)|).
    \end{aligned}
    \label{ineq:linear-upperbound}
\end{equation}

Plugging \eqref{ineq:qusi-lowerbound} to the right hand side of \eqref{ineq:linear-upperbound}  and taking an appropriate $\delta>0$, we derive the following inequality 
\begin{equation}
\J\left(\un\right)-\J(u)\leqs C_2\int_\Omega \kappa(x) \vphi_{|\nabla \un |}(|\nabla \wn|),
\end{equation}
where $C_2$ only depending on $M_C$ and $\Delta_2(\vphi,\vphi^*)$.
It leads to the first inequality \eqref{eqn:upperbound-qn-reg} by applying  Lemma \ref{lem:nfunction} again. And, the second inequality \eqref{eqn:upperbound-nstep-qn-reg} can be proved similarly as in the proof of Lemma \ref{lem:quasinorm} by \eqref{eqn:upaw},\eqref{eqn:direction-wn-reg} and Assumption \ref{asm:cn} .
\end{proof}

\end{document}